\documentclass[aos,preprint]{imsart}

\RequirePackage{amsthm,amsmath,amsfonts,amssymb}
\RequirePackage[authoryear]{natbib} %% uncomment this for author-year bibliography
\RequirePackage[colorlinks,citecolor=blue,urlcolor=blue]{hyperref}
\RequirePackage{graphicx}

\usepackage{enumitem}
\usepackage{xcolor}
\usepackage[normalem]{ulem}

\startlocaldefs
\newtheorem{theorem}{Theorem}[section]
\newtheorem{lemma}[theorem]{Lemma}
\newtheorem{corollary}{Corollary}[theorem]
\newtheorem{proposition}[theorem]{Proposition}
\theoremstyle{remark}

\newtheorem{remark}{Remark}
\DeclareMathOperator*{\argmax}{arg\,max}
\newcommand{\btheta}{\boldsymbol{\theta}}
\newcommand{\asconv}{\stackrel{a.s.}{\rightarrow}}
\newcommand{\iidY}{\stackrel{iid}{\sim}}
\newcommand{\sumN}{\sum_{i=1}^n}
\usepackage{xparse}
\NewDocumentCommand{\evalat}{sO{\big}mm}{%
  \IfBooleanTF{#1}
   {\mleft. #3 \mright|_{#4}}
   {#3#2|_{#4}}%
}

\definecolor{Mulberry}{rgb}{0.77, 0.29, 0.55}
\definecolor{blue}{rgb}{0.0, 0.0, 1.0}

\newcommand\numberthis{\addtocounter{equation}{1}\tag{\theequation}}
\endlocaldefs

\begin{document}

\begin{frontmatter}
\title{Uniqueness and global optimality of the maximum likelihood estimator for the generalized extreme value distribution\thanksref{t1}}
%\title{A sample article title with some additional note\thanksref{t1}}
\runtitle{Unique and global optimality of the GEV MLE}
\thankstext{T1}{Supported in part by NSF grant DMS-2001433.}

\begin{aug}
%%%%%%%%%%%%%%%%%%%%%%%%%%%%%%%%%%%%%%%%%%%%%%
%%Only one address is permitted per author. %%
%%Only division, organization and e-mail is %%
%%included in the address.                  %%
%%Additional information can be included in %%
%%the Acknowledgments section if necessary. %%
%%%%%%%%%%%%%%%%%%%%%%%%%%%%%%%%%%%%%%%%%%%%%%
\author[A]{\fnms{Likun} \snm{Zhang}\ead[label=e1,mark]{likunz@psu.edu}},
\author[B]{\fnms{Benjamin} \snm{Shaby}\ead[label=e2]{bshaby@colostate.edu}}
% \and
% \author[B]{\fnms{Third} \snm{Author}\ead[label=e3,mark]{third@somewhere.com}}
%%%%%%%%%%%%%%%%%%%%%%%%%%%%%%%%%%%%%%%%%%%%%%
%% Addresses                                %%
%%%%%%%%%%%%%%%%%%%%%%%%%%%%%%%%%%%%%%%%%%%%%%
\address[A]{Department of Statistics,
Pennsylvania State University,
\printead{e1}}

\address[B]{Department of Statistics,
Colorado State University,
\printead{e2}
% \printead{e2,e3}
}
\end{aug}

\begin{abstract}
The three-parameter generalized extreme value distribution arises from classical univariate extreme value theory and is in common use for analyzing the far tail of observed phenomena.  Curiously, important asymptotic properties of likelihood-based estimation under this standard model have yet to be established. In this paper, we formally prove that the maximum likelihood estimator is global and unique. An interesting secondary result entails the uniform consistency of a class of limit relations in a tight neighborhood of the shape parameter. 
\end{abstract}

%\begin{keyword}[class=MSC2010]
%\kwd[Primary ]{62E20}
%\kwd[; secondary ]{60G70}
%\end{keyword}


\end{frontmatter}
%%%%%%%%%%%%%%%%%%%%%%%%%%%%%%%%%%%%%%%%%%%%%%
%% Please use \tableofcontents for articles %%
%% with 50 pages and more                   %%
%%%%%%%%%%%%%%%%%%%%%%%%%%%%%%%%%%%%%%%%%%%%%%
%\tableofcontents

%%%%%%%%%%%%%%%%%%%%%%%%%%%%%%%%%%%%%%%%%%%%%%%%%%
%%                   Section 1                  %%
%%%%%%%%%%%%%%%%%%%%%%%%%%%%%%%%%%%%%%%%%%%%%%%%%%
\section{Introduction}\label{sec:intro}

Classical extreme value theory was introduced almost a century ago \citep{fisher1928limiting} and is in wide practical use, yet a basic theoretical elucidation of likelihood-based inference under its central distributional construct remains incomplete.  Here, we fill in some of the important gaps. The generalized extreme value (GEV) distribution arises as the only limit of suitably renormalized maxima over independent and identically distributed (iid) random variables, and has therefore routinely been used in modeling the tail behavior of observed phenomena. However, since the support of the distribution depends on its parameters, standard regularity conditions of classic asymptotic theory are not satisfied. It is only recently that consistency and asymptotic normality of the maximum likelihood estimator (MLE), found locally on a restricted compact set, have been established. In this paper, we show that the local MLE uniquely and globally maximizes the GEV log-likelihood function.  In addition, we establish a number of convergence properties related to the GEV, including uniform consistency of a class of limit relations, revealing a much richer understanding of the GEV likelihood than has previously appeared.

The family of GEV distributions forms a continuous parametric family with respect to $\btheta=(\tau,\mu,\xi)$ on some measurable space $(\mathcal{X},\mathcal{A})$,
\begin{align*}
    P_{\btheta}(y)&=\exp\left\{-\left[1+\xi\left(\frac{y-\mu}{\tau}\right)\right]^{-1/\xi}\right\}\qquad \xi\neq 0\\
    &=\exp\left\{-\exp\left[-\frac{y-\mu}{\tau}\right]\right\}\qquad \xi=0,
\end{align*}
where $1+\xi\left(\frac{y-\mu}{\tau}\right)>0$ for $\xi\neq 0$,  and the scale parameter $\tau>0$, location parameter $\mu \in\mathbb{R}$, shape parameter $\xi\in\mathbb{R}$. The GEV distribution unites the Gumbel, Fr\'{e}chet and Weibull distributions into a single family to allow various shapes.  

The estimation of GEV parameters, especially the shape parameter $\xi$, is pivotal in studying tail behavior. The Picklands estimators \citep{pickands1975statistical}, probability weighted moments estimators \citep{hosking1985estimation}, and method of moments quantile estimators \citep{madsen1997comparison} are among many estimators available in the literature. \citet{beirlant2006statistics} provides a detailed review of the aforementioned estimators. In this paper, we focus on the asymptotic properties of estimators obtained from the maximum likelihood method. Denote $p_{\btheta}$ as the density function of $P_{\btheta}$ with respect to some dominating measure $\mathcal{P}$. Since the support of the GEV density function is not independent of $\btheta$, the regularity conditions for standard likelihood inference do not hold, which gives rise to fundamental difficulties when studying the existence, consistency and asymptotic normality of the MLE.

Suppose $\btheta_0=(\tau_0,\mu_0,\xi_0)$ is the true parameter, and $Y_1,\ldots, Y_n$ are independent samples from $P_{\btheta_0}$. \citet{smith1985maximum} was the first to consider the MLE of a large class of irregular parametric families, whose formulation includes the GEV distribution when $-1<\xi_0<0$. Treating the samples as coming from a distribution in the domain of attraction of a GEV, \citet{dombry2015existence} derived the existence of \textit{local} MLE, which is implicitly defined as solutions of the score functions, under the setting of triangular arrays of block maxima when $\xi_0>-1$. He proved that for any fixed compact set $K\subset\{\btheta:\tau>0,\mu\in \mathbb{R},\xi>-1\}$ that contains $\btheta_0$, the maximum of the likelihood function in $K$ is confined in an arbitrarily smaller neighborhood of $\btheta_0$, $\tilde{K}$, for all $n$ large enough. The corresponding MLE $\hat{\btheta}_n=(\hat{\xi}_n,\hat{\mu}_n,\hat{\tau}_n)$ solves the score functions, and it converges almost surely to $\btheta_0$. 

\citet{bucher2017maximum} extended this result, in the simpler setting where $Y_1,\ldots, Y_n$ are assumed to be independent samples from univariate GEV distribution, establishing a $O_p(1/\sqrt{n})$ rate of convergence for the local MLE, and refining the incomplete proof of \citet{smith1985maximum} to establish the asymptotic normality of $\hat{\btheta}_n$ for $\xi_0>-1/2$. 

However, it is not guaranteed that the local MLE $\hat{\btheta}_n$ studied by \citet{dombry2015existence} and \citet{bucher2017maximum} attains a unique, global maximum of the log-likelihood
\begin{equation*}
    L_n(\btheta)=\sumN l_{\btheta}(Y_i),
\end{equation*}
in which $l_{\btheta}: \btheta\mapsto \log p_{\btheta}(y)$, and $\btheta\in \Omega_n=\{\btheta:p_{\btheta}(Y_i)>0, i=1,\ldots,n\}$. Among other things, the uniform and global properties of $L_n$ in $\Omega_n$ are needed in Bayesian theory to develop optimal decision rules and perform posterior-based inference \citep{hartigan2012bayes}, to establish asymptotic posterior normality \citep{von1931wahrscheinlichkeit,chen1985asymptotic}, and to construct rule-based noninformative priors \citep{bernardo2005reference}. If $L_n$ is highly peaked and concentrates in a small neighborhood of $\btheta_0$, the information contained in the observations $Y_1,\ldots, Y_n$ will dominate any prior knowledge as $n$ approaches infinity, and hence posterior distribution will behave like a normal kernel.

In this paper, we only consider $\btheta_0\in \Theta= (0,\infty)\times \mathbb{R}\times (-1/2,\infty)$. We will prove that the local MLE gives a unique, global maximum point for the log-likelihood function by following a two-step strategy: 
\begin{enumerate}[label=\textbf{(\Roman*)},ref=(\Roman*)]
    \item\label{step2} %\LZedit{A small compact set $\tilde{K}$ containing $\btheta_0$ in its interior is constructed using $\btheta_0$.  We prove that for all large $n$, $L_n$ in $\tilde{K}$ is strictly concave and attains a unique maximum point;} 
    We first construct a small compact set $\tilde{K}$ containing $\btheta_0$ in its interior, and prove that for all large $n$, $L_n$ in $\tilde{K}$ is strictly concave and attains a unique maximum point;
    \item\label{step1}  We then define a larger compact set $K$, explicitly defined in terms of $\btheta_0$, such that $\tilde{K}\subsetneq K$. We prove for all large $n$, global maximum must be attained in $K$; that is, $\argmax_{\btheta\in \Theta}L_n(\btheta)\in K$.
    %\LZedit{A larger compact set $K$  is also explicitly defined by $\btheta_0$ such that $\tilde{K}\subsetneq K$. We prove for all large $n$, global maximum must be attained in $K$; that is, $\argmax_{\btheta\in \Theta}L_n(\btheta)\in K$.}
\end{enumerate}
By Proposition 2 in \citet{dombry2015existence}, $\hat{\btheta}_n=\argmax_{\btheta\in K}L_n(\btheta)\in \tilde{K}$ for all large $n$. We will therefore conclude that $L_n(\hat{\btheta}_n)$ is indeed the unique and global maximum $L_n$---the global optimality is ensured by \ref{step1}, while the uniqueness is ensured by \ref{step2}. This main result is stated in the following theorem.

\begin{theorem}[Global optimality and uniqueness]\label{thm:global_mle}
Suppose $Y_1, Y_2, \ldots\iidY P_{\btheta_0}$, and $\hat{\btheta}_n$ is the sequence of local maxima of $L_n$ that is found on a fixed compact neighborhood of $\btheta_0$. Define $\Theta_n=\{\btheta\in \Theta: \xi<n-1\}$. Then 
$$\mathcal{P}\left(\exists N>0\text{ such that for all } n>N,\;\argmax_{\btheta\in \Theta_n}L_n(\btheta)\text{ is unique and equals $\hat{\btheta}_n$}\right)=1.$$
\end{theorem}
\begin{remark}
Later we will show that for any $n$ and a sequence of numbers $Y_1,\ldots,Y_n$ that are not equal, $\max_{\btheta\in\{\xi\geq n-1\}}L_n(\btheta)=\infty$. One may object that the optimality result is not truly global because of the restriction $\xi<n-1$. Considering the shape parameters are less than $1$ for most observed data-generating processes, the ever-expanding $\Theta_n$ is hardly a restriction, and does not interfere with the derivation of asymptotic posterior properties. 
\end{remark}
\begin{remark}
Our asymptotic results are probabilistic in nature---that is, we do not treat the observed $Y_1,Y_2,\ldots$ as a deterministic sequence \citep[as opposed to][e.g.]{chen1985asymptotic}. Under the dominating measure $\mathcal{P}$, the desired properties hold almost surely.
\end{remark}

To prove  \ref{step2} and \ref{step1}, we first examine the support $\Omega_n$ of the log-likelihood $L_n$ in Section \ref{sec:supp}. In Section \ref{sec:profile_lik}, we study the properties of $L_n$ for a fixed $n$ and $Y_1,\ldots, Y_n$. To circumvent the non-convexity of the support $\Omega_n$, we slice $\Omega_n$ at different levels of $\xi$, and work with the maximum profile likelihood $PL_n(\xi)$ \citep{murphy2000profile}, which is defined as the maximum value of $L_n$ on the slice of $\xi$. With the help of the classic Seitz inequalities \citep{seitz1936remarque}, we show that $L_n$ is uniquely maximized on each slice of $\xi$. To locate the global maximum of $L_n$, we only need to compare the profile likelihood $PL_n(\xi)$ across different $\xi$ values. In Section \ref{sec:converge_rate}, we note that the boundary of $\Omega_n$ becomes infinitely close to $\btheta_0$ as $n$ approaches infinity. Although the convergence rate of the boundary to $\btheta_0$ is slower than the $1/\sqrt{n}$ rate of local MLE, the close proximity of the two poses fundamental challenges on deriving asymptotic properties. In Section \ref{sec:global_mle}, we overcome these challenges, and obtain the local concavity condition \ref{step2} in section \ref{sec:(2)} via establishing pointwise and uniform consistencies for a class of limit relations. We then prove the result \ref{step1} in Section \ref{sec:(1)}. Finally, we conclude that $\hat{\btheta}_n$ is indeed the unique global maximum point for $L_n$.

%%%%%%%%%%%%%%%%%%%%%%%%%%%%%%%%%%%%%%%%%%%%%%%%%%
%%                   Section 2                  %%
%%%%%%%%%%%%%%%%%%%%%%%%%%%%%%%%%%%%%%%%%%%%%%%%%%
\section{Preliminaries}\label{sec:preliminary}
\subsection{The joint likelihood function and its support}\label{sec:supp}
First we define 
\begin{equation*}
    \beta=\beta(\btheta)=\mu-\frac{\tau}{\xi}.
\end{equation*}
This one-to-one mapping from $(\tau,\mu,\xi)$ to $(\tau,\beta,\xi)$ will be used in the subsequent analysis to simplify  notation. In addition, define
\begin{equation}\label{wi}
    w_i(\btheta):=1+\xi\left(\frac{Y_i-\mu}{\tau}\right)=\frac{\xi}{\tau}(Y_i-\beta),
\end{equation}
which helps alleviate the complexity of the log-likelihood function:
\begin{equation}\label{Log-lik}
        L_n (\btheta)=
                -n\log\tau-\frac{\xi+1}{\xi}\sumN \log w_i(\btheta)-\sumN w_i^{-1/\xi}(\btheta),\text{ when }\xi\neq 0.
\end{equation} 

The form of $\beta$ also helps to concisely delineate the support of the joint density function, i.e., the domain of the log-likelihood $L_n (\btheta)$, which can be written out as follows, given the observations $Y_1,\ldots,Y_n$: 
\begin{equation}\label{supp}
    \begin{split}
        \Omega_n&=\{\btheta\in \Theta:w_i(\btheta)>0, i=1,\ldots,n\}\\
        &=\{\btheta\in \Theta:\xi(Y_i-\beta)>0, i=1,\ldots,n\}.
    \end{split}
\end{equation}
It can be easily verified that $\Omega_n$ is not a convex set, which means Taylor expansion will not be helpful for studying $L_n (\btheta)$, which precludes the use of routine techniques such as the mean-value theorem, and makes it difficult to approximate the difference of the function on a certain intervals. Nonetheless, if we slice $\Omega_n$ at different levels of $\xi$, every cross section is convex; see Figure \ref{domain_fig} for illustration. On a cross section at a fixed $\xi$, the value of $\beta=\mu-\tau/\xi$ can be construed as the intercept of the line which has a slope of $1/\xi$ and passes through $(\tau,\mu)$. When $\xi>0$, the condition in \eqref{supp} requires this intercept $\beta<Y_{(1)}$, and when $\xi<0$, the intercept $\beta>Y_{(n)}$, where $Y_{(1)}$ and $Y_{(n)}$ are the minimum and maximum values of the observations, respectively. Therefore, for any $\btheta\in\Theta$, we can immediately tell whether $\btheta\in\Omega_n$ using only $Y_{(1)}$ and $Y_{(n)}$.
\begin{figure}
    \centering
    \includegraphics[width=0.325\textwidth]{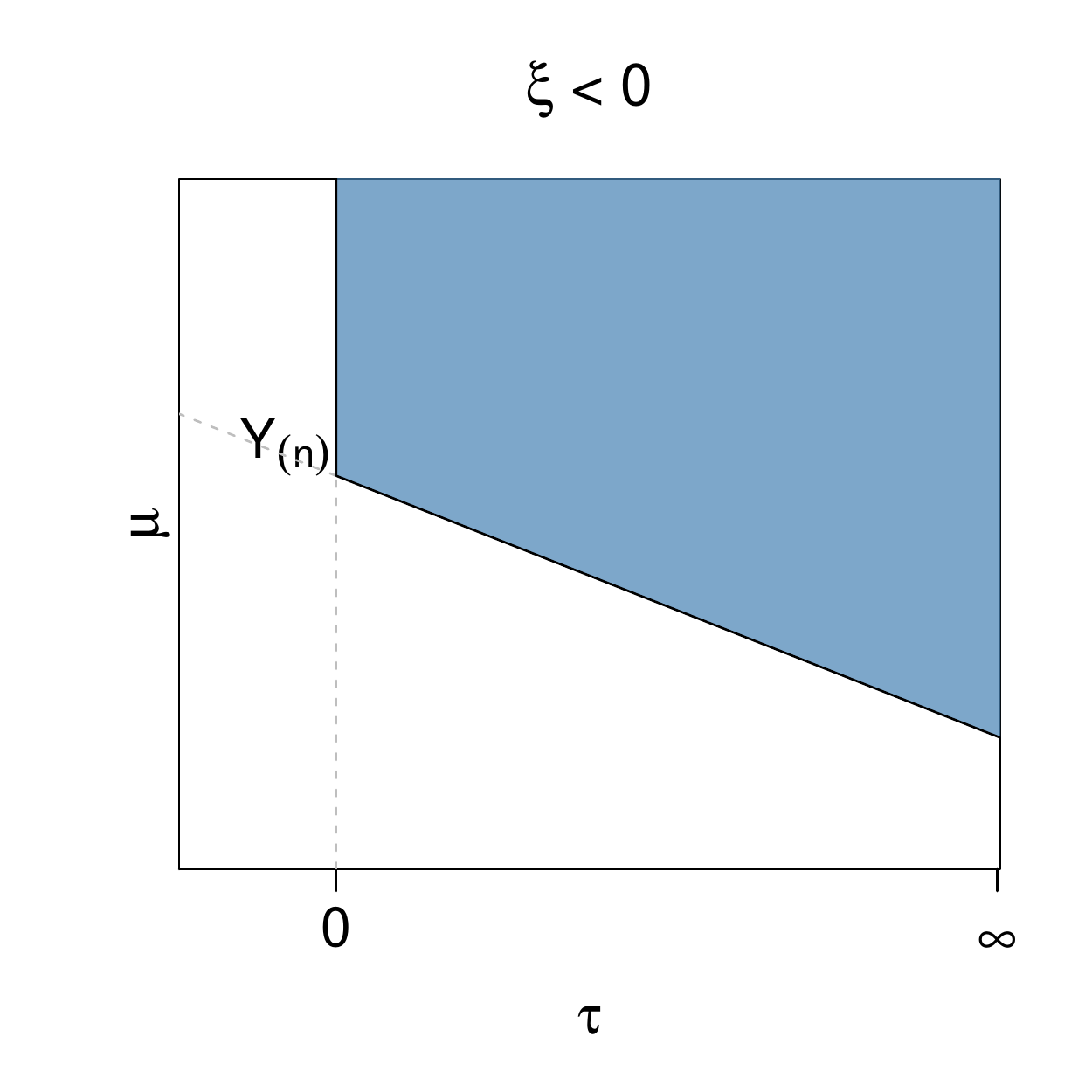}
    \includegraphics[width=0.325\textwidth]{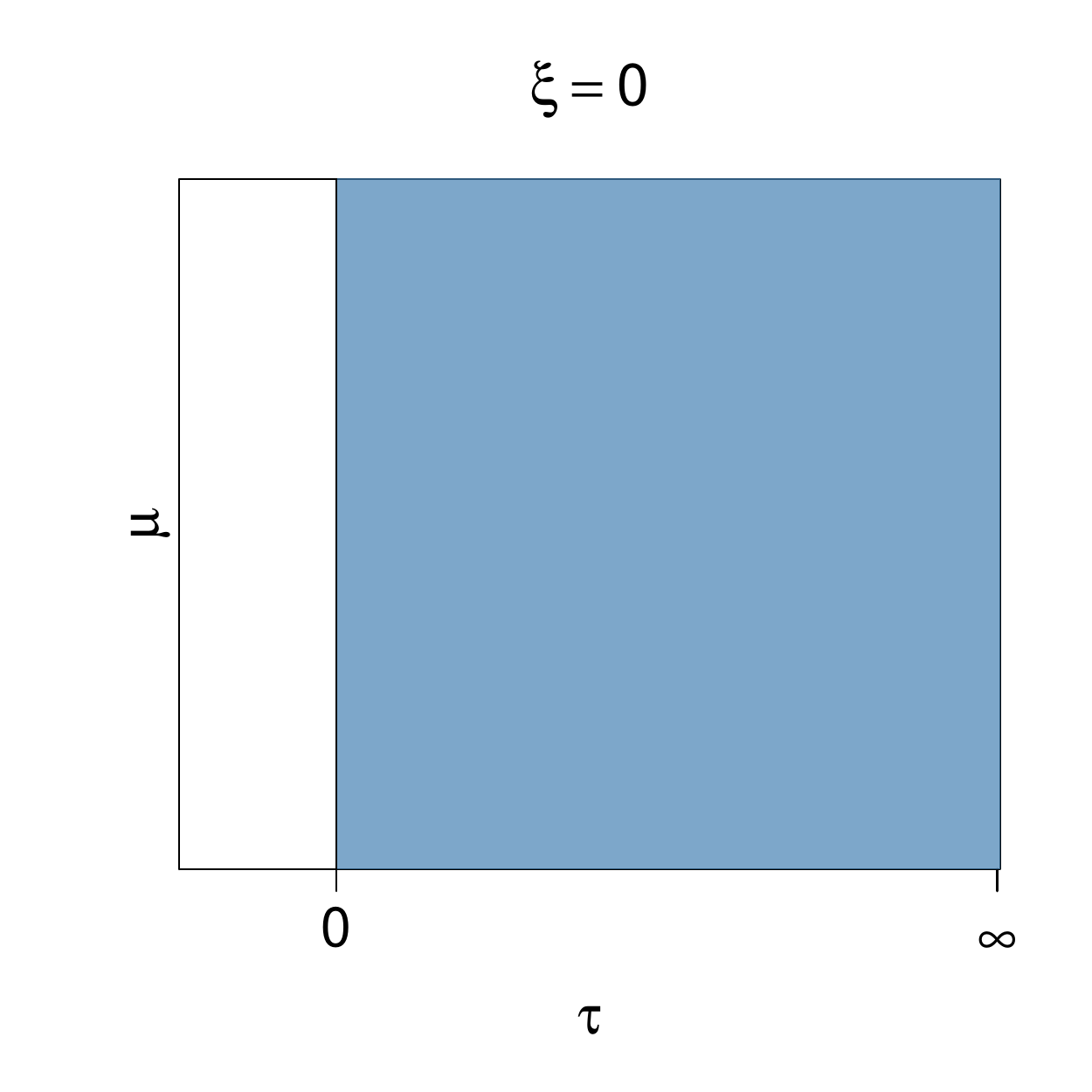}
    \includegraphics[width=0.325\textwidth]{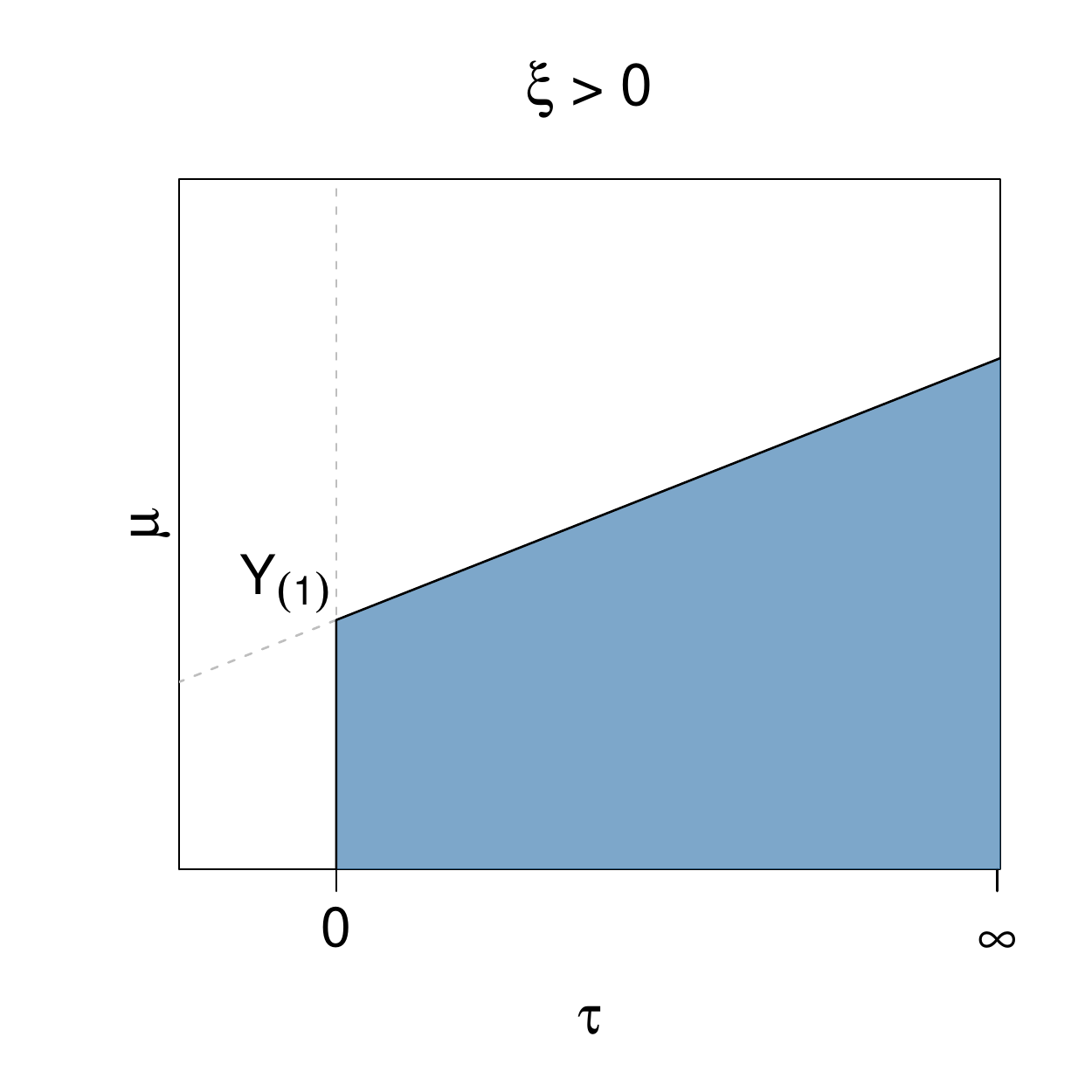}
    \caption{Slicing the support $\Omega_n$ at different levels of $\xi\in (-1/2,\infty)$. A cross section (shaded area) at any $\xi$ is convex with respect to $(\tau,\mu)$. When $\xi\neq 0$, the linear boundary of the cross section has a slope of $1/\xi$. If $\xi>0$, $\beta<Y_{(1)}$, and if $\xi<0$, $\beta>Y_{(n)}$.}
    \label{domain_fig}
\end{figure}

\subsection{Profile likelihood}\label{sec:profile_lik}
Denote the cross section of $\Omega_n$ at a certain $\xi$ as $\Omega_n(\xi)$. The convexity of $\Omega_n(\xi)$ suggests examining the log-likelihood via profiling out $(\tau,\mu)$:
\begin{equation}\label{profile_lik}
    PL_n(\xi):=\sup_{(\tau,\mu)\in \Omega_n(\xi)} L_n (\btheta).
\end{equation}
We show in the following proposition that on each cross section $\Omega_n(\xi)$, $L_n (\btheta)$ is uniquely maximized. 

\begin{proposition}\label{prop:profile_lik_uniqueML}
Suppose the sequence $Y_1,\ldots, Y_n%\iidY P_{\btheta_0}
$ are not all equal to each other. For any $\xi\neq 0$ and $-1<\xi<n-1$, there exists a maximizer $(\tau_n(\xi),\mu_n(\xi))$ of the log-likelihood $L_n$ that is unique and global on the cross section $\Omega_n(\xi)$. The values of $\tau_n(\xi)$ and $\mu_n(\xi)$ are implicitly defined by $\xi$ and the observations through
\begin{equation}\label{max_cross_section_cond}
    \begin{cases}
    &\tau=\left\{\frac{1}{n}\sumN [\xi(Y_i-\beta)]^{-1/\xi}\right\}^{-\xi},\\
    &(\xi+1)\sumN [\xi(Y_i-\beta)]^{-1}=\frac{n\sumN [\xi(Y_i-\beta)]^{-1-1/\xi}}{\sumN [\xi(Y_i-\beta)]^{-1/\xi}}.
    \end{cases}
\end{equation}
\end{proposition}
\begin{proof}
See Appendix \ref{proof:profile_lik}.
\end{proof} 
\begin{remark}\label{remark:PL_n}
As defined in \eqref{profile_lik}, $PL_n(\xi)=L_n(\tau_n(\xi),\mu_n(\xi),\xi)$. Inserting  \eqref{max_cross_section_cond} into \eqref{Log-lik},
\begin{equation}\label{eqn:PL_xi}
     PL_n(\xi)=-n\log\left\{\frac{1}{n}\sumN [\xi(Y_i-\beta_n(\xi))]^{-1/\xi}\right\}-\frac{\xi+1}{\xi}\sumN\log[\xi(Y_i-\beta_n(\xi))]-n,
\end{equation}
where $\beta_n(\xi)=\mu_n(\xi)-\tau_n(\xi)/\xi$, $-1<\xi<n-1$ and $\xi\neq 0$.
\end{remark}

\begin{remark}\label{remark:PL_n_0}
We show within the proof of Proposition \ref{prop:profile_lik_uniqueML} in Appendix \ref{proof:profile_lik} that $PL_n(\xi)=\infty, \text{ if } \xi>n-1$. For $\xi=0$, there also exists an unique maximizer $(\tau_n(0),\mu_n(0))$ on the cross section $\Omega_n(0)$, which is implicitly defined by
\begin{equation*}%\label{zero_cross_section_cond}
    \begin{cases}
    &n\tau=\sumN \left[1-\exp\left(-\frac{Y_i-\mu}{\tau}\right)\right]Y_i,\\
    &n=\sumN \exp\left(-\frac{Y_i-\mu}{\tau}\right).
    \end{cases}
\end{equation*}

By the continuity of $L_n$ at $\xi=0$, we know that
$$\lim_{\xi\rightarrow 0}\mu_n(\xi)=\mu_n(0),\;\lim_{\xi\rightarrow 0}\tau_n(\xi)=\tau_n(0)\text{ and }\lim_{\xi\rightarrow 0}PL_n(\xi)=PL_n(0).$$
\end{remark}

\begin{remark}
Notice that for this proposition, we are not assuming that $Y_1,\ldots,Y_n$ are drawn from $P_{\btheta_0}$. Rather they can be independent copies of any unitary distribution on $\mathcal{X}$. The fact that $PL_n$ becomes senseless when $\xi>n-1$ is a result of the mathematical form of the GEV log-likelihood $L_n$.
\end{remark}

Since this result does not rely on the asymptotics of the GEV distribution, the lower bound of $\xi$ is extended to $-1$. The reason we restrict the lower bound to $-1/2$ in Theorem \ref{thm:global_mle} is that the propositions in the forthcoming sections require second-order consistencies. The existence and first-order consistency arguments in \citet{dombry2015existence} also pertain to $\xi>-1$.

To find the global maximum, we now only need to compare the `representative' maximum likelihood from the each cross section. If the profile likelihood $PL_n$ as a function of $\xi$ is strictly concave in $(-1,n-1)$, it has a unique maximum at a $\xi$ value such that $PL'_n(\xi)=0$, and then $(\tau_n(\xi),\mu_n(\xi),\xi)$ is the unique global maximizer for $L_n$. Unfortunately, $PL_n$ is not a strictly concave function of $\xi$. As demonstrated in the following proposition, the first derivative $PL'_n$ is not monotonically decreasing, and it behaves irregularly when $\xi$ approaches the bounds of $(-1,n-1)$.

\begin{proposition}\label{prop:xi_derivative}
Under the assumptions of Proposition \ref{prop:profile_lik_uniqueML},the first derivative $PL'_n$ is well-defined and continuous in $\xi\in(-1,n-1)$. When $\xi\searrow -1$, $PL'_n(\xi)\rightarrow -\infty$. When $\xi\nearrow n-1$, $PL'_n(\xi)\rightarrow \infty$. By the intermediate zero theorem, there must exist a $\xi\in (-1,n-1)$ such that $PL'_n(\xi)=0$.
\end{proposition}
\begin{proof}
See Appendix \ref{proof:profile_lik}.
\end{proof} 
\begin{remark}
For $\xi\neq 0$,
\begin{small}
\begin{equation}\label{eqn:xi_derivative}
    PL'_n(\xi)=-\frac{n}{\xi}-\frac{n\sumN[\xi(Y_i-\beta_n(\xi))]^{-1/\xi}\log[\xi(Y_i-\beta_n(\xi))]}{\xi^2\sumN[\xi(Y_i-\beta_n(\xi))]^{-1/\xi}}+\frac{1}{\xi^2}\sumN \log[\xi(Y_i-\beta_n(\xi))].
\end{equation}
\end{small}
For $\xi=0$, the first derivative can be defined as the limit:
\begin{equation*}
    \lim_{\xi\rightarrow 0}PL'_n(\xi)=\frac{n\mu'_n(0)-\sumN [Y_i-\mu_n(0)+\tau'_n(0)]}{\tau_n(0)}+\frac{\sumN [Y_i-\mu_n(0)+\tau'_n(0)]^2-n\tau'_n(0)^2}{2\tau_n(0)^2}.
\end{equation*}
\end{remark}

It is easy to verify that if a $\xi$ value solves $PL'_n(\xi)=0$, \eqref{max_cross_section_cond} and \eqref{eqn:xi_derivative} together ensure that $(\tau_n(\xi),\mu_n(\xi),\xi)$ solves the score equations of $L_n$. Hence this result provides an alternative approach to proving the existence of the local MLE for $L_n$. However, acquiring the strong consistency of the local MLE requires the assumption that $Y_1,\ldots Y_n\iidY P_{\btheta_0}$ and the limiting behavior as $n\rightarrow\infty$. 

Figure \ref{fig:PL_n_xi} illustrates some key features of the profile likelihood function.  We simulate $Y_1,\ldots,Y_n$ from $P_{\btheta_0}$ and calculate the profile likelihood $PL_n$ at a grid of $\xi$ values ranging from $-1$ to $1$. For both positive shape parameter $\xi_0=0.2$ and negative shape parameter $\xi_0=-0.2$, $PL_n$ %\remove{is}
appears to be uniquely maximized by the local MLE which is close to $\xi_0$. Although it is not a concave function globally, we still observe local concavity around $\xi_0$, which suggests adoption of the two-step strategy introduced in Section \ref{sec:intro}. Roughly speaking, these two steps are established in Section \ref{sec:global_mle} via proving \ref{step2} $PL_n$ is strictly concave in a small neighborhood of $\hat{\xi}_n$, and \ref{step1} $PL_n(\xi)<PL_n(\hat{\xi}_n)$ for $\xi$ that is far from $\hat{\xi}_n$. 
\begin{figure}
    \centering
    \includegraphics[width=0.325\textwidth]{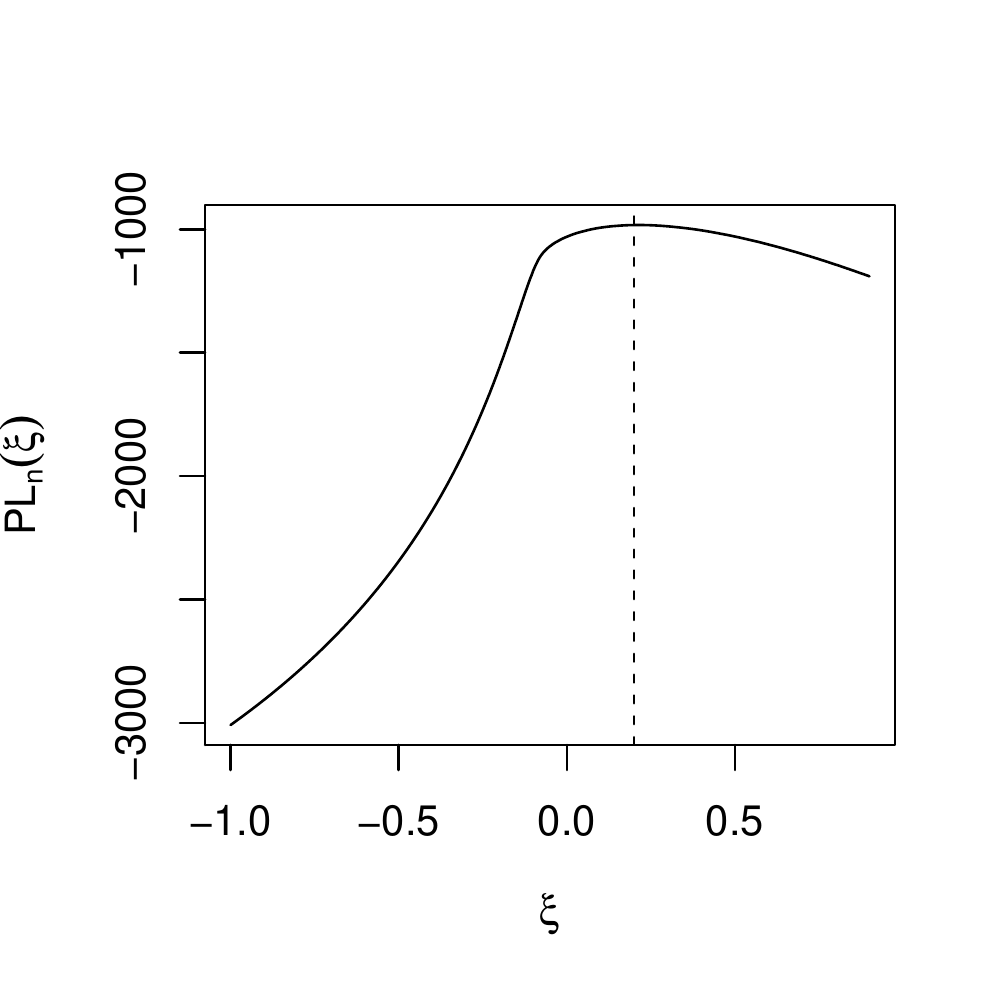}
    \includegraphics[width=0.325\textwidth]{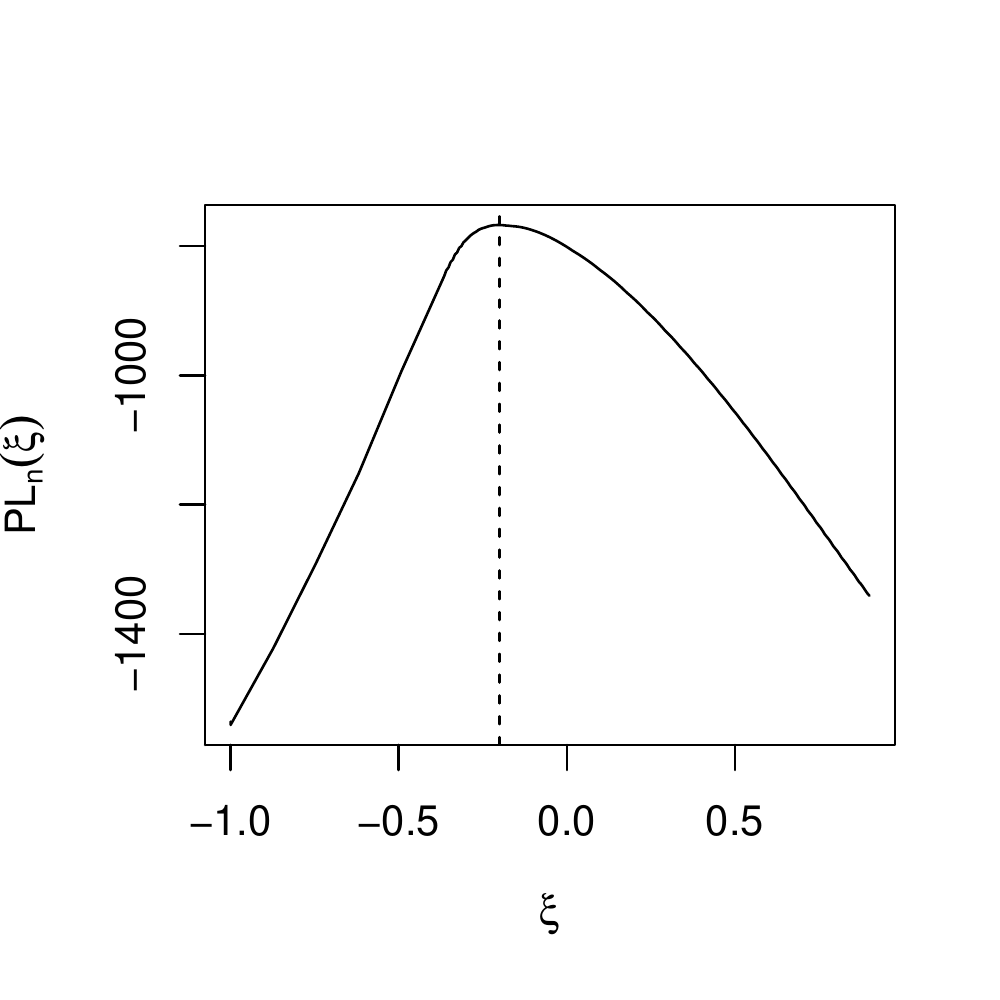}
    \caption{$PL_n(\xi)$ under $Y_i,\;i=1,\ldots,n$ that are sampled from true $\xi_0=0.2$ (left) and $\xi_0=-0.2$ (right), with dashed lines marking the local MLE $\hat{\xi}_n$. For both scenarios, $(\tau_0,\mu_0)=(0.5,20)$ and sample size $n=1,000$. We see that $PL_n$ is not a concave function.
    }
    \label{fig:PL_n_xi}
\end{figure}

%%%%%%%%%%%%%%%%%%%%%%%%%%%%%%%%%%%%%%%%%%%%%%%%%%
%%                   Section 3                  %%
%%%%%%%%%%%%%%%%%%%%%%%%%%%%%%%%%%%%%%%%%%%%%%%%%%

\section{Convergence rate of the support boundary}\label{sec:converge_rate}

To prove \ref{step2} and \ref{step1}, we will need to study the distance between the true parameter $\btheta_0$ and the boundary of the support $\Omega_n$. It is true from the definition of $\Omega_n$ that if $Y_1,\ldots, Y_n$ are drawn from $P_{\btheta_0}$, 
\begin{equation*}
    \btheta_0\in \Omega_n \text{ for any }n\geq 1.
\end{equation*}
It is clear that $\Omega_n$ is an open set for any $n$, and hence the true parameter $\btheta_0$ is always an interior point of $\Omega_n$. This raises the question: can we always find a neighborhood of $\btheta_0$ which is contained by $\Omega_n$ that is large enough to allow us to examine the log-likelihood in the vicinity of $\btheta_0$? Unfortunately, this is not possible because $\btheta_0$ becomes infinitely close to the boundary as $n$ approaches infinity.

To quantify the distance between $\btheta_0$ and the boundary of $\Omega_n$, we first assume $\xi_0>0$ and examine the cross section $\Omega_n(\xi_0)$. This is illustrated in Figure \ref{convergence_rate_fig}, where $\btheta_0=(\tau_0,\mu_0,\xi_0)$ is shown as a red point, and $\beta_0=\mu_0-\tau_0/\xi_0$ is the intercept of the line that passes through $(\tau_0,\xi_0)$ with a slope of $1/\xi_0$. Figure \ref{convergence_rate_fig} illustrates that the difference of intercepts, $Y_{(1)}-\beta_0$, is a good measure of the distance. By analogy, if true shape parameter $\xi_0<0$, the distance can be well-measured by $\beta_0-Y_{(n)}$. 

Note that when $\xi_0>0$, the distribution of $P_{\btheta_0}$ is lower bounded by $\beta_0$, which guarantees $Y_{(1)}\asconv\beta_0$. When $\xi_0<0$, the distribution of $P_{\btheta_0}$ is upper bounded by $\beta_0$, which guarantees $Y_{(n)}\asconv\beta_0$. Thus in both cases, the distance between $\btheta_0$ and the boundary of $\Omega_n$ converges almost surely to zero. Also, \citet{bucher2017maximum} showed that 
\begin{equation*}
    \hat{\btheta}_n-\btheta_0=O_p(1/\sqrt{n}),
\end{equation*}
which means $\hat{\btheta}_n$ is also infinitely close to $\btheta_0$ as $n$ grows, and thus close to the boundary of $\Omega_n$. This is concerning for the purpose of proving global optimity of $\hat{\btheta}_n$ because it would be rather challenging to handle the log-likelihood near the boundary of the support.
\begin{figure}
    \centering
    \includegraphics[width=0.35\textwidth]{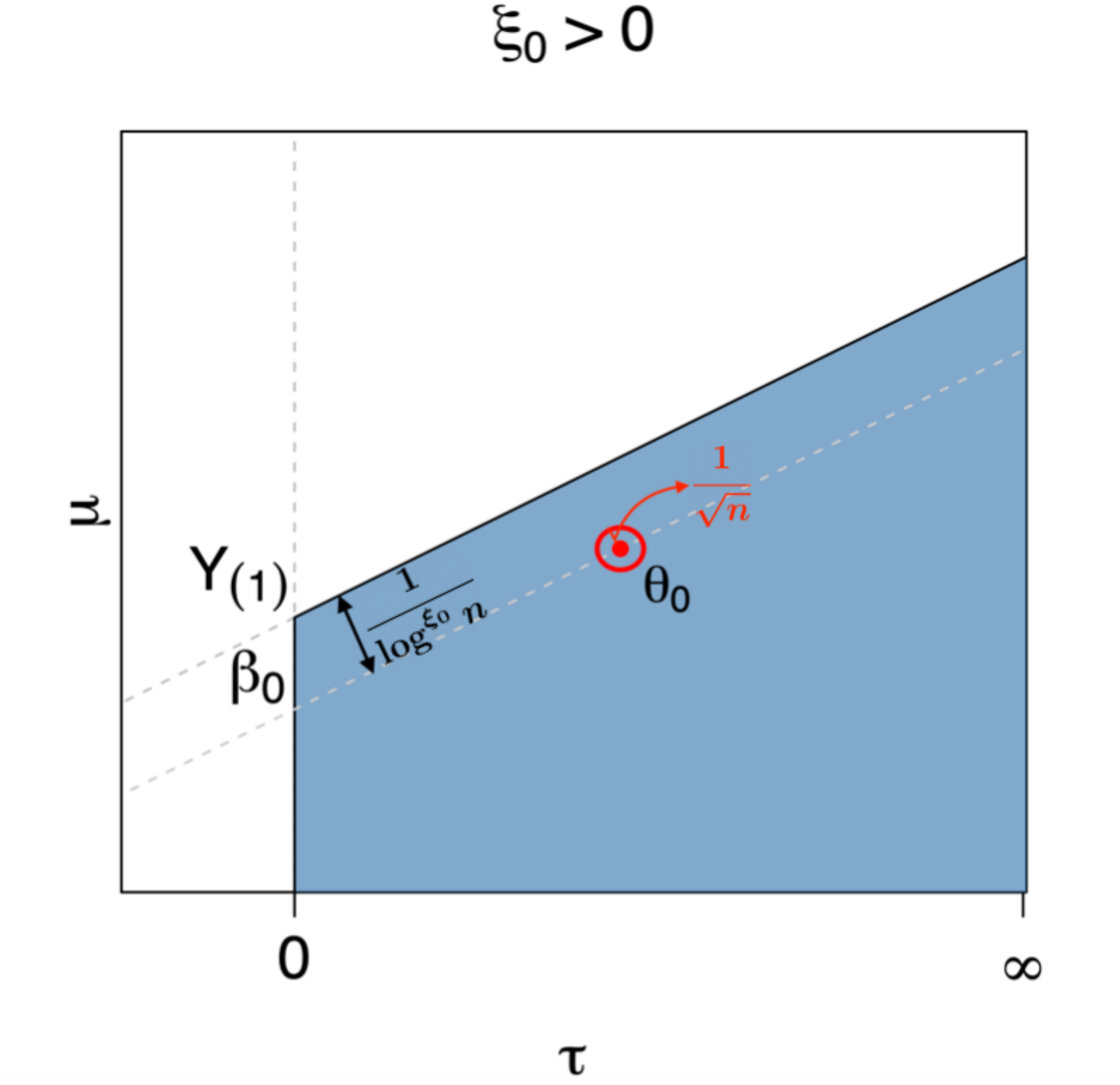}
    \caption{The cross section $\Omega_n(\xi_0)$ if true $\xi_0>0$. The two parallel dashed lines have a slope of $1/\xi_0$. The red point is  $\btheta_0=(\tau_0,\mu_0,\xi_0)$. Here we also compare the convergence rates of $\hat{\btheta}_n$ and $Y_{(1)}$, which are $1/\sqrt{n}$ and $1/\log^{\xi_0}n$. The red circle marks the neighborhood of $\btheta_0$ with radius $1/\sqrt{n}$.}
    \label{convergence_rate_fig}
\end{figure}

Therefore, it is imperative that we compare the convergence rate of the distance between $\btheta_0$ and the boundary with $1/\sqrt{n}$ in order to get a clearer picture of $L_n (\btheta)$ near the boundary.
\begin{proposition}\label{prop:rate_Min}
Suppose $Y_1,\ldots, Y_n\iidY P_{\btheta_0}$, and $\gamma>0$ is an arbitrary constant.
\begin{enumerate}[label=\textbf{(\Alph*)},ref=(\Alph*)]
\item\label{boundary_pos} If $\xi_0>0$, $Y_{(1)}\asconv\beta_0$ and $Y_{(n)}\asconv\infty$, as $n\rightarrow\infty$. Moreover,
\begin{equation*}
\begin{split}
    (\log n)^{(1+\gamma)\xi_0}(Y_{(1)}-\beta_0)\asconv\infty,&\text{  }(\log n)^{(1-\gamma)\xi_0}(Y_{(1)}-\beta_0)\asconv0,\\
    n^{-(1+\gamma)\xi_0}Y_{(n)}\asconv 0,&\text{ and  }n^{-(1-\gamma)\xi_0}Y_{(n)}\asconv\infty.
\end{split}
\end{equation*}

\item If $\xi_0<0$, $Y_{(1)}\asconv-\infty$ and $Y_{(n)}\asconv\beta_0$, as $n\rightarrow\infty$. Moreover,
\begin{equation*}
\begin{split}
    (\log n)^{(1+\gamma)\xi_0}Y_{(1)}\asconv 0,&(\log n)^{(1-\gamma)\xi_0}Y_{(1)}\asconv -\infty,\\
    n^{-(1+\gamma)\xi_0}(\beta_0-Y_{(n)})\asconv \infty,&\text{ and  }n^{-(1-\gamma)\xi_0}(\beta_0-Y_{(n)})\asconv 0.
\end{split}
\end{equation*}
\end{enumerate}
\end{proposition}
\begin{proof}
See Appendix \ref{proof:converge_rate}.
\end{proof}
When $\xi_0>0$, this theorem demonstrates that the convergence rate of $Y_{(1)}$ to $\beta_0$ is roughly $1/\log^{\xi_0}n$. The convergence rate of $\hat{\btheta}_n$ to $\btheta_0$, $1/\sqrt{n}$, is much faster than the rate of $Y_{(1)}$ to $\beta_0$. These two rates are compared schematically in Figure \ref{convergence_rate_fig}. If $\xi_0<0$, the convergence rate of $Y_{(n)}$ to $\beta_0$ is $n^{\xi_0}$, which is still slower than $1/\sqrt{n}$ because of the restriction $\xi_0>-1/2$. Thus for a ball neighborhood of $\hat{\btheta}_n$ to be contained in $\Omega_n$, its radius can be up to $1/n^\gamma$ for some $\gamma\in(0,1/2)$. This property will be of vital importance in the proof of \ref{step2} and \ref{step1}.

%%%%%%%%%%%%%%%%%%%%%%%%%%%%%%%%%%%%%%%%%%%%%%%%%%
%%                   Section 4                  %%
%%%%%%%%%%%%%%%%%%%%%%%%%%%%%%%%%%%%%%%%%%%%%%%%%%

\section{Proof of Theorem \ref{thm:global_mle}}\label{sec:global_mle}

\subsection{Step \ref{step2} and its proof}\label{sec:(2)}
Construct the following compact set
\begin{equation*}
    \tilde{K}=\{\btheta\in\Theta: |\tau-\tau_0|\leq r,|\beta-\beta_0|\leq r,|\xi-\xi_0|\leq r\},
\end{equation*}
where $r$ is a small constant to be determined by $\btheta_0$ such that the local concavity holds in $\tilde{K}$. Slicing $\tilde{K}$ at different levels of $\xi$ will produces parallelograms; see Figure \ref{K_tilde} for illustration. In this section, we will
prove that for all large $n$, the Hessian matrix of $L_n$ is negative definite in $\tilde{K}\cap\Omega_n$, and hence $L_n$  is strictly concave.
\begin{figure}
    \centering
     \includegraphics[width=0.33\textwidth]{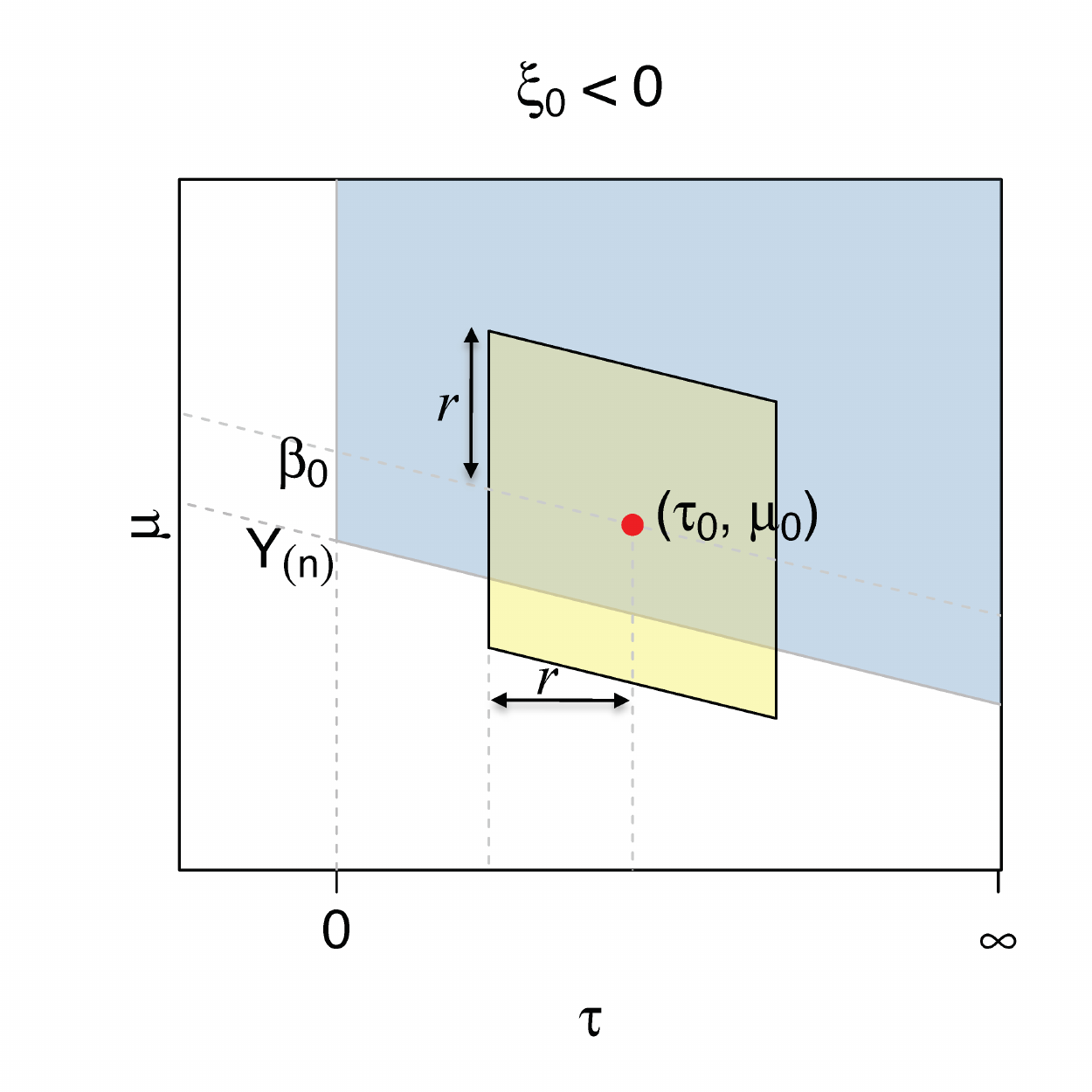}
     \includegraphics[width=0.33\textwidth]{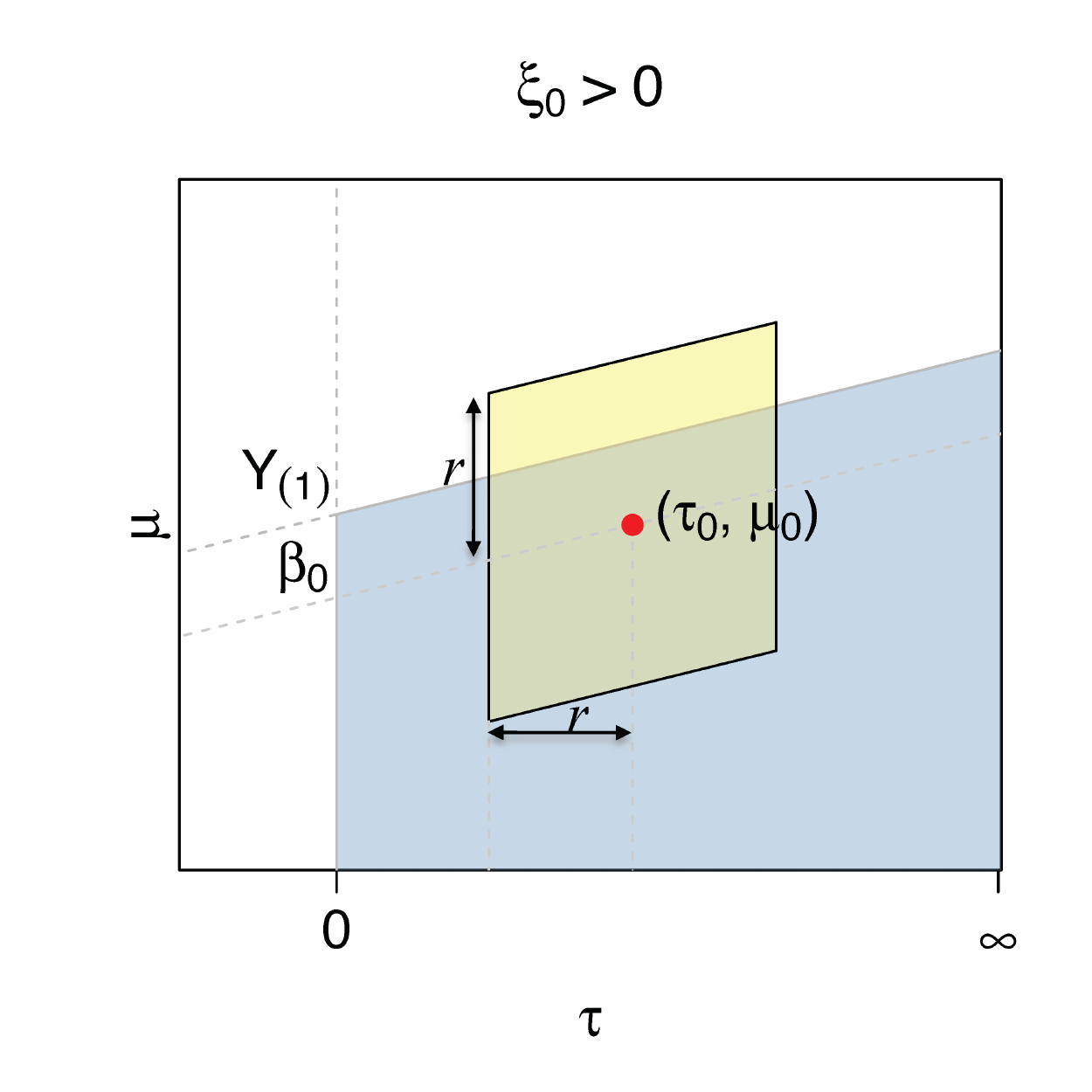}
    \caption{Illustrating $\tilde{K}$ for $\xi_0<0$ (left) and $\xi_0>0$ (right). In both cases, the set $\tilde{K}$ sliced at $\xi=\xi_0$ is shown in yellow, with $\Omega_n(\xi_0)$ shown in blue. Since $|\tau-\tau_0|<r$ and $|\beta-\beta_0|<r$, the slice is a parallelogram. Similarly, $\tilde{K}$ sliced at another $\xi$ in $(\xi_0-r,\xi_0+r)$ will also yield a parallelogram.}
    \label{K_tilde}
\end{figure}

\subsubsection{Smoothness of Hessian matrix}
First we study $L''_n(\hat{\btheta}_n)$, the Hessian at the local MLE $\hat{\btheta}_n$. The log-likelihood $L_n(\btheta)$ in \eqref{Log-lik} and elements of its Hessian matrix $L''_n(\btheta)$ can all be written as linear combinations of sums of the form
\begin{equation*}
    \sumN w_i^{-k-\frac{a}{\xi}}(\btheta)\log^b w_i(\btheta),
\end{equation*}
where $k,b=0,1,2$, $a=0,1$; see Appendix \ref{proof:loc_concavity} for the explicit expressions of the Hessian.

For $\xi_0\neq 0$ and constants $k$ and $a$ such that $k\xi_0+a+1>0$, it is straightforward to calculate
\begin{equation*}
    E_{\btheta_0}\left[w^{-k-\frac{a}{\xi_0}}(\btheta_0)\log^b w(\btheta_0)\right]=(-\xi_0)^b\Gamma^{(b)}(k\xi_0+a+1),
\end{equation*}
where $w(\btheta_0)=\xi_0(Y-\beta_0)/\tau_0$ with $Y\sim P_{\btheta_0}$, and $\Gamma^{(b)}$ is the $b$th-order derivative of the Gamma function. Since this is an iid sequence, the strong law of large numbers immediately gives strong consistency for a class of limit relations 
\begin{equation*}%\label{SLLN}
    \frac{1}{n}\sumN w_i^{-k-\frac{a}{\xi_0}}(\btheta_0)\log^b w_i(\btheta_0) \asconv  (-\xi_0)^b\Gamma^{(b)}(k\xi_0+a+1).
\end{equation*}

To examine $L''_n(\hat{\btheta}_n)$, $\btheta_0$ in the preceding averages needs to be changed to $\hat{\btheta}_n$. Since $\hat{\btheta}_n\asconv \btheta_0$, the continuity of the sums with respect to $\btheta$ permits a pseudo large law of numbers for the sums in $L''_n(\hat{\btheta}_n)$.
\begin{proposition}\label{prop:pseudo_SLLN}
Suppose $Y_1, Y_2,\ldots, \iidY P_{\btheta_0}$ and $\xi_0\neq 0$, and $\hat{\btheta}_n$ is the local MLE of $L_n(\btheta)$ that is strongly consistent. Then for constants $k$ and $a$ such that $k\xi_0+a+1>0$,
\begin{equation}\label{pseudo_SLLN}
    \frac{1}{n}\sumN w_i^{-k-\frac{a}{\hat{\xi}_n}}(\hat{\btheta}_n)\log^b w_i(\hat{\btheta}_n) \asconv  (-\xi_0)^b\Gamma^{(b)}(k\xi_0+a+1),
\end{equation}
where $b$ is a non-negative integer.
\end{proposition}
\begin{proof}
The proof of this result depends on Proposition \ref{prop:rate_Min}. For details see Appendix \ref{proof:pseudo_SLLN}.
\end{proof}

Proposition \ref{prop:pseudo_SLLN} ensures that $L''_n(\hat{\btheta}_n)$ behaves like $L''_n(\btheta_0)$ for large $n$. The following result states that if we select $r$ carefully, $L''_n(\btheta)$ can be approximated by $L''_n(\hat{\btheta}_n)$ in the neighborhood $\tilde{K}\cap\Omega_n$; hence the negative-definiteness of $L''_n(\btheta)$ in this neighborhood.

\begin{proposition}\label{prop:hessian}
Let $Y_1, Y_2, \ldots\iidY P_{\btheta_0}$, and $\hat{\btheta}_n$ is the local MLE of $L_n(\btheta)$ that is strongly consistent. For a small $r>0$ that is chosen according to the value of $\btheta_0$, there almost surely exists $N$ such that, for any $n>N$ and $\btheta\in \tilde{K}\cap\Omega_n$,
\begin{equation}\label{smoothness}
    \boldsymbol{I}-\boldsymbol{A}_0(r)\leq L''_n(\btheta)\{L''_n(\hat{\btheta}_n)\}^{-1}\leq \boldsymbol{I}+\boldsymbol{A}_0(r),
\end{equation}
where $\boldsymbol{I}$ is the $3\times 3$ identity matrix and $\boldsymbol{A}_0(r)$ is a $3\times 3$ symmetric positive-semidefinite matrix whose elements only depend on $\btheta_0$ and the radius $r$, and whose largest eigenvalue tends to zero as $r\rightarrow 0$.
\end{proposition}
\begin{proof}
The rule to choose $r$ is specified in Appendix \ref{proof:loc_concavity}, which requires $r$ to satisfy \eqref{range_xi_supp} and $r<\frac{\tau_0}{3}$. The proof of this proposition can also be found in Appendix \ref{proof:loc_concavity}.
\end{proof}

As a side result, we obtain the uniform consistency of limit relations in \eqref{pseudo_SLLN} as the powers of the $w_i$ terms change in a closed interval. In Proposition \ref{prop:pseudo_SLLN}, changing the power continuously produces a continuous path of the limit. If we fix the non-negative integer $b$ and regard $\Phi_n(\alpha)=\frac{1}{n}\sumN w_i^{-\alpha}(\hat{\btheta}_n)\log^b w_i(\hat{\btheta}_n)$ as a stochastic process, $\Phi_n(\alpha)$ converges pointwise almost surely to $\Phi(\alpha)=(-\xi_0)^b\Gamma^{(b)}(\alpha\xi_0+1)$. In the following, we will show that the rate of convergence of sequences of $\Phi_n(\alpha)$ is essentially the same within a closed interval of $\alpha$. That is, there is uniform consistency, which is a stronger property than stochastic equicontinuity. The uniformity will be crucial to proving step \ref{step1}.

\begin{proposition}[Uniform consistency]
\label{prop:uniform_consistency}
Suppose $Y_1, \ldots, Y_n\iidY P_{\btheta_0}$ where $\xi_0\neq0$, and $\hat{\btheta}_n$ is the local MLE of $L_n(\btheta)$ that is strongly consistent. Fix the non-negative integer $b$, and define the closed interval $I=[m,M]$ if $\xi_0>0$, and $I=[-M,-m]$ if $\xi_0<0$, where $M>0$ and $m\in(-1/|\xi_0|,0)$ are arbitrary constants. For $\alpha\in I$, denote $\Phi_n(\alpha)=\frac{1}{n}\sum_{i=1}^nw_i^{-\alpha}(\hat{\btheta}_n)\log^b w_i(\hat{\btheta}_n)$ and $\Phi(\alpha)=(-\xi_0)^b\Gamma^{(b)}(\alpha\xi_0+1)$. Then,
\begin{equation*}
    \sup_{\alpha\in I}\left|\Phi_n(\alpha)-\Phi(\alpha)\right|\asconv 0,
\end{equation*}
\end{proposition}
\begin{proof}
See Appendix \ref{proof:uniform_consistency}.
\end{proof}

\subsubsection{Local concavity}
\begin{proposition}[Step \textbf{\ref{step2}}]\label{prop:loc_concavity}
Let $Y_1,Y_2,\ldots \iidY P_{\btheta_0}$, and $\hat{\btheta}_n$ is the local MLE of $L_n(\btheta)$ that is strongly consistent. Then we can find a $r>0$ small enough such that $L_n(\btheta)$ is a strictly concave function in $\tilde{K}\cap\Omega_n$. Namely,
\begin{equation*}%\label{smoothness_hessian}
    \mathcal{P}\left(\exists N>0\text{ such that for all }n>N,\; L''_n(\btheta)<0 \text{ in }\tilde{K}\cap \Omega_n\right)=1.
\end{equation*}
Equivalently, $\hat{\btheta}_n$ is an unique maximum point in $\tilde{K}$.
\end{proposition}
\begin{proof}
Proposition \ref{prop:pseudo_SLLN} ensures that
\begin{equation*}
    \frac{1}{n}L_n''(\hat{\btheta}_n)\asconv  -I({\btheta_0}),
\end{equation*}
where $I({\btheta_0})$ is the Fisher information of $P_{\btheta_0}$, and we know $|I({\btheta_0})|>0$ for all $\xi_0>-1/2$. Therefore, $I({\btheta_0})$ is positive definite, and there almost surely exists $N>0$ such that for all $n>N$,
\begin{equation*}
    L_n''(\hat{\btheta}_n)<0.
\end{equation*}

By Proposition \ref{prop:hessian}, $A_0(r)$ only depends on $\btheta_0$ and $r$. We now fix $r$ small enough such that the smallest eigenvalue of $\boldsymbol{I}-\boldsymbol{A}_0(r)$ is positive. By \eqref{smoothness}, 
\begin{equation*}
    L''_n(\btheta)\leq L''_n(\hat{\btheta}_n)\left(\boldsymbol{I}-\boldsymbol{A}_0(r)\right)<0.
\end{equation*}
Note that the choice of $r$ only depends on $\btheta_0$. 
\end{proof}

\subsection{Step \ref{step1} and its proof}\label{sec:(1)}
% \BASadd{Should we delete the first sentence to avoid misunderstanding about whether the result relies on looking at the plot (similarly to the reference prior paper)?}
% \remove{As we observed in Figure \ref{fig:PL_n_xi} and many other numerical simulations, the log-likelihood $L_n$ decays rapidly when $\btheta$ moves away from the local MLE $\hat{\btheta}_n$.} 
Step \ref{step1} confines the global MLE to a fixed compact set $K$ which is constructed using the values of $\btheta_0$. Since $\hat{\btheta}_n=\argmax_{\btheta\in K} L_n(\btheta)$, we can deduce the global optimality of $\hat{\btheta}_n$. When $\xi_0>0$, $K$ is defined as 
\begin{equation}\label{K_compact}
    K=\left\{\btheta:\beta_0-1\leq \mu \leq \beta_0+1+M_0,\; 0\leq \tau \leq \xi_0 M_0,\; c_0\xi_0\leq \xi \leq C_0\xi_0\right\},
\end{equation}
in which
\begin{align*}
    M_0&=\max\left\{\frac{\tau_0C_0}{\xi_0}2.5^{C_0\xi_0+1},\frac{2^{\frac{1}{\kappa_0}}\Gamma^{\frac{1}{\kappa_0}}(1-\kappa_0\xi_0)\tau_0}{\xi_0}+\frac{2^{\frac{1}{\kappa_0}+1}\xi_0^{\frac{1}{\xi_0}-2}\kappa_1^{-\frac{1}{c_0\xi_0}}}{\tau_0c_0^3}\right\},\\
    \kappa_0&=\min\left\{\frac{1}{2},\frac{1}{2\xi_0}\right\},\kappa_1=\min\{1,c_0\xi_0\},
\end{align*}
and $c_0\in (0,1)$ is a fixed constant such that when $0<x<c_0$, $$\frac{1}{2}\log\Gamma\left(\frac{1}{x}\right)+\frac{1}{2}\log x+\frac{1}{20x}-\frac{21}{10}>0,$$ 
and $C_0>1$ is a fixed constant such that when $x>C_0$, $$\left(\frac{1}{x}-1\right)\frac{\log\tau_0}{\xi_0}+\log\Gamma\left(\frac{1}{x}\right)>0.$$
The construction of $K$ when $\xi_0<0$ is done similarly in the supplementary material. Note $K$ is only dependent on the value of $\btheta_0$, and for small $r$ defined in Section \ref{sec:(2)}, $\tilde{K}\subset K$.

\begin{proposition}[Step \textbf{\ref{step1}}]\label{prop:step1}
Let $Y_1,Y_2,\ldots \iidY P_{\btheta_0}$, and $(\mu_n(\xi),\tau_n(\xi))$ be the maximizer of $L_n$ on the cross section $\Omega_n(\xi)$. Then the global maximum must be in the compact set $K$ as defined in \eqref{K_compact} for large $n$ - that is,
$$\mathcal{P}\left(\exists N>0\text{ such that for all } n>N,\;\argmax_{\btheta\in \Theta_n}L_n(\btheta)\in K\right)=1.$$
\end{proposition}
\begin{proof}
Applying Proposition \ref{prop:profile_lik_uniqueML} and \ref{prop:xi_derivative} from Section \ref{sec:preliminary}, we will show in Appendix \ref{proof:step1} that 
\begin{equation}\label{A_cond}
    PL_n(\xi)<PL_n(\hat{\xi}_n)\text{ if }\xi\not\in [c_0\xi_0,C_0\xi_0],
\end{equation}
and 
\begin{equation}\label{B_cond}
    (\mu_n(\xi),\tau_n(\xi),\xi)\in K\text{ if }\xi\in [c_0\xi_0,C_0\xi_0],
\end{equation}
hold for all large $n$.

Since $\hat{\xi}_n$ converges quickly to $\xi_0$, we have $\hat{\xi}_n\in[c_0\xi_0,C_0\xi_0]$ for sufficiently large $n$. Denote $K_1=\{\btheta\in\Theta: c_0\xi_0\leq\xi\leq C_0\xi_0\}$. Clearly, $K\subsetneq K_1$ and \eqref{A_cond} implies
\begin{equation*}
    \argmax_{\btheta\in \Theta_n}L_n(\btheta)\in K_1.
\end{equation*}
When $\xi\in [c_0\xi_0,C_0\xi_0]$, \eqref{B_cond} encloses the unique maximizer $(\mu_n(\xi),\tau_n(\xi))$ on the cross section $\Omega_n(\xi)$ in $K$. Equivalently,
\begin{equation*}
    \argmax_{\btheta\in K_1}L_n(\btheta)\in K.
\end{equation*}
Combining \eqref{A_cond} and \eqref{B_cond},
\begin{equation*}
    \argmax_{\btheta\in \Theta_n}L_n(\btheta)\in K.
\end{equation*}
\end{proof}

\subsection{Completing the Proof of Theorem \ref{thm:global_mle}}
Proposition 2 in \citet{dombry2015existence} ascertained that for all large $n$, the argmax point on the set $K$ defined in Proposition \ref{prop:step1} is confined in any smaller neighborhood $\tilde{K}$. Although his result was developed within the framework of triangular arrays of block maxima, the proof can be adapted to work on iid GEV samples.

\begin{lemma}[Consistency]\label{lem:dombry}
Let $K\subset \Theta$ be a compact set that contains $\btheta_0$ as an interior point, and $Y_1, Y_2, \ldots$ be a sequence of independent and identically distributed random variables with common distribution $P_{\btheta_0}$. Then a sequence of estimators $\hat{\btheta}_n$ can be found to maximize the log-likelihood $L_n$ over $K$, and for any smaller neighborhood $\tilde{K}$ of $\btheta_0$ such that $\tilde{K}\subset K$,
$$\mathcal{P}\left(\exists N>0\text{ such that for all }n>N, \;\hat{\btheta}_n\in \tilde{K}\right)=1,$$
and hence $\hat{\btheta}_n\asconv \btheta_0$ as $n\rightarrow \infty$.
\end{lemma}
\begin{proof}
\citet{bucher2017maximum} noted that Proposition 2 in \citet{dombry2015existence} is applicable for unitary GEV distributions. Noticing that a GEV distribution is in its own domain of attraction, the block size sequence $m(n)$ is set to be $1$ with $a_m=\tau_0$ and $b_m=\mu_0$.
% @inproceedings{fisher1928limiting,
%   title={Limiting forms of the frequency distribution of the largest or smallest member of a sample},
%   author={Fisher, Ronald Aylmer and Tippett, Leonard Henry Caleb},
%   booktitle={Mathematical Proceedings of the Cambridge Philosophical Society},
%   volume={24, no. 2},
%   pages={180--190},
%   year={1928},
%   organization={Cambridge University Press}
% }

Following the lines of the proof in \citet{dombry2015existence}, $\tilde{K}$ is limited to be a ball neighborhood of $\btheta_0$ with an arbitrarily small radius. It is straightforward to generalize the proof to any small neighborhood of $\btheta_0$ such that $\tilde{K}\subset K$. Because the closure of the set $\Delta=K\setminus\tilde{K}$ is still compact, any open cover of $\Delta$ has a finite subcover, whence the remaining proof can be passed though without modification.
\end{proof}

Combining Proposition \ref{prop:step1} and Lemma \ref{lem:dombry}, we obtain
\begin{equation*}
    \argmax_{\btheta\in \Theta} L_n(\btheta)\in \tilde{K}\cap \Omega_n,
\end{equation*}
and by the local strict concavity in $\tilde{K}\cap \Omega_n$ ensured by Proposition \ref{prop:loc_concavity},
\begin{equation*}
    \hat{\btheta}_n=\argmax_{\tilde{K}\cap \Omega_n} L_n(\btheta),
\end{equation*}
whence we conclude that $\hat{\btheta}_n$ attains the unique and global maximum of $L_n$.

\section{Discussion}

In this paper, we proved the uniqueness and global optimity of the MLE for the generalized extreme value distribution. This improves on previous results where the existence and asymptotic normality of the local MLE in \citet{dombry2015existence} and \citet{bucher2017maximum} were established on a predetermined compact set. Our proof uses a two-step strategy. We first prove that the log-likelihood is strictly concave in a fixed small compact neighborhood when the sample size is large, and then enclose the global maximizer in a larger neighborhood. Finally, the two main results are integrated using Proposition 2 in \citet{dombry2015existence} to ensure that the local MLE truly maximizes the log-likelihood globally.

Intermediate results necessary for the the proofs of local strict concavity and boundedness of the global MLE unveiled additional strong and interesting characteristics of the GEV likelihood function that may be of independent interest. We found that the profile likelihood attains a unique maximum at each slice of the support, identified the convergence rate of the support boundary to the local MLE, and demonstrated the uniform consistency of a class of limit relations that are the building blocks of the Hessian matrix. These results enhance our understanding of the GEV distribution, and  in future work will become instrumental in establishing asymptotic properties of the GEV posterior distribution that are key in formulating rule-based reference analysis.

In applications, however, observations are never generated exactly from a GEV distribution; rather, they come from a distribution which we typically assume to be in the domain of attraction of a GEV. Dividing the observations into non-overlapping blocks, we make the approximating assumption that the maxima extracted from each block are GEV distributed. Thus, the asymptotic setup of \citet{dombry2015existence} should be viewed as the more realistic, and our work offers theoretical foundations for maximum likelihood estimation using the GEV when the block size is large.

Finally, the number of block maxima in any observational record is limited. For future research, it is important to examine the minimum sample size required for the observations to manifest large-sample behavior, as had been done for previous asymptotic results in extreme value statistics. Relatedly, small-sample estimators for the GEV tend to be unstable, so taking advantage of the profile likelihood might provide an effective---and to our knowledge unexplored---approach to estimating the shape parameter. 
That is, one could first calculate the maximum likelihood on the cross sections of the support at different levels of $\xi$, and then find the $\xi$ that maximizes the profile likelihood. Doing so is guaranteed (at least asymptotically) to find the global MLE, and might improve numerical stability in small samples.

\section*{Acknowledgments}
The authors gratefully acknowledge support from the US National Science Fountation NSF DMS-2001433.

\begin{appendix}
\numberwithin{equation}{section}
%%%%%%%%%%%%%%%%%%%%%%%%%%%%%%%%%%%%%%%%%%%%%%%%%%
%%                 Appendix A                   %%
%%%%%%%%%%%%%%%%%%%%%%%%%%%%%%%%%%%%%%%%%%%%%%%%%%

\section{Proofs concerning the profile likelihood}\label{proof:profile_lik}

For the proof of Proposition \ref{prop:profile_lik_uniqueML} and \ref{prop:xi_derivative}, we focus on the mathematical form of the log-likelihood $L_n$ while fixing $n$ and a sequence of numbers $Y_1,\ldots, Y_n$. We begin by recalling the Seitz inequality, which generalizes both Cauchy and Chebyshev's sum inequalities; see \citet{seitz1936remarque}. The following lemma is a special case of the original Seitz inequality.
\begin{lemma}[Seitz inequality]\label{seitz_ineq}
Let $\boldsymbol{x}=(x_1,\ldots,x_n)$, $\boldsymbol{y}=(y_1,\ldots,y_n)$, $\boldsymbol{z}=(z_1,\ldots,z_n)$ and $\boldsymbol{u}=(u_1,\ldots,u_n)$ be given sequences of real numbers. If for every pair of indices $i,j\;(i<j)$ and for every pair $r,s\;(r<s)$
\begin{equation*}
\begin{vmatrix}
x_i&x_j\\
y_i&y_j
\end{vmatrix}\times
\begin{vmatrix}
z_r&z_s\\
u_r&u_s
\end{vmatrix}\geq 0,
\end{equation*}
then the scalar products between the vectors satisfy
\begin{equation*}
    (\boldsymbol{x}\cdot \boldsymbol{z})\times (\boldsymbol{y}\cdot \boldsymbol{u})\geq  (\boldsymbol{y}\cdot \boldsymbol{z})\times (\boldsymbol{x}\cdot \boldsymbol{u}).
\end{equation*}
\end{lemma}

\begin{lemma}\label{lem:monotone_on_beta}
Suppose $Y_1, \ldots, Y_n\in \mathbb{R}$ are not all equal to each other, and fix any non-zero $\xi$ such that $-1<\xi<n-1$ . Define
$$H_n(\beta)=\frac{\sumN [\xi(Y_i-\beta)]^{-1-1/\xi}}{\sumN [\xi(Y_i-\beta)]^{-1/\xi}\sumN [\xi(Y_i-\beta)]^{-1}}-\frac{\xi+1}{n},$$
the domain of which is $\{\beta<Y_{(1)}\}$ if $\xi>0$, and $\{\beta>Y_{(n)}\}$ if $\xi<0$. Then $H_n$ is a strictly increasing function of $\beta$, and there exists a unique $\tilde{\beta}$ in the domain such that $H_n(\tilde{\beta})=0$, and $\text{sgn}(H_n(\beta))=\text{sgn}(\beta-\tilde{\beta})$ for any $\beta$ in the domain.
\end{lemma}
\begin{proof}
\textbf{A. $\xi>0$.} We first show that $H_n'(\beta)>0$ for $\beta\in (-\infty,Y_{(1)})$, where
\begin{small}
\begin{equation*}
    \begin{split}
        &H_n'(\beta)=\frac{\frac{1}{\xi}\sumN (Y_i-\beta)^{-1}\left[\sumN (Y_i-\beta)^{-2-\frac{1}{\xi}}\sumN (Y_i-\beta)^{-\frac{1}{\xi}}-\left(\sumN (Y_i-\beta)^{-1-\frac{1}{\xi}}\right)^2\right]}{\left(\sumN (Y_i-\beta)^{-1}\sumN (Y_i-\beta)^{-\frac{1}{\xi}}\right)^2}+\\
        &\frac{\sumN (Y_i-\beta)^{-\frac{1}{\xi}}\left[\sumN (Y_i-\beta)^{-2-\frac{1}{\xi}}\sumN (Y_i-\beta)^{-1}-\sumN (Y_i-\beta)^{-1-\frac{1}{\xi}}\sumN (Y_i-\beta)^{-2}\right]}{\left(\sumN (Y_i-\beta)^{-1}\sumN (Y_i-\beta)^{-\frac{1}{\xi}}\right)^2}.
    \end{split}
\end{equation*}
\end{small}
By Cauchy inequality,
$$\sumN (Y_i-\beta)^{-2-\frac{1}{\xi}}\sumN (Y_i-\beta)^{-\frac{1}{\xi}}>\left(\sumN (Y_i-\beta)^{-1-\frac{1}{\xi}}\right)^2.$$
Note the equality does not hold because $Y_1,\ldots,Y_n$ are not all equal to each other. Therefore, it suffices to prove 
\begin{equation}\label{seitz1}
    \sumN (Y_i-\beta)^{-2-\frac{1}{\xi}}\sumN (Y_i-\beta)^{-1}\geq \sumN (Y_i-\beta)^{-1-\frac{1}{\xi}}\sumN (Y_i-\beta)^{-2}.
\end{equation}
Denote $x_i=(Y_{(i)}-\beta)^{-\frac{1}{\xi}}$, $y_i=1$, $z_i=(Y_{(i)}-\beta)^{-2}$ and $u_i=(Y_{(i)}-\beta)^{-1}$. Since for any pairs of indices $i<j$ and $r<s$,
\begin{equation*}
\begin{vmatrix}
x_i&x_j\\
y_i&y_j
\end{vmatrix}\times
\begin{vmatrix}
z_r&z_s\\
u_r&u_s
\end{vmatrix}=
\frac{[(Y_{(r)}-\beta)^{-1}-(Y_{(s)}-\beta)^{-1}][(Y_{(i)}-\beta)^{-\frac{1}{\xi}}-(Y_{(j)}-\beta)^{-\frac{1}{\xi}}]}{(Y_{(r)}-\beta)(Y_{(s)}-\beta)}>0,
\end{equation*}
we can apply the Seitz inequality to obtain \eqref{seitz1}, and hence  $H_n'(\beta)>0$ for $\beta\in (-\infty,Y_{(1)})$.

On the other hand, we can easily check the limits of $H_n$ on the bounds of the domain:
\begin{equation*}
    \begin{split}
        \lim_{\beta\searrow-\infty}H_n(\beta)&=\frac{1}{n}-\frac{\xi+1}{n}=-\frac{\xi}{n}<0,\\
        \lim_{\beta\nearrow Y_{(1)}}H_n(\beta)&=1-\frac{\xi+1}{n}>0.
    \end{split}
\end{equation*}
By the intermediate zero theorem, there exists an unique $\tilde{\beta}\in(-\infty,Y_{(1)})$ that is determined by $\xi$ and the $Y_i$'s such that $H_n(\tilde{\beta})=0$, and $H_n(\beta)<0$ when $\beta<\tilde{\beta}$, or vice versa.

Furthermore, when the fixed $\xi$ diminishes to 0, $\lim_{\beta\searrow-\infty}H_n(\beta)$ is close to 0. This means the root can only be found close to the boundary because $H_n$ is increasing and thus positive almost everywhere on $(-\infty,Y_{(1)})$. Treat $\tilde{\beta}$ as a function of fixed $\xi$, and we conclude
$$\tilde{\beta}\rightarrow -\infty, \text{ as }\xi\searrow 0.$$
Similarly, when the fixed $\xi$ approaches $n-1$ from below, $\lim_{\beta\nearrow Y_{(1)}}H_n(\beta)$ is close to 0, and $H_n$ is negative almost everywhere on $(-\infty,Y_{(1)})$. Therefore,
$$\tilde{\beta}\rightarrow Y_{(1)}, \text{ as }\xi\nearrow n-1.$$

\textbf{B. $-1<\xi<0$.} Following a similar proof, we can verify that $H'_n(\beta)>0$ is also true in this case. We now evaluate the limits of $H_n$ on the bounds of the domain:
\begin{equation*}
    \begin{split}
        \lim_{\beta\searrow Y_{(n)}}H_n(\beta)&=-\frac{\xi+1}{n}<0,\\
        \lim_{\beta\nearrow \infty}H_n(\beta)&=\frac{1}{n}-\frac{\xi+1}{n}=-\frac{\xi}{n}>0.
    \end{split}
\end{equation*}
Therefore, there exists an unique $\tilde{\beta}\in(Y_{(n)},\infty)$ that is determined by the fixed $\xi$ and $Y_i$'s such that $H_n(\tilde{\beta})=0$, and $H_n(\beta)<0$ when $\beta<\tilde{\beta}$, or vice versa.

As we move $\xi$ to $-1$ from the right and 0 from the left, it can be shown accordingly that
\begin{equation*}
        \tilde{\beta}\rightarrow Y_{(n)}, \text{ as }\xi\searrow -1,
        \text{ and }\tilde{\beta}\rightarrow \infty, \text{ as }\xi\nearrow 0.
\end{equation*}
\end{proof}

\subsection{Proof of Proposition \ref{prop:profile_lik_uniqueML}}

\begin{proof}
% Fix $n$ and $\xi\neq 0$, and the cross section of interest can be written as
% \begin{equation*}
%     \begin{split}
%         \Omega_n(\xi)&=\{(\tau,\mu,\xi):\beta<Y_{(1)},\tau>0\} \text{ if }\xi>0\\
%         &\text{or }\{(\tau,\mu,\xi):\beta>Y_{(n)},\tau>0\}\text{ if }\xi<0.
%     \end{split}
% \end{equation*}
By \eqref{wi} and \eqref{Log-lik}, the log-likelihood function on one cross section $\Omega_n(\xi)$ can be expressed in terms of $(\tau,\beta)$:
\begin{equation*}
    \begin{split}
        h(\tau,\beta):=&L_n(\tau,\beta,\xi)\\
        =&\frac{n}{\xi}\log\tau-\frac{\xi+1}{\xi}\sumN\log[\xi(Y_i-\beta)]-\tau^{1/\xi}\sumN [\xi(Y_i-\beta)]^{-1/\xi}.
    \end{split}
\end{equation*}
Now we profile out $\tau$ via fixing $\beta$. Taking the partial derivative with respect to $\tau$, we obtain
$$\frac{\partial h}{\partial \tau}=\frac{n}{\xi\tau}-\tau^{1/\xi}\frac{1}{\xi\tau}\sumN [\xi(Y_i-\beta)]^{-1/\xi}.$$
The root of the above partial derivative is $$\tau(\beta)=\left\{\frac{1}{n}\sumN [\xi(Y_i-\beta)]^{-1/\xi}\right\}^{-\xi}.$$
For both $\xi>0$ and $\xi<0$, it holds that $\frac{\partial h}{\partial \tau}>0$ when $\tau<\tau(\beta)$, and $\frac{\partial h}{\partial \tau}<0$ when $\tau>\tau(\beta)$. Fixing $\beta$, $h$ as a function of $\tau$ monotonically increases and then decreases when $\tau$ grows from $0$ to $\infty$. Thus, the maximizer of $L_n$ on $\Omega_n(\xi)$ must be on the parametric graph $\{(\beta,\tau):\tau=\tau(\beta)\}$. To locate the maximizer, we simply need to find $\beta$ such that $h(\beta,\tau(\beta))$ is maximized.

Simple calculations lead to
$$h_1(\beta):=h(\tau(\beta),\beta)=-n\log\left\{\frac{1}{n}\sumN [\xi(Y_i-\beta)]^{-1/\xi}\right\}-\frac{\xi+1}{\xi}\sumN\log[\xi(Y_i-\beta)]-n.$$
Similarly, we examine the derivative:
\begin{equation*}%\label{important_deriv}
    h_1'(\beta)=\sumN [\xi(Y_i-\beta)]^{-1}\left\{(\xi+1)-\frac{n\sumN [\xi(Y_i-\beta)]^{-1-1/\xi}}{\sumN [\xi(Y_i-\beta)]^{-1/\xi}\sumN [\xi(Y_i-\beta)]^{-1}}\right\}.
\end{equation*}
By Lemma \ref{lem:monotone_on_beta}, there exists an unique $\tilde{\beta}$ in the domain such that $h_1'(\tilde{\beta})=0$. Meanwhile, $h_1'(\beta)>0$ when $\beta<\tilde{\beta}$ and $h_1'(\beta)<0$ when $\beta>\tilde{\beta}$. That is, $h_1(\beta)$ attains the global unique maximum at $\tilde{\beta}$.

To sum up, there exists a maximizer $(\tau(\tilde{\beta}),\tilde{\beta})$ of $L_n$ on the cross section $\Omega_n(\xi)$ that is unique and global with respect to $\Omega_n(\xi)$. The sufficient and necessary condition for $(\tilde{\tau},\tilde{\beta})$ to be the maximizer is \eqref{max_cross_section_cond}. The proof for maximizing $L_n$ on $\Omega_n(0)$ is deferred to the supplementary material.

Since $(\tau,\mu)\mapsto(\tau,\beta)$ is a one-to-one mapping, it's clear that $(\tilde{\tau},\tilde{\mu})=(\tilde{\tau},\tilde{\beta}+\tilde{\tau}/\xi)$ is the unique maximizer on $\Omega_n(\xi)$ under the parametrization $(\tau,\mu)$. We write $(\tilde{\tau},\tilde{\mu})$ as $(\tau_n(\xi),\mu_n(\xi))$ to highlight that  $(\tilde{\tau},\tilde{\mu})$ depends on $Y_1,\ldots,Y_n$ and $\xi$.
\end{proof}

\vskip 0.3cm
\subsection{Proof of Proposition \ref{prop:xi_derivative}}
% \BASadd{Should this be a subsection?}
\begin{proof}
For notational simplicity, we will denote $\hat{\delta}_i=\hat{\xi}_n(Y_i-\hat{\beta}_n)$ and $\delta_i:=\delta_i(\xi,n)=\xi(Y_i-\beta_n(\xi))$, where $\beta_n(\xi)=\mu_n(\xi)-\tau_n(\xi)/\xi$. We calculate %derived from the cross section $\Omega_n(\xi)$ to $\beta_n$ when no confusion can arise. We have %\BASadd{It might be better to just put in the limits for the summations.  It's just cut and paste, right?}\LZadd{You mean to move this proof to later sections?}\BASadd{No, I just mean replace $\sumN $ with $\sum_{i=1}^{n}$ in the formulas in this section.  This is an un-important detail.}\LZadd{Got it. I will replace them.}
\begin{equation*}
    \begin{split}
        \left\{\sumN\delta_i^{-1/\xi}\right\}'&=\left\{\sumN[\xi(Y_i-\beta_n(\xi))]^{-1/\xi}\right\}'\\
        &=\frac{1}{\xi^2} \sumN\delta_i^{-1/\xi}\log \delta_i+\beta'_n(\xi)\sumN\delta_i^{-1-1/\xi}-\frac{1}{\xi^2}\sumN\delta_i^{-1/\xi},
    \end{split}
\end{equation*}
and by \eqref{eqn:PL_xi},
\begin{equation}\label{eqn:PL_dev}
    \begin{split}
        PL'_n(\xi)&=-\frac{n\left\{\sumN\delta_i^{-1/\xi}\right\}'}{\sumN\delta_i^{-1/\xi}}+\frac{\sumN\log\delta_i}{\xi^2}+\frac{\xi+1}{\xi}\sumN \left\{\log\delta_i\right\}'\\
        &=-\frac{n}{\xi}-\frac{n\sumN\delta_i^{-1/\xi}\log\delta_i}{\xi^2\sumN\delta_i^{-1/\xi}}+\frac{1}{\xi^2}\sumN \log\delta_i\\
        &\hspace{2cm}+\beta_n'(\xi)\left\{(\xi+1)\sumN\delta_i^{-1}-\frac{n\sumN\delta_i^{-1-1/\xi}}{\sumN\delta_i^{-1/\xi}}\right\},
    \end{split}
\end{equation}
in which the term in the curly brackets is zero for $\xi\neq 0$ due to the second equation of \eqref{max_cross_section_cond}.

Next we examine the limits of $ PL'_n(\xi)$ as $\xi\nearrow n-1$, $\xi\rightarrow 0$ and $\xi\searrow -1$.

\textbf{A. $\xi\nearrow n-1$.} From Lemma \ref{lem:monotone_on_beta}, we know that
$$\beta_n(\xi)\nearrow Y_{(1)}, \text{ as }\xi\nearrow n-1,$$
and thus
\begin{small}
\begin{equation*}
    \begin{split}
        \sum_{i=2}^n (Y_{(i)}-\beta_n(\xi))^{-1/\xi}\log(Y_{(i)}-\beta_n(\xi))&\rightarrow \sum_{i=2}^n \left(Y_{(i)}-Y_{(1)}\right)^{-1/(n-1)}\log\left(Y_{(i)}-Y_{(1)}\right)=:C_1,\\
        \sum_{i=2}^n (Y_{(i)}-\beta_n(\xi))^{-1/\xi}&\rightarrow \sum_{i=2}^n \left(Y_{(i)}-Y_{(1)}\right)^{-1/(n-1)}=:C_2,\\
        \sum_{i=2}^n \log(Y_{(i)}-\beta_n(\xi))&\rightarrow \sum_{i=2}^n\log\left(Y_{(i)}-Y_{(1)}\right)=:C_3.\\
    \end{split}
\end{equation*}
\end{small}
As a result,
\begin{footnotesize}
\begin{equation*}
    \begin{split}
        \lim_{\xi\nearrow n-1}PL'_n(\xi)=&-\lim_{\xi\nearrow n-1}\frac{1}{\xi^2}\left[\frac{nC_1+n(Y_{(1)}-\beta_n(\xi))^{-1/\xi}\log(Y_{(1)}-\beta_n(\xi))}{C_2+(Y_{(1)}-\beta_n(\xi))^{-1/\xi}}-\log(Y_{(1)}-\beta_n(\xi))\right]+\\
        &\hspace{5cm}\frac{C_3}{(n-1)^2}-\frac{n}{n-1}\\
        =&-\lim_{\xi\nearrow n-1}\frac{(n-1)\log(Y_{(1)}-\beta_n(\xi))}{\xi^2}+\frac{C_3}{(n-1)^2}-\frac{n}{n-1}=\infty.
    \end{split}
\end{equation*}
\end{footnotesize}

\textbf{B. $\xi\rightarrow 0$.} In this case, we need to be more cautious about $\beta_n'(\xi)$ and the term in the curly brackets in \eqref{eqn:PL_dev}. It holds that as $\xi\rightarrow 0$,
\begin{equation}\label{eqn:PL_n_0_dev}
    \begin{split}
        PL'_n(\xi)=&\frac{n\mu'_n(0)-\sumN [Y_i-\mu_n(0)+\tau'_n(0)]}{\tau_n(0)}+\\
        &\frac{\sumN [Y_i-\mu_n(0)+\tau'_n(0)]^2-n\tau'_n(0)^2}{2\tau_n(0)^2}+o(1),
    \end{split}
\end{equation}
the proof of which is deferred to the supplementary material.

\textbf{C. $\xi\searrow -1$.} Denote $\delta_{(i)}=\xi(Y_{(i)}-\beta_n(\xi))$. From Lemma \ref{lem:monotone_on_beta}, we know that
$$\beta_n(\xi)\searrow Y_{(n)}, \text{ as }\xi\searrow -1.$$
Thus, $\delta_{(n)}=\xi(Y_{(n)}-\beta_n(\xi))\searrow 0$ and
\begin{equation*}
    \begin{split}
        \sum_{i=1}^{n-1} \delta_{(i)}^{-1/\xi}\log\delta_{(i)}&\rightarrow \sum_{i=1}^{n-1} \left(Y_{(n)}-Y_{(i)}\right)\log\left(Y_{(n)}-Y_{(i)}\right)=:C'_1,\\
        \sum_{i=1}^{n-1} \delta_{(i)}^{-1/\xi}&\rightarrow \sum_{i=1}^{n-1} \left(Y_{(n)}-Y_{(i)}\right)=:C'_2,\\
        \sum_{i=1}^{n-1} \log\delta_{(i)}&\rightarrow \sum_{i=1}^{n-1}\log\left(Y_{(n)}-Y_{(i)}\right)=:C'_3.\\
    \end{split}
\end{equation*}
As a result,
\begin{equation*}
    \begin{split}
        \lim_{\xi\searrow -1}PL'_n(\xi)&=n+C'_3-\lim_{\xi\searrow -1}\left\{\frac{nC'_1+n\delta_{(n)}^{-1/\xi}\log\delta_{(n)}}{\xi^2C'_2+\xi^2\delta_{(n)}^{-1/\xi}}-\frac{\log\delta_{(n)}}{\xi^2}\right\}\\
        &=n+C'_3-\frac{nC'_1}{C'_2}+\lim_{\xi\searrow -1}\left\{\frac{\log\delta_{(n)}}{\xi^2}\right\}=-\infty.
    \end{split}
\end{equation*}
\end{proof}

%%%%%%%%%%%%%%%%%%%%%%%%%%%%%%%%%%%%%%%%%%%%%%%%%%
%%                 Appendix B                   %%
%%%%%%%%%%%%%%%%%%%%%%%%%%%%%%%%%%%%%%%%%%%%%%%%%%

\section{Convergence rate of the support boundary}\label{proof:converge_rate}
\begin{proposition}\label{rate_MLE}
Suppose $Y_1,\ldots, Y_n$ are i.i.d. samples from some parametric model $\mathcal{M}=\{p(y|\btheta): y\in \mathcal{Y}, \btheta\in \Omega\}$, where $\mathcal{Y}$ may either depend on $\btheta$ or not. If $\hat{\beta}_n=\hat{\beta}(Y_1,\ldots, Y_n)$ is a strongly consistent estimator of $\beta=\beta(\btheta)$ that satisfies $\sqrt{n}(\hat{\beta}_n-\beta)=V+o_p(1)$ and $\hat{\beta}_n\asconv \beta$, where $V$ is a continuous random variable. Then
\begin{equation*}
    n^\gamma (\hat{\beta}_n-\beta)\asconv 0,\;\forall \gamma\in (0,1/2).
\end{equation*}
\end{proposition}
\begin{proof}
Fix $\gamma\in (0,1/2)$, and find integer $q$ such that $q(1/2-\gamma)>1$. Since $\sqrt{n}(\hat{\beta}_n-\beta)/V-1=o_p(1)$, we know that $E\left(|\sqrt{k}(\hat{\beta}_k-\beta)/V-1|^q\right)\rightarrow 0$ as $k\rightarrow \infty$, that is, there exists $N_q$, such that $E\left(|\sqrt{k}(\hat{\beta}_k-\beta)/V-1|^q\right)<1$ for all $k>N_q$. 

First, we prove 
\begin{equation}\label{Oas}
    P(\left\{\omega\in\mathcal{X}: \exists N \text{ such that, }\forall n\geq N,\; n^\gamma|\hat{\beta}_n(\omega)-\beta|<|V(\omega)|\right\})=1.
\end{equation}
Define
\begin{equation*}
    A_n=\bigcap_{k=n}^\infty \left\{\omega\in \mathcal{X}: k^\gamma|\hat{\beta}_k(\omega)-\beta|<|V(\omega)|\right\}.
\end{equation*}
Since $A_n\subseteq A_{n+1}$, it is clear that \eqref{Oas} is equivalent to $P(A_n)\rightarrow 1$ as $n\rightarrow\infty$. However,
\begin{equation*}
\begin{split}
     P(\left\{\omega\in \mathcal{X}: k^\gamma|\hat{\beta}_k(\omega)-\beta|\geq\right.&|V(\omega)|\Big\})=P(k^{\frac{1}{2}}|\hat{\beta}_k-\beta|-|V|\geq(k^{\frac{1}{2}-\gamma}-1)|V|)\\
     &\leq P\left(\left|\frac{\sqrt{k}(\hat{\beta}_k-\beta)}{V}-1\right|\geq k^{\frac{1}{2}-\gamma}-1\right)+P(V=0)\\
     &\leq \frac{E\left(|\sqrt{k}(\hat{\beta}_k-\beta)/V-1|^q\right)}{(k^{\frac{1}{2}-\gamma}-1)^q}\leq \frac{1}{(k^{\frac{1}{2}-\gamma}-1)^q}, \;\forall k>N_q,
\end{split}
\end{equation*}
in which we utilized Markov's inequality and the fact that $V$ is continuous. 

Note $\sum_{k=2}^\infty \frac{1}{(k^{1/2-\gamma}-1)^q}$ is convergent for $q(1/2-\gamma)>1$. Countable subadditivity implies
\begin{equation*}
    \begin{split}
        P(A_n^c)=P\left(\bigcup_{k=n}^\infty \left\{\omega\in \mathcal{X}: k^\gamma|\hat{\beta}_k(\omega)-\beta|\geq|V(\omega)|\right\}\right)\leq \sum_{k=n}^\infty \frac{1}{(k^{\frac{1}{2}-\gamma}-1)^q}\rightarrow 0,
    \end{split}
\end{equation*}
which completes the proof of \eqref{Oas}.

Since $\gamma<1/2$, there exists $\eta>0$ such that $1-2\gamma-\eta>0$ and $\frac{\gamma}{2\gamma+\eta}<1/2$. Define
\begin{equation*}
    \begin{split}
        \gamma=&\left\{\omega\in \mathcal{X}: \lim_{n\rightarrow\infty}\hat{\beta}_n(\omega)=\beta\right\},\;\; C=\{\omega\in \mathcal{X}:V(\omega)\neq 0\},\\
        D=&\left\{\omega\in\mathcal{X}: \exists N \text{ such that, }\forall n\geq N,\; n^{\gamma/(2\gamma+\eta)}|\hat{\beta}_n(\omega)-\beta|<|V(\omega)|\right\}.
    \end{split}
\end{equation*}

By definition of strong consistency and the preceding discussion, $P(B)=P(C)=P(D)=1$, which implies $P(B\cap C \cap D)=1-P(B^c\cup C^c\cup D^c)\geq 1-P(B^c)-P(C^c)-P(D^c)=1$, i.e. $P(B\cap C \cap D)=1$. Focusing on a particular element $\omega \in B\cap C\cap D$, we have for $\epsilon>0$, there exists $N_1$ such that, $|\hat{\beta}_{n}(\omega)-\beta|<\left(\epsilon/|V(\omega)|^{2\gamma+\eta}\right)^{1/(1-2\gamma-\eta)}$ for any $n>N_1$; there also exists $N_2$ such that, $n^{\gamma/(2\gamma+\eta)}|\hat{\beta}_n(\omega)-\beta|<|V(\omega)|$ for any $n>N_2$. Therefore, for any $n>\max\{N_1,N_2\}$,
\begin{small}
\begin{equation*}
   n^\gamma|\hat{\beta}_n(\omega)-\beta|=\left[n^{\gamma/(2\gamma+\eta)}|\hat{\beta}_n(\omega)-\beta|\right]^{2\gamma+\eta}\cdot |\hat{\beta}_n(\omega)-\beta|^{1-2\gamma-\eta}\leq |V(\omega)|^{2\gamma+\eta}\cdot \frac{\epsilon}{|V(\omega)|^{2\gamma+\eta}}=\epsilon. 
\end{equation*}
\end{small}
Equivalently, if $\omega \in B\cap C\cap D$, then $n^\gamma(\hat{\beta}_n(\omega)-\beta)\rightarrow 0$. Since $P(B\cap C\cap D)=1$, $n^\gamma (\hat{\beta}_n-\beta)\asconv 0$.
\end{proof}

\vskip 0.6cm
\subsection{Proof of Proposition \ref{prop:rate_Min}}
% \BASadd{Should this be a subsection?}

\begin{proof}
Here we only provide the proof of \ref{boundary_pos}, the positive shape $\xi_0$. The proof of the case $\xi_0<0$ is analogous.

(i) \textbf{Prove $(\log n)^{(1+\gamma)\xi_0}(Y_{(1)}-\beta_0)\asconv \infty$.} Note that $Y_i>\beta_0,\; i=1,\ldots,n$, and $Y_{(1)}>\beta_0$. Fix $M>0$, and define
\begin{equation*}
\begin{split}
    A_n&=\left\{\omega:\forall k\geq n, (\log k)^{(1+\gamma)\xi_0}\left(\min_{1\leq i\leq k}Y_i(\omega)-\beta_0\right)\geq M\right\}\\
    &=\left\{\omega:Y_k\geq \beta_0+\frac{M}{(\log n)^{(1+\gamma)\xi_0}},\; 1\leq k\leq n,\text{ and } Y_k\geq \beta_0+\frac{M}{(\log k)^{(1+\gamma)\xi_0}},\; k\geq n+1\right\}.
\end{split}
\end{equation*}
Then $(\log n)^{(1+\gamma)\xi_0}(Y_{(1)}-\beta_0)\asconv \infty$ if and only if $P(A_n)\rightarrow 1$ for any $M>0$.

We can easily calculate
\begin{equation*}
\begin{split}
    P(A_n)&=\left[P\left(Y_1\geq\beta_0+\frac{M}{(\log n)^{(1+\gamma)\xi_0}}\right)\right]^n\cdot\prod_{k=n+1}^{\infty}P\left(Y_k\geq\beta_0+\frac{M}{(\log k)^{(1+\gamma)\xi_0}}\right)\\
    &=\left[1-\exp\left(-M_1(\log n)^{1+\gamma}\right)\right]^n\prod_{k=n+1}^{\infty}\left[1-\exp\left(-M_1(\log k)^{1+\gamma}\right)\right],
\end{split}
\end{equation*}
where constant $M_1=\left({M\xi_0}/{\tau_0}\right)^{-1/\xi_0}>0$. By L'Hospital rule,
\begin{footnotesize}
\begin{equation}\label{l'hospital}
     \lim_{x\rightarrow\infty}x\log\left\{1-\exp\left(-M_1 (\log x)^{1+\gamma}\right)\right\}=-M_1(1+\gamma)\lim_{x\rightarrow\infty}\frac{x\log^\gamma x\cdot\exp\left(-M_1 (\log x)^{1+\gamma}\right)}{1-\exp\left(-M_1 (\log x)^{1+\gamma}\right)}=0.
\end{equation}
\end{footnotesize}
Thus, $\left[1-\exp\left(-M_1(\log n)^{1+\gamma}\right)\right]^n\rightarrow\exp(0)=1$ as $n\rightarrow\infty$.

\vskip 0.2cm
On the other hand, consider $\sum_{k=1}^\infty\log\left\{1-\exp\left(-M_1(\log k)^{1+\gamma}\right)\right\}$. Apply L'Hospital rule again, and we know
\begin{footnotesize}
\begin{equation*}
     \lim_{x\rightarrow\infty}x^2\log\left\{1-\exp\left(-M_1 (\log x)^{1+\gamma}\right)\right\}=-\frac{M_1(1+\gamma)}{2}\lim_{x\rightarrow\infty}\frac{x^2\log^\gamma x\cdot\exp\left(-M_1 (\log x)^{1+\gamma}\right)}{1-\exp\left(-M_1 (\log x)^{1+\gamma}\right)}=0.
\end{equation*}
\end{footnotesize}
Thus, $k^2\log\left\{1-\exp\left(-M_1(\log k)^{1+\gamma}\right)\right\}\rightarrow 0$ as $k\rightarrow\infty$. By limit comparison test, we know that $\sum_{k=1}^\infty\log\left\{1-\exp\left(-M_1(\log k)^{1+\gamma}\right)\right\}<\infty$. Hence, as $n\rightarrow \infty$,
\begin{equation*}
    \prod_{k=n+1}^{\infty}\left[1-\exp\left(-M_1(\log k)^{1+\gamma}\right)\right]=\exp\left\{\sum_{k=n+1}^\infty\log\left\{1-\exp\left(-M_1(\log k)^{1+\gamma}\right)\right\}\right\}\rightarrow 1,
\end{equation*}
which completes proving $P(A_n)\rightarrow 1$ for any $M>0$.

\vskip 0.37cm
\noindent (ii) \textbf{Prove $(\log n)^{(1-\gamma)\xi_0}(Y_{(1)}-\beta_0)\asconv0$.} If $\gamma\geq 1$, the second strong convergence holds because $Y_{(1)}\asconv \beta_0$. The more interesting case is when $1-\gamma>0$.

We fix $\epsilon>0$, and define
\begin{equation*}
\begin{split}
     B_n&=\bigg\{\omega: \exists k\geq n, \log^{(1-\gamma)\xi_0} (k)\cdot\left(\min_{1\leq i\leq k}Y_i(\omega)-\beta_0\right)>\epsilon\bigg\}=\bigcup_{k=n}^\infty C_k,
\end{split}
\end{equation*}
where $C_k=\left\{\omega: \log^{(1-\gamma)\xi_0} k\cdot\left(\min_{1\leq i\leq k}Y_i(\omega)-\beta_0\right)>\epsilon\right\}$, and
\begin{equation*}
        P( C_k)=\prod_{i=1}^k P\left(Y_i\geq \beta_0+\frac{\epsilon}{\log^{(1-\gamma)\xi_0}(k)}\right)=[1-\exp\left(-M_2\log^{1-\gamma}(k)\right)]^k,
\end{equation*}
where $M_2=\left({\xi_0\epsilon}/{\tau_0}\right)^{-1/\xi_0}>0$. 

Now we want to prove $k^{2}P(C_k)\rightarrow 0$ as $k\rightarrow\infty$. Application of L'Hospital rule similar to that leading to \eqref{l'hospital} yields
\begin{footnotesize}
\begin{equation*}
\begin{split}
    \lim_{x\rightarrow\infty}&\frac{x\log[1-\exp\left(-M_2\log^{1-\gamma}(x)\right)]}{\log x}=\lim_{x\rightarrow\infty}\frac{M_2(1-\gamma)\log^{-\gamma}(x)x^{-1}\exp\left(-M_2\log^{1-\gamma}(x)\right)}{(x^{-2}-x^{-2}\log x)[1-\exp\left(-M_2\log^{1-\gamma}(x)\right)]}\\
    &=-M_2(1-\gamma)\lim_{x\rightarrow\infty}\frac{\log(x)}{\log(x)-1}\cdot \frac{\exp\left(-M_2\log^{1-\gamma}(x)\right)}{1-\exp\left(-M_2\log^{1-\gamma}(x)\right)}\cdot x\log^{-\gamma-1}(x)\\
    &=-M_2(1-\gamma)\lim_{x\rightarrow\infty}\sqrt{x}\log^{-\gamma-1}(x)\cdot\exp\left(-M_2\log^{1-\gamma}(x)+\frac{1}{2}\log(x)\right) =-\infty,
\end{split}
\end{equation*}
\end{footnotesize}
which in turn results in
\begin{footnotesize}
\begin{equation*}
    \log \{x^2[1-\exp\left(-M_2\log^{1-\gamma}(x)\right)]^x\}=\log x\cdot\left\{\frac{x\log[1-\exp\left(-M_2\log^{1-\gamma}(x)\right)]}{\log x}+2\right\}\rightarrow -\infty,\text{ as } x\rightarrow\infty.
\end{equation*}
\end{footnotesize}
Therefore, we have $P(C_k)=o(1/k^{2})$. 

By the limit comparison test, the series $\sum_{k=1}^\infty P(C_k)$ is convergent, and its tail satisfies
\begin{equation*}
    \forall \epsilon>0, \;\;P(B_n)=P\left(\bigcup_{k=n}^\infty C_k\right)\leq \sum_{k=n}^\infty P(C_k)\rightarrow 0,\text{ as } n\rightarrow\infty.
\end{equation*}
This proves $(\log n)^{(1-\gamma)\xi_0}\left(\min_{i}Y_i-\beta_0\right)\asconv 0,\;\forall \gamma\in (0,1)$.

\vskip 0.37cm
\noindent (iii) \textbf{Prove $ n^{-(1+\gamma)\xi_0}Y_{(n)}\asconv 0$.} Similarly to (i), we fix $\epsilon>0$, and define
\begin{equation*}
\begin{split}
    D_n&=\left\{\omega:\forall k\geq n, k^{-(1+\gamma)\xi_0}\left(\max_{1\leq i\leq k}Y_i(\omega)-\beta_0\right)\leq \epsilon\right\}\\
    &=\left\{\omega:Y_k\leq \beta_0+\epsilon n^{(1+\gamma)\xi_0},\; 1\leq k\leq n,\text{ and } Y_k\leq \beta_0+\epsilon k^{(1+\gamma)\xi_0}, \;  k\geq n+1\right\}.
\end{split}
\end{equation*}
To prove $n^{-(1+\gamma)\xi_0}Y_{(n)}\asconv  0$, it suffices to show $n^{-(1+\gamma)\xi_0}(Y_{(n)}-\beta_0)\asconv  0$, which is true if and only if $P(D_n)\rightarrow 1$ for any $\epsilon>0$.

The independence of $Y_k$, $k\geq 1$ gives 
\begin{equation*}
    P(D_n)=\exp\left(-\frac{M_2}{n^\gamma}\right)\exp\left(-M_2\sum_{k=n+1}^\infty \frac{1}{k^{1+\gamma}}\right),
\end{equation*}
where $M_2=\left({\xi_0\epsilon}/{\tau_0}\right)^{-1/\xi_0}>0$. Since $1+\gamma>1$, we know by limit comparison test that $\sum_{k=n+1}^\infty 1/k^{1+\gamma}\rightarrow 0$, whence $P(D_n)\rightarrow 1$ as $n\rightarrow\infty$.

\vskip 0.37cm
\noindent (iv) \textbf{Prove $n^{-(1-\gamma)\xi_0}Y_{(n)}\asconv\infty$}. Similar to (ii), we only consider $0<\gamma<1$, and for $M>0$ define 
\begin{equation*}
\begin{split}
     E_n&=\bigg\{\omega: \exists k\geq n, k^{-(1-\gamma)\xi_0}\left(\max_{1\leq i\leq k}Y_i(\omega)-\beta_0\right)< M\bigg\}=\bigcup_{k=n}^\infty F_k,
\end{split}
\end{equation*}
where $F_k=\{\omega:k^{-(1-\gamma)\xi_0}(\max_i Y_i(\omega)-\beta_0)< M\}$, and
\begin{equation*}
    P(F_k)=\prod_{i=1}^k P\left(Y_i\leq \beta_0+M k^{(1-\gamma)\xi_0}\right)=\exp(-M_1 k^\gamma),
\end{equation*}
where $M_1=\left({M\xi_0}/{\tau_0}\right)^{-1/\xi_0}>0$.

Clearly the series $\sum_{k=1}^\infty P(F_k)$ is convergent, and its tail satisfies
\begin{equation*}
    \forall \epsilon>0, \;\;P(E_n)=P\left(\bigcup_{k=n}^\infty F_k\right)\leq \sum_{k=n}^\infty P(F_k)\rightarrow 0,\text{ as } n\rightarrow\infty.
\end{equation*}
This proves $n^{-(1-\gamma)\xi_0}\left(Y_{(n)}-\beta_0\right)\asconv \infty$, which implies $n^{-(1-\gamma)\xi_0}Y_{(n)}\asconv \infty$.
\end{proof}

\vskip 0.5cm
\begin{corollary}\label{rate_comp}
Suppose $Y_1,\ldots, Y_n\iidY P_{\btheta_0}$. If $\xi_0>0$, for any $ \gamma\in [0,1/2)$,
\begin{equation*}
    \frac{n^\gamma(\hat{\beta}_n-\beta_0)}{Y_{(1)}-\beta_0}\asconv 0.
\end{equation*}

If $-1/2<\xi_0<0$, for any $\gamma\in [0,\xi_0+1/2)$,
\begin{equation*}
    \frac{n^\gamma(\hat{\beta}_n-\beta_0)}{\beta_0-Y_{(n)}}\asconv 0.
\end{equation*}
\end{corollary}

\begin{proof}
When $\xi_0>0$, choose $\eta>0$ such that $\gamma+\eta<1/2$. By Proposition \ref{rate_MLE}, $n^{\gamma+\eta}(\hat{\beta}_n-\beta_0)\asconv 0$, and by Proposition \ref{prop:rate_Min}, 
\begin{equation*}
    \frac{n^\gamma(\hat{\beta}_n-\beta_0)}{Y_{(1)}-\beta_0}=\frac{n^{\gamma+\eta}(\hat{\beta}_n-\beta_0)}{n^\eta(Y_{(1)}-\beta_0)}\asconv 0.
\end{equation*}

When $-1/2<\xi_0<0$, since $-\xi_0<1/2-\gamma$, choose $\eta>0$ such that $-\xi_0<\eta<1/2-\gamma$. By Proposition \ref{prop:rate_Min}, $n^\eta(\beta_0-Y_{(n)})\asconv \infty$, and by Proposition \ref{rate_MLE},
\begin{equation*}
    \frac{n^\gamma(\hat{\beta}_n-\beta_0)}{\beta_0-Y_{(n)}}=\frac{n^{\gamma+\eta}(\hat{\beta}_n-\beta_0)}{n^\eta(\beta_0-Y_{(n)})}\asconv 0.
\end{equation*}
\end{proof}

%%%%%%%%%%%%%%%%%%%%%%%%%%%%%%%%%%%%%%%%%%%%%%%%%%
%%                 Appendix C                   %%
%%%%%%%%%%%%%%%%%%%%%%%%%%%%%%%%%%%%%%%%%%%%%%%%%%

\section{Proofs of pseudo-law of large numbers}\label{proof:pseudo_SLLN}

For the proof of Proposition \ref{prop:pseudo_SLLN}, we begin by proving a few useful results.
\begin{lemma}\label{lem:wi_abs_logwi}
Suppose $\xi_0\neq 0$, and $Y\sim P_{\btheta_0}$. Denote $w(\btheta_0)=\xi_0(Y-\beta_0)/tau_0$. Then for any constant $\alpha$ such that $\alpha\xi_0+1>0$ and a positive integer $b$, we have
\begin{equation}\label{eqn:wi_abs_logwi}
    E_{\btheta_0}\left[w^{-\alpha}(\btheta_0)|\log w(\btheta_0)|^b\right] < |\xi_0|^b\Gamma^{(b)}(\alpha\xi_0+1)+\frac{3|\xi_0|^b}{(\alpha\xi_0+1)^{b+1}}\Gamma(b+1).
\end{equation}
\end{lemma}
\begin{proof}
It is easy to calculate
\begin{equation*}
    \begin{split}
         E_{\btheta_0}&\left[w^{-\alpha}(\btheta_0)|\log w(\btheta_0)|^b\right]
       %&=\int_{0}^{\infty}\frac{1}{\xi_0}z^{-\frac{1}{\xi_0}-1-a}\cdot\exp(-z^{-1/\xi_0})|\log z|^bdz\qquad\qquad \textcolor{black}{(z=w(\boldsymbol{\theta}_0))}\\
       =|\xi_0|^b\int_{0}^{\infty} s^{\alpha\xi_0}\exp(-s)|\log s|^bds\\
       =&|\xi_0|^b\int_{0}^{\infty} s^{\alpha\xi_0}\exp(-s)\log^b sds+\left((-1)^b-1\right)|\xi_0|^b\int_{0}^{1} s^{\alpha\xi_0}\exp(-s)\log^b sds\\
       =&|\xi_0|^b\Gamma^{(b)}(\alpha\xi_0+1)+\left((-1)^b-1\right)|\xi_0|^b\int_{0}^{1} s^{\alpha\xi_0}\exp(-s)\log^b sds.
    \end{split}
\end{equation*}
Furthermore,
\begin{equation*}
\begin{split}
    \left|\left((-1)^b-1\right)\xi_0^b\right.&\left.\int_{0}^{1} s^{\alpha\xi_0}\exp(-s)\log^b sds\right|\leq 2|\xi_0|^b \int_{0}^{1} s^{\alpha\xi_0}\exp(-s)(-\log s)^bds\\
    &< 3|\xi_0|^b \int_{0}^{1} s^{\alpha\xi_0}(-\log s)^bds=\frac{3|\xi_0|^b}{(\alpha\xi_0+1)^{b+1}}\Gamma(b+1),
    % &= 2\xi_0^b \int_{0}^{\infty} e^{-(\alpha\xi_0+1)t}t^bdt \quad \textcolor{black}{(t=-\log s)}\\
    % &=2|\xi_0|^b(\alpha\xi_0+1)^{-b-1}\int_{0}^{\infty}v^be^{-v}dv \qquad \textcolor{black}{(v=(\alpha\xi_0+1)t))}\\
\end{split}
\end{equation*}
which proves \eqref{eqn:wi_abs_logwi}.
\end{proof}

\begin{lemma}\label{jacobian}
If $\hat{\btheta}_n\in\Omega_n$ solves the likelihood equations of  $L_n(\btheta)$, $\triangledown L_n(\btheta)=0$, then
\begin{equation*}
    \begin{split}
        \sumN (1+\hat{\xi}_n)w_i^{-1}(\hat{\btheta}_n)&=\sumN w_i^{-1-1/\hat{\xi}_n}(\hat{\btheta}_n),\\
        \sumN w_i^{-1/\hat{\xi}_n}(\hat{\btheta}_n)&=n,\\
        \sumN \log w_i(\hat{\btheta}_n)-n\hat{\xi}_n&= \sumN w_i^{-1/\hat{\xi}_n}(\hat{\btheta}_n)\log w_i(\hat{\btheta}_n).
    \end{split}
\end{equation*}
\end{lemma}
\begin{remark}
This lemma is an immediate result of the definition of local MLE. It dovetails with Proposition \ref{prop:pseudo_SLLN}, which says that $\frac{1}{n}\sumN w_i^{-1/\hat{\xi}_n}(\hat{\btheta}_n) \asconv 1$. However, for the power of $1/\hat{\xi}_n$, $\frac{1}{n}\sumN w_i^{-1/\hat{\xi}_n}(\hat{\btheta}_n)\equiv 1$. Similarly for the other two equations, the asymptotic relations turned out to be always true.
\end{remark}

To prove Proposition \ref{prop:pseudo_SLLN}, let us first only change the power with functions of the local MLE before we replace the $\btheta_0$ entirely with $\hat{\btheta}_n$.
\begin{lemma}\label{mle_onpower}
Suppose $Y_1, \ldots, Y_n\iidY P_{\btheta_0}$ where $\xi_0\neq 0$, and $\hat{\btheta}_n$ is the local MLE of $L_n(\btheta)$ that is strongly consistent. Then for constants $k$ and $a$ such that $k\xi_0+a+1>0$,
\begin{equation}\label{eqn:mle_onpower}
    \frac{1}{n}\sumN w_i^{-k-\frac{a}{\hat{\xi}_n}}(\btheta_0)\log^b w_i(\btheta_0) \asconv   (-\xi_0)^b\Gamma^{(b)}(k\xi_0+a+1),
\end{equation}
where $b$ is a non-negative integer.
\end{lemma}
\begin{proof}
Firstly, we prove that 
\begin{equation}\label{a_power_x}
    1+x\log t\leq t^x\leq 1+x\log t + (t^{\eta}+t^{-\eta})(\log t)^2x^2
\end{equation} 
for $t>0$, and $|x|<\eta$, where $\eta>0$. The first inequality holds immediately on account of $1+x\leq e^x, \forall x\in\mathbb{R}$. For the second inequality, we denote $f(x)=t^x- 1-x\log t - (t^{\eta}+t^{-\eta})(\log t)^2x^2$. Since $|x|<\eta$, it's obvious that $f''(x)=t^x(\log t)^2-2(t^{\eta}+t^{-\eta})(\log t)^2<0$, and $f$ is a strictly concave function. The maximum of $f$ in $(-\eta,\eta)$ is $x=0$ since $f'(0)=0$. Therefore, $f(x)\leq f(0)=0$.

Secondly, we examine the difference between $w_i^{-k-\frac{a}{\hat{\xi}_n}}(\btheta_0)$ and $w_i^{-k-\frac{a}{\xi_0}}(\btheta_0)$. 

Fix $\eta>0$ such that $|\eta\xi_0|<k\xi_0+a+1$. Since $\frac{1}{\hat{\xi}_n}\asconv \frac{1}{\xi_0}$, there almost surely exists $N$ such that $\left|\frac{a}{\hat{\xi}_n}-\frac{a}{\xi_0}\right|<\eta$ for all $n>N$. Apply \eqref{a_power_x} to get
\begin{small}
\begin{equation*}
    \left|w_i^{\frac{a}{\xi_0}-\frac{a}{\hat{\xi}_n}}(\btheta_0)-1-\left(\frac{a}{\xi_0}-\frac{a}{\hat{\xi}_n}\right)\log w_i(\btheta_0)\right|\leq [w_i^{\eta}(\btheta_0)+w_i^{-\eta}(\btheta_0)]\log^2 w_i(\btheta_0)\left(\frac{a}{\xi_0}-\frac{a}{\hat{\xi}_n}\right)^2,
\end{equation*}
\end{small}
in which $w_i(\btheta_0)>0$ for $i=1,\ldots,n$. Multiplying both sides by $w_i^{-k-\frac{a}{\xi_0}}(\btheta_0)\left|\log w_i(\btheta_0)\right|^b$,
\begin{small}
\begin{equation*}
\begin{split}
     \Big|w_i^{-k-\frac{a}{\hat{\xi}_n}}(\btheta_0)\log^b w_i(\btheta_0)&-w_i^{-k-\frac{a}{\xi_0}}(\btheta_0)\log^b w_i(\btheta_0)-\left(\frac{a}{\xi_0}-\frac{a}{\hat{\xi}_n}\right)w_i^{-k-\frac{a}{\xi_0}}(\btheta_0)\log^{b+1} w_i(\btheta_0)\Big|\\
     &\leq [w_i^{-k-\frac{a}{\xi_0}+\eta}(\btheta_0)+w_i^{-k-\frac{a}{\xi_0}-\eta}(\btheta_0)]\left|\log w_i(\btheta_0)\right|^{b+2}\left(\frac{a}{\xi_0}-\frac{a}{\hat{\xi}_n}\right)^2.
\end{split}
\end{equation*}
\end{small}

Summing over $i$, we obtain
\begin{equation}\label{maclaurin}
\begin{split}
     \frac{1}{n}\sumN w_i^{-k-\frac{a}{\hat{\xi}_n}}(\btheta_0)&\log ^bw_i(\btheta_0)=\frac{1}{n}\sumN w_i^{-k-\frac{a}{\xi_0}}(\btheta_0)\log^b w_i(\btheta_0)+\\
     &\left(\frac{a}{\xi_0}-\frac{a}{\hat{\xi}_n}\right)\cdot \frac{1}{n}\sumN w_i^{-k-\frac{a}{\xi_0}}(\btheta_0)\log^{b+1} w_i(\btheta_0)+R(\hat{\btheta}_n),
\end{split}
\end{equation}
where
\begin{small}
\begin{equation*}
\begin{split}
    |R(\hat{\btheta}_n)|\leq\left(\frac{a}{\xi_0}-\frac{a}{\hat{\xi}_n}\right)^2\times\left[\frac{1}{n}\sumN w_i^{-k-\frac{a}{\xi_0}+\eta}(\btheta_0)|\log w_i(\btheta_0)|^{b+2}+\right.\\
    \hspace*{1cm}\left.\frac{1}{n}\sumN w_i^{-k-\frac{a}{\xi_0}-\eta}(\btheta_0)|\log w_i(\btheta_0)|^{b+2}\right].
\end{split}
\end{equation*}
\end{small}
Lemma \ref{lem:wi_abs_logwi} ensures that $E_{\btheta_0}\left\{ w(\btheta_0)^{-\alpha}|\log w(\btheta_0)|^{b}\right\}<\infty$ for any non-negative integer $b$ and $\alpha$ such that $\alpha\xi_0+1>0$. Since $|\eta\xi_0|<k\xi_0+a+1$, we can assure $(k-\eta)\xi_0+a+1>0$ and $(k+\eta)\xi_0+a+1>0$. By law of large numbers, $\frac{1}{n}\sumN w_i^{-k-\frac{a}{\xi_0}+\eta}(\btheta_0)|\log w_i(\btheta_0)|^{b+2}$ and $\frac{1}{n}\sumN w_i^{-k-\frac{a}{\xi_0}-\eta}(\btheta_0)|\log w_i(\btheta_0)|^{b+2}$ are bounded almost surely. Meanwhile, $\left(\frac{a}{\xi_0}-\frac{a}{\hat{\xi}_n}\right)^2\asconv  0$, and hence $R(\hat{\btheta}_n)\asconv  0$.

Also, $\left(\frac{a}{\xi_0}-\frac{a}{\hat{\xi}_n}\right)\cdot \frac{1}{n}\sumN w_i^{-k-\frac{a}{\xi_0}}(\btheta_0)\log^{b+1} w_i(\btheta_0)\asconv  0$, and 
$$\frac{1}{n}\sumN w_i^{-k-\frac{a}{\xi_0}}(\btheta_0)\log^b w_i(\btheta_0)\asconv (-\xi_0)^b\Gamma^{(b)}(k\xi_0+a+1).$$
Consequently we can conclude from \eqref{maclaurin} that \eqref{eqn:mle_onpower} is true.
\end{proof}

\subsection{Proof of Proposition \ref{prop:pseudo_SLLN}}
% \BASadd{Should this be a subsection?}
Next we replace the $\btheta_0$ with $\hat{\btheta}_n$ in $w_i(\btheta_0)$ to prove Proposition \ref{prop:pseudo_SLLN}. 

\begin{proof}
\textbf{(A) Case $\xi_0>0$.} Firstly, we notice that for $\alpha<0$ and $|x-1|<\eta$,
\begin{equation}\label{abs_distance_nb}
   \log^2 x\leq \frac{2}{(1-\eta)^2}(x-1)^2,
\end{equation}
and 
\begin{equation}\label{x_power_a}
    1+t(x-1)\leq x^t\leq 1+t(x-1)+\frac{t(t-1)}{2}(1-\eta)^{t-2}(x-1)^2,
\end{equation} 
where $\eta\in (0,1)$ is a small fixed number. Concavity argument similar to that yielding \eqref{a_power_x} will prove \eqref{abs_distance_nb} and \eqref{x_power_a}.

\vspace{0.5cm}
Secondly, we prove
\begin{equation}\label{log_theta0_theta}
    \frac{1}{n}w_i^{-k-\frac{a}{\hat{\xi}_n}}(\btheta_0)\log^b w_i(\hat{\btheta}_n)\asconv (-\xi_0)^b\Gamma^{(b)}(k\xi_0+a+1)
\end{equation}
for $b\geq 1$ (the case $b=0$ is covered by Lemma \ref{mle_onpower}).

Since $\frac{1}{\hat{\xi}_n}\asconv \frac{1}{\xi_0}$ and by Corollary \ref{rate_comp} $\frac{\hat{\beta}_n-\beta_0}{Y_{(1)}-\beta_0}\asconv 0$ , there almost surely exists $N$ such that for all $n>N$,
\begin{equation*}
    1/{\hat{\xi}_n}>0,\text{ and } \left|\frac{\xi_0(\beta_0-\hat{\beta}_n)}{\tau_0w_i(\btheta_0)}\right|=\left|\frac{\beta_0-\hat{\beta}_n}{Y_i-\beta_0}\right|\stackrel{Y_i>\beta_0}{\leq} \left|\frac{\beta_0-\hat{\beta}_n}{Y_{(1)}-\beta_0}\right|<\eta.
\end{equation*}
We expand the terms using multinomial theorem
\begin{footnotesize}
\begin{equation*}\label{mix_theta0_theta}
\begin{split}
   \log^b w_i(\hat{\btheta}_n)&=\left[\log\left(\frac{\hat{\xi}_n\tau_0}{\hat{\tau}_n\xi_0}\right)+\log w_i(\btheta_0)+\log\left(1+\frac{\xi_0(\beta_0-\hat{\beta}_n)}{\tau_0w_i(\btheta_0)}\right)\right]^b=\log^b\left(\frac{\hat{\xi}_n\tau_0}{\hat{\tau}_n\xi_0}\right)+\log^b w_i(\btheta_0)\\
    &+\sum_{\substack{b_1+b_2+b_3=b\\b_3\geq 1}} \binom{b }{b_1, b_2, b_3} \log^{b_1}\left(\frac{\hat{\xi}_n\tau_0}{\hat{\tau}_n\xi_0}\right) \cdot\log^{b_2} w_i(\btheta_0)\cdot\log^{b_3}\left(1+\frac{\xi_0(\beta_0-\hat{\beta}_n)}{\tau_0w_i(\btheta_0)}\right).
\end{split}
\end{equation*}
\end{footnotesize}
Thus,
\begin{footnotesize}
\begin{equation*}
\begin{split}
   \Big|\log^b w_i(\hat{\btheta}_n)&-\log^b\left(\frac{\hat{\xi}_n\tau_0}{\hat{\tau}_n\xi_0}\right)-\log^b w_i(\btheta_0)\Big|\\
    &\leq\sum_{\substack{b_1+b_2+b_3=b\\b_3\geq 1}} \binom{b }{b_1, b_2, b_3} \left|\log\left(\frac{\hat{\xi}_n\tau_0}{\hat{\tau}_n\xi_0}\right)\right|^{b_1} \cdot|\log w_i(\btheta_0)|^{b_2}\cdot\left|\log\left(1+\frac{\xi_0(\beta_0-\hat{\beta}_n)}{\tau_0w_i(\btheta_0)}\right)\right|^{b_3}\\
    &\leq\sum_{\substack{b_1+b_2+b_3=b\\b_3\geq 1}} \binom{b }{b_1, b_2, b_3} \left|\log\left(\frac{\hat{\xi}_n\tau_0}{\hat{\tau}_n\xi_0}\right)\right|^{b_1} \cdot|\log w_i(\btheta_0)|^{b_2}\cdot\frac{2^{b_3}}{(1-\eta)^{b_3}}\left|\frac{\xi_0(\beta_0-\hat{\beta}_n)}{\tau_0w_i(\btheta_0)}\right|^{b_3}.
\end{split}
\end{equation*}
\end{footnotesize}
The last inequality holds due to \eqref{abs_distance_nb} with $x=1+\frac{\xi_0(\beta_0-\hat{\beta}_n)}{\tau_0w_i(\btheta_0)}$. Multiplying both sides of the inequality by $w^{-k-\frac{a}{\hat{\xi}_n}}_i(\btheta_0)$ and summing over $i$,
\begin{small}
\begin{equation*}
    \begin{split}
        \frac{1}{n}\sumN w^{-k-\frac{a}{\hat{\xi}_n}}_i(\btheta_0)\log^b w_i(\hat{\btheta}_n)=&\frac{1}{n}\sumN w^{-k-\frac{a}{\hat{\xi}_n}}_i(\btheta_0)\log^b w_i(\btheta_0)+\\
        &\frac{1}{n}\log^b\left(\frac{\hat{\xi}_n\tau_0}{\hat{\tau}_n\xi_0}\right)\cdot\sum_{i=1}^nw^{-k-\frac{a}{\hat{\xi}_n}}_i(\btheta_0)+R_1(\hat{\btheta}_n),
    \end{split}
\end{equation*}
\end{small}
where
\begin{small}
\begin{equation*}
\begin{split}
    |R_1(\hat{\btheta}_n)|\leq \sum_{\substack{b_1+b_2+b_3=b\\b_3\geq 1}} \binom{b }{b_1, b_2, b_3}&\frac{2^{b_3}}{(1-\eta)^{b_3}}\left|\frac{\xi_0(\beta_0-\hat{\beta}_n)}{\tau_0}\right|^{b_3} \left|\log\left(\frac{\hat{\xi}_n\tau_0}{\hat{\tau}_n\xi_0}\right)\right|^{b_1} \times\\
    &\frac{1}{n}\sumN w^{-k-\frac{a}{\hat{\xi}_n}-b_3}_i(\btheta_0)|\log w_i(\btheta_0)|^{b_2}.
\end{split}
\end{equation*}
\end{small}
Lemma \ref{mle_onpower} ensures the averages in the right side of last inequality are bounded almost surely. Since $b_3\geq 1$ and $\beta_0-\hat{\beta}_n\asconv 0$, we know $R_1(\hat{\btheta}_n)\asconv  0$.

In the meantime, Lemma \ref{mle_onpower} entails $\frac{1}{n}\log^b\left(\frac{\hat{\xi}_n\tau_0}{\hat{\tau}_n\xi_0}\right)\cdot\sum_{i=1}^nw^{-k-\frac{a}{\hat{\xi}_n}}_i(\btheta_0)\stackrel{a.s}{\rightarrow}0$, and thus \eqref{log_theta0_theta} is proved.

\vspace{0.5cm}
Lastly, we prove $\frac{1}{n}w_i^{-k-\frac{a}{\hat{\xi}_n}}(\hat{\btheta}_n)\log^b w_i(\hat{\btheta}_n)\asconv (-\xi_0)^b\Gamma^{(b)}(k\xi_0+a+1)$.

Apply \eqref{x_power_a} with $x=1+\frac{\xi_0(\beta_0-\hat{\beta}_n)}{\tau_0w_i(\btheta_0)}$, and
\begin{small}
\begin{equation*}
\begin{split}
    \left|\left(1+\frac{\xi_0(\beta_0-\hat{\beta}_n)}{\tau_0w_i(\btheta_0)}\right)^{-k-\frac{a}{\hat{\xi}_n}}-1\right.&\left.-\left(k+\frac{a}{\hat{\xi}_n}\right)\cdot \frac{\xi_0(\beta_0-\hat{\beta}_n)}{\tau_0w_i(\btheta_0)}\right|\leq\\
    &\frac{\left(k+\frac{a}{\hat{\xi}_n}\right)\left(k+\frac{a}{\hat{\xi}_n}+1\right)(1-\eta)^{-k-\frac{a}{\hat{\xi}_n}-2}}{2}\cdot\frac{\xi_0^2(\beta_0-\hat{\beta}_n)^2}{\tau_0^2w_i^2(\btheta_0)}.
\end{split}
\end{equation*}
\end{small}
Multiplying every term by $w^{-k-\frac{a}{\hat{\xi}_n}}_i(\btheta_0)|\log(\hat{\btheta}_n)|^b$ and summing over $i$,
\begin{small}
\begin{equation}\label{maclaurin_self}
    \begin{split}
        \frac{1}{n}\sumN  \left[w_i(\btheta_0)+\vphantom{\frac{\xi_0(\beta_0-\hat{\beta}_n)}{\tau_0}}\right.&\left. \frac{\xi_0(\beta_0-\hat{\beta}_n)}{\tau_0}\right]^{-k-\frac{a}{\hat{\xi}_n}}\log^b(\hat{\btheta}_n)=\frac{1}{n}\sumN w_i^{-k-\frac{a}{\hat{\xi}_n}}(\btheta_0)\log^b(\hat{\btheta}_n)-\\
        &\left(k+\frac{1}{\hat{\xi}_n}\right)\frac{\xi_0(\beta_0-\hat{\beta}_n)}{\tau_0}\cdot\frac{1}{n}\sumN w_i^{-k-1-\frac{a}{\hat{\xi}_n}}(\btheta_0)\log^b(\hat{\btheta}_n)+R_2(\hat{\btheta}_n),
    \end{split}
\end{equation}
\end{small}
in which the left-hand side is also equal to $\left(\frac{\hat{\xi}_n\tau_0}{\hat{\tau}_n\xi_0}\right)^{k+\frac{a}{\hat{\xi}_n}}\cdot\frac{1}{n}\sumN w_i^{-k-\frac{a}{\hat{\xi}_n}}(\hat{\btheta}_n)\log^b(\hat{\btheta}_n)$, and
\begin{small}
\begin{equation*}
    \begin{split}
        |R_2(\hat{\btheta}_n)|\leq & \frac{\xi^2_0\left(k+\frac{1}{\hat{\xi}_n}\right)\left(k+1+\frac{1}{\hat{\xi}_n}\right)(1-\eta)^{-k-\frac{a}{\hat{\xi}_n}-2}(\beta_0-\hat{\beta}_n)^2}{2\tau^2_0}\times\\
        &\hspace{1cm}\frac{1}{n}\sumN w_i^{-\left(k+2+\frac{1}{\hat{\xi}_n}\right)}(\btheta_0)|\log(\hat{\btheta}_n)|^b.
    \end{split}
\end{equation*}
\end{small}

Using arguments similar to that leading to $R_1(\hat{\btheta}_n)\asconv 0$ yields $R_2(\hat{\btheta}_n)\asconv 0$. Simplifying terms in \eqref{maclaurin_self}, we conclude
\begin{equation*}
    \frac{1}{n}\sumN w_i^{-k-\frac{a}{\hat{\xi}_n}}(\hat{\btheta}_n)\log^b(\hat{\btheta}_n)  \sim \left(\frac{\hat{\xi}_n\tau_0}{\hat{\tau}_n\xi_0}\right)^{-k-\frac{a}{\hat{\xi}_n}}\cdot \frac{1}{n}\sumN w_i^{-\left(k+\frac{1}{\hat{\xi}_n}\right)}(\btheta_0)\log^b(\hat{\btheta}_n)\hspace{12pt} a.s.,
\end{equation*}
the right-hand side of which converges almost surely to $(-\xi_0)^b\Gamma^{(b)}(k\xi_0+a+1)$ by \eqref{log_theta0_theta}. This ends the proof of Proposition \ref{prop:pseudo_SLLN} for $\xi_0>0$.

\textbf{(B) Case $\xi_0<0$.} Since $k\xi_0+a+1>0$, $k+\frac{a}{\xi_0}<-\frac{1}{\xi_0}$, which means $k+\frac{a}{\xi_0}$ can be positive. Therefore we need to approximate $x^t$ for $t>0$ in $(1-\eta,1+\eta)$ by the polynomial terms of $(x-1)$---that is, to establish different versions of \eqref{x_power_a}. When $1\leq t <2$, \eqref{x_power_a} holds. When $0<t<1$,
\begin{equation*}
    1+t(x-1)+\frac{t(t-1)}{2}(1-\eta)^{t-2}(x-1)^2\leq x^t\leq 1+t(x-1),
\end{equation*} 
and when $t\geq 2$,
\begin{equation*}
    1+t(x-1)\leq x^t\leq 1+t(x-1)+\frac{t(t-1)}{2}(1+\eta)^{t-2}(x-1)^2.
\end{equation*}
On the other hand, we know by Corollary \ref{rate_comp} $\frac{\hat{\beta}_n-\beta_0}{\beta_0-Y_{(n)}}\asconv 0$. Hence there almost surely exists $N$ such that for all $n>N$,
\begin{equation*}
    1/{\hat{\xi}_n}<0,\text{ and } \left|\frac{\xi_0(\beta_0-\hat{\beta}_n)}{\tau_0w_i(\btheta_0)}\right|=\left|\frac{\beta_0-\hat{\beta}_n}{Y_i-\beta_0}\right|\stackrel{\beta_0>Y_i}{\leq} \frac{|\beta_0-\hat{\beta}_n|}{\beta_0-Y_{(n)}}<\eta,
\end{equation*}
where $\eta\in (0,1)$ is fixed.

To prove Proposition \ref{prop:pseudo_SLLN} for $\xi_0<0$, apply \eqref{abs_distance_nb} and the varieties of \eqref{x_power_a} with $x=1+\frac{\xi_0(\beta_0-\hat{\beta}_n)}{\tau_0w_i(\btheta_0)}$, and follow the same steps as proving the case $\xi_0>0$.
% Applying \eqref{abs_distance_nb} and the varieties of \eqref{x_power_a} with $x=1+\frac{\xi_0(\beta_0-\hat{\beta}_n)}{\tau_0w_i(\btheta_0)}$, we will prove Proposition \ref{prop:pseudo_SLLN} for $\xi_0<0$ while following a proof similar as the case $\xi_0>0$.
\end{proof}

\subsection{Proof of Proposition \ref{prop:uniform_consistency}}
\label{proof:uniform_consistency}
\begin{proof}
To prove $\sup_{\alpha\in I}\left|\Phi_n(\alpha)-\Phi(\alpha)\right|\asconv 0$, it suffices to prove that for any fixed $\epsilon>0$, there almost surely exists $N>0$ such that for all $n>N$,
\begin{equation*}
    \sup_{\alpha\in I}\left|\Phi_n(\alpha)-\Phi(\alpha)\right|<\epsilon.
\end{equation*}

We first consider the case where $\xi_0>0$. Since $\Phi(\alpha)=(-\xi_0)^b\Gamma^{(b)}(\alpha\xi_0+1)$ is uniformly continuous in $I=[m,M]$, we can find $\eta_1>0$ such that for any $\alpha_1,\alpha_2\in I$ with $|\alpha_1-\alpha_2|<\eta_1$, 
\begin{equation}\label{eqn:uni_cont}
    |\Phi(\alpha_1)-\Phi(\alpha_2)|<\frac{\epsilon}{3}.
\end{equation}
Define the interval $\tilde{I}=\left(\frac{m-1/\xi_0}{2},\frac{2M+m+1/\xi_0}{2}\right)$ and
\begin{equation*}
\begin{split}
    \Pi_2=&\xi_0^{b+1}\sup_{\alpha\in I}\Gamma^{(b+1)}(\alpha\xi_0+1)+\frac{3\xi_0^{b+1}}{(m\xi_0+1)^{b+2}}\Gamma(b+2),\\
    \Pi_3=&\xi_0^{b+2}\sup_{\alpha\in \tilde{I}}\Gamma^{(b+2)}(\alpha\xi_0+1)+\frac{3\xi_0^{b+2}}{[(m\xi_0+1)/2]^{b+3}}\Gamma(b+3).
\end{split}
\end{equation*}
Fix $\eta=\min\left\{\eta_1,\frac{\epsilon}{9\Pi_2},\sqrt{\frac{\epsilon}{9\Pi_3}},\frac{m+1/\xi_0}{2}\right\}$. Denote $B(\alpha,\eta)=(\alpha-\eta,\alpha+\eta)$. Note the set $I$ is compact and is covered by the intervals  $\{B(\alpha,\eta):\alpha\in I\}$. Let $B_i=B(\alpha_i,\eta)$, $ 1\leq i\leq p$, be a finite cover. Then
\begin{equation}\label{eqn:three_parts_origin}
    \begin{split}
    \sup_{\alpha\in I}|\Phi_n(\alpha)-&\Phi(\alpha)|\leq \max_{1\leq i\leq p}\sup_{\alpha\in B_i}\left|\Phi_n(\alpha)-\Phi(\alpha)\right|\\
    \leq& \max_{1\leq i\leq p}\sup_{\alpha\in B_i}\left|\Phi_n(\alpha)-\Phi_n(\alpha_i)\right|+\max_{1\leq i\leq p}\sup_{\alpha\in B_i}\left|\Phi(\alpha)-\Phi(\alpha_i)\right|+\\
    & \max_{1\leq i\leq p}\left|\Phi_n(\alpha_i)-\Phi(\alpha_i)\right|.
    \end{split}
\end{equation}

From \eqref{eqn:uni_cont}, we know
\begin{equation}\label{eqn:part1}
    \max_{1\leq i\leq p}\sup_{\alpha\in B_i}\left|\Phi(\alpha)-\Phi(\alpha_i)\right|<\frac{\epsilon}{3}.
\end{equation}
Since $p$ is a finite number, the pointwise consistency in Proposition \ref{prop:pseudo_SLLN} will ensure that there almost surely exists $N_1>0$ such that for all $n>N_1$,
\begin{equation}\label{eqn:part2}
    \max_{1\leq i\leq p}\left|\Phi_n(\alpha_i)-\Phi(\alpha_i)\right|<\frac{\epsilon}{3}.
\end{equation}

Now we examine the first term on the right-hand side of \eqref{eqn:three_parts_origin}. Since $\eta\leq \frac{m+1/\xi_0}{2}$, we have 
\begin{equation}\label{eqn:range_alpha}
    \begin{split}
    \alpha_i-\eta&\geq \frac{m-1/\xi_0}{2},\;\; \xi_0(\alpha_i-\eta)+1\geq \frac{m\xi_0+1}{2}>0,\\
    \alpha_i+\eta&\leq \frac{2M+m+1/\xi_0}{2},
    \end{split}
\end{equation}
for all $1\leq i \leq p$. From \eqref{a_power_x}, we deduce for $\alpha\in B_i$ that 
\begin{equation*}
\begin{split}
    \left|w_i^{-\alpha}(\hat{\btheta}_n)- w_i^{-\alpha_i}(\hat{\btheta}_n)\right|\leq& |\alpha-\alpha_i|\cdot |w_i^{-\alpha_i}(\hat{\btheta}_n)\log w_i(\hat{\btheta}_n)|+\\
    &(\alpha-\alpha_i)^2\cdot (w_i^{-\alpha_i+\eta}(\hat{\btheta}_n)+w_i^{-\alpha_i-\eta}(\hat{\btheta}_n))\log^2 w_i(\hat{\btheta}_n).\\
\end{split}
\end{equation*}
Therefore,
\begin{equation}\label{eqn:three_parts}
\begin{split}
    \left|\Phi_n(\alpha)-\Phi_n(\alpha_i)\right|=&\left|\frac{1}{n}\sumN \{w_i^{-\alpha}(\hat{\btheta}_n)- w_i^{-\alpha_i}(\hat{\btheta}_n)\}\log^b w_i(\hat{\btheta}_n)\right|\\
    %\leq & \frac{1}{n}\sumN \left|w_i^{-\alpha}(\hat{\btheta}_n)- w_i^{-\alpha_i}(\hat{\btheta}_n)\right|\cdot \left|\log^b w_i(\hat{\btheta}_n)\right|\\
    \leq & \frac{|\alpha-\alpha_i|}{n}\sumN w_i^{-\alpha_i}(\hat{\btheta}_n)\left|\log w_i(\hat{\btheta}_n)\right|^{b+1}+\\
    & \frac{(\alpha-\alpha_i)^2}{n}\sumN w_i^{-\alpha_i+\eta}(\hat{\btheta}_n)\left|\log w_i(\hat{\btheta}_n)\right|^{b+2}+\\
    &\frac{(\alpha-\alpha_i)^2}{n}\sumN w_i^{-\alpha_i-\eta}(\hat{\btheta}_n)\left|\log w_i(\hat{\btheta}_n)\right|^{b+2}.
\end{split}
\end{equation}
Recall that Lemma \ref{lem:wi_abs_logwi} guarantees that there almost surely exists $N_2>0$ such that for all $n>N_2$ and $1\leq i\leq p$,
\begin{equation*}
    \frac{1}{n}\sumN w_i^{-\alpha_i}(\hat{\btheta}_n)\left|\log w_i(\hat{\btheta}_n)\right|^{b+1}<\xi_0^{b+1}\Gamma^{(b+1)}(\alpha_i\xi_0+1)+\frac{3\xi_0^{b+1}\Gamma(b+2)}{(\alpha_i\xi_0+1)^{b+2}}<\Pi_2.
\end{equation*}
Similarly, there almost surely exists $N_3>0$ such that for all $n>N_3$ and $1\leq i\leq p$,
\begin{small}
\begin{equation*}
    \begin{split}
        \frac{1}{n}\sumN w_i^{-\alpha_i+\eta}(\hat{\btheta}_n)\left|\log w_i(\hat{\btheta}_n)\right|^{b+2}&<\xi_0^{b+2}\Gamma^{(b+2)}(\xi_0(\alpha_i-\eta)+1)+\frac{3\xi_0^{b+2}\Gamma(b+3)}{[\xi_0(\alpha_i-\eta)+1]^{b+3}}\stackrel{\eqref{eqn:range_alpha}}{<}\Pi_3,\\
        \frac{1}{n}\sumN w_i^{-\alpha_i-\eta}(\hat{\btheta}_n)\left|\log w_i(\hat{\btheta}_n)\right|^{b+2}&<\xi_0^{b+2}\Gamma^{(b+2)}(\xi_0(\alpha_i+\eta)+1)+\frac{3\xi_0^{b+2}\Gamma(b+3)}{[\xi_0(\alpha_i+\eta)+1]^{b+3}}\stackrel{\eqref{eqn:range_alpha}}{<}\Pi_3.
    \end{split}
\end{equation*}
\end{small}
Therefore, the bound in \eqref{eqn:three_parts} can be relaxed as
\begin{equation*}
    \left|\Phi_n(\alpha)-\Phi_n(\alpha_i)\right|
    < |\alpha-\alpha_i|\Pi_2+2(\alpha-\alpha_i)^2\Pi_3,
\end{equation*}
and thus for all $1\leq i \leq p$,
\begin{equation}\label{eqn:part3}
    \sup_{\alpha\in B_i}\left|\Phi_n(\alpha)-\Phi_n(\alpha_i)\right|
    < \eta\Pi_2+2\eta^2\Pi_3<\frac{\epsilon}{9}+\frac{2\epsilon}{9}=\frac{\epsilon}{3}.
\end{equation}
The last inequality stems from $\eta<\frac{\epsilon}{9\Pi_2}$ and $\eta<\sqrt{\frac{\epsilon}{9\Pi_3}}$.

To sum up, we plug \eqref{eqn:part1}, \eqref{eqn:part2} and \eqref{eqn:part3} back into \eqref{eqn:three_parts_origin}. We conclude that there almost surely exists $N=\max\{N_1,N_2,N_3\}$ such that for all $n>N$,
\begin{equation*}
    \sup_{\alpha\in I}|\Phi_n(\alpha)-\Phi(\alpha)|<\epsilon,
\end{equation*}
which completes the proof of this proposition for $\xi_0>0$. For the case $\xi_0<0$, the proof is analogous.
\end{proof}

%%%%%%%%%%%%%%%%%%%%%%%%%%%%%%%%%%%%%%%%%%%%%%%%%%
%%                 Appendix D                   %%
%%%%%%%%%%%%%%%%%%%%%%%%%%%%%%%%%%%%%%%%%%%%%%%%%%

\section{Proofs concerning the local concavity}\label{proof:loc_concavity}
%\BASadd{The structure of this section seems strange.  Those bulletted points are shown in lemmas below, right?  So maybe it would help to re-phrase them a bit.}
\subsection{Proof of Proposition \ref{prop:hessian}}
\begin{proof}
Let $B_r(\btheta)=\{\btheta'\in\Theta: ||\btheta'-\btheta||_{\infty}<r\}$, where $||\cdot||_{\infty}$ is the maximum norm. For the conciseness of the appendices, we only provide proof for the case when the shape parameter $\xi_0<0$. The proof of the case $\xi_0>0$ is much easier because $k+\frac{1}{\xi_0}$ is always positive when $k$ is a positive integer.

If $\xi_0\in (-1/2,0)$, the proof requires more careful treatment. We first find an integer $K_0\geq 2$  such that $-1/K_0<\xi_0\leq -1/(K_0+1)$. We further select $\epsilon_0>0$ to tighten the left bound such that $-1/(K_0+\epsilon_0)<\xi_0\leq -1/(K_0+1)$. Find $r>0$ small enough such that 
\begin{equation}\label{range_xi_supp}
    -\frac{1}{K_0+\epsilon_0}<\xi_0-r<\xi_0+r<-\frac{1}{K_0+2},
\end{equation}
which leads to $\xi_0/(\xi_0-r)>-\xi_0(K_0+\epsilon_0)$. Therefore, for any integer $k\in\{0,1,\ldots,K_0+1\}$,
\begin{equation}\label{range_xi_supp1}
\begin{split}
    k\xi_0+\frac{\xi_0}{\xi_0-r}+1&\geq (K_0+1)\xi_0+\frac{\xi_0}{\xi_0-r}+1\\
    &>(K_0+1)\xi_0-\xi_0(K_0+\epsilon_0)+1=\xi_0(1-\epsilon_0)+1>0,
\end{split}
\end{equation}
where the last inequality holds because $\xi_0\in (-1,0)$.

Meanwhile, for any $\xi\in (\xi_0-r,\xi_0+r)$,
\begin{equation}\label{range_xi_supp2}
    K_0+\epsilon_0<-\frac{1}{\xi_0-r}<-\frac{1}{\xi}<-\frac{1}{\xi_0+r}%<K_0+1+\epsilon_0.
\end{equation}

To approximate the Hessian in $B_r(\btheta)$ and prove Proposition \ref{prop:hessian}, we need to check a few things:
\begin{enumerate}[label=\textbf{(\roman*)},ref=(\roman*)]
    \item\label{cond1} We calculate the Hessian matrix of the log-likelihood function $L_n(\btheta)$, and show that each element of the matrix is a linear combination of a few different types of sums: $\sumN w_i^{-k-\frac{1}{\xi}}(\btheta)$, $\sumN w_i^{-k-\frac{1}{\xi}}(\btheta)\log^b w_i(\btheta)$, where $k$ and $b$ are integers. This is verified in Lemma \ref{hessian_elements}.
    \item\label{cond2} For small $r>0$ that satisfies \eqref{range_xi_supp} and $r<\min\left\{\frac{\tau_0}{3},\frac{4(\xi_0-\tau_0)}{\xi_0\tau_0}\right\}$, we prove that these sums divided by $n$ are uniformly bounded for large sample size $n$ in the neighborhood $B_{r}(\hat{\btheta}_n)$ due to the pseudo-law of large numbers (Proposition \ref{prop:pseudo_SLLN}). The bounds only depend on $\btheta_0$ and $r$. This is established in Lemma \ref{shrink_neighborhood} and \ref{shrink_neighborhood2}.
    \item\label{cond3} Utilizing the convexity of the set $B_{r}(\hat{\btheta}_n)\cap \Omega_n$ and the mean value theorem, we prove the sums are Lipschitz continuous:
    $$\left|\sumN w_i^{-k-\frac{1}{\xi}}(\btheta)\log^b w_i(\btheta)-\sumN w_i^{-k-\frac{1}{\hat{\xi}_n}}(\hat{\btheta}_n)\log^b w_i(\hat{\btheta}_n)\right|<nM||\btheta-\hat{\btheta}_n||_{\infty},$$
    where $M$ is positive constant that is defined by $\btheta_0$ and the radius $r$. This is assured by Proposition \ref{lem:lipschitz_cont}.
\end{enumerate}
With the Lipschitz continuity, we can easily show in the neighborhood $B_{r}(\hat{\btheta}_n)\cap \Omega_n$,
    \begin{equation*}
    \boldsymbol{I}-\boldsymbol{A}_0(r)\leq L''_n(\btheta)\{L''_n(\hat{\btheta}_n)\}^{-1}\leq \boldsymbol{I}+\boldsymbol{A}_0(r),
    \end{equation*}
    where $\boldsymbol{I}$ is the $3\times 3$ identity matrix and $\boldsymbol{A}_0(r)$ is a symmetric positive-semidefinite matrix which only depends on $\btheta_0$ and the radius $r$, and whose largest eigenvalue tends to zero as $r\rightarrow 0$.
\end{proof}

\subsection{Proofs of \ref{cond1} - \ref{cond3}}
In this subsection, we will utilize the classic $c_r$ inequality which states that for any $a,b\geq 0$,
\begin{equation}\label{Cr_ineq}
    (a+b)^r\leq C_r(a^r+b^r),
\end{equation}
where $C_r=1$ if $0<r\leq 1$, and $C_r=2^{r-1}$ if $r>1$.

\begin{lemma}[Hessian matrix]\label{hessian_elements}
Given i.i.d random variables $Y_1, Y_2, \ldots$ with common distribution $P_{\btheta_0}$, the elements in the Hessian matrix, $L''_n(\btheta)$, can each be expressed as a linear combination of the following terms
\begin{equation*}
\begin{split}
    \sumN w_i^{-k}(\btheta),\;\sumN& w_i^{-k-\frac{1}{\xi}}(\btheta),\;\sumN \log w_i(\btheta),\\
    \sumN w_i^{-k'-\frac{1}{\xi}}(\btheta)\log w_i&(\btheta),\;\;\sumN w_i^{-\frac{1}{\xi}}(\btheta)(\log w_i(\btheta))^2,
\end{split}
\end{equation*}
where $k=0,1,2$, $k'=0,1$ and $\xi\neq 0$.
\end{lemma}
{\it Proof. }
Given i.i.d observations $Y_1,\ldots,Y_n$, the elements in the Hessian matrix, $\partial^2 L_n/\partial\btheta\partial\btheta^T$, can be obtained as follows (the detailed calculations are omitted):
\begin{footnotesize}
\begin{equation*}
    \begin{split}
        \evalat[\Big]{\frac{\partial^2 L_n}{\partial \mu^2}}{\btheta}=&\frac{(1+\xi)\xi}{\tau^2}\sumN w_i^{-2}(\btheta)-\frac{(1+\xi)}{\tau^2}\sumN w_i^{-2-\frac{1}{\xi}}(\btheta),\\
        \evalat[\Big]{\frac{\partial^2 L_n}{\partial \mu \partial \xi}}{\btheta}=&-\frac{1}{\xi\tau}\sumN w_i^{-1}(\btheta)+\frac{\xi+1}{\tau\xi^2}\sumN w_i^{-1-\frac{1}{\xi}}(\btheta)+\frac{\xi+1}{\tau\xi}\sumN w_i^{-2}(\btheta)\\
        &-\left(\frac{1}{\tau\xi^2}+\frac{1}{\tau\xi}\right)\sumN w_i^{-2-\frac{1}{\xi}}(\btheta)-\frac{1}{\xi^2\tau}\sumN w_i^{-1-\frac{1}{\xi}}(\btheta)\log w_i(\btheta),\\
        \evalat[\Big]{\frac{\partial^2 L_n}{\partial \mu \partial \tau}}{\btheta}=&-\frac{1}{\tau^2\xi}\sumN w_i^{-1-\frac{1}{\xi}}(\btheta)-\frac{(\xi+1)}{\tau^2}\sumN w_i^{-2}(\btheta)+\frac{\xi+1}{\tau^2\xi}\sumN w_i^{-2-\frac{1}{\xi}}(\btheta),\\
        \evalat[\Big]{\frac{\partial^2 L_n}{\partial \tau^2}}{\btheta}=&-\frac{n\xi}{\xi^2\tau^2}+\frac{(\xi-1)}{\xi^2\tau^2}\sumN w_i^{-\frac{1}{\xi}}(\btheta)+\frac{2}{\xi^2\tau^2}\sumN w_i^{-1-\frac{1}{\xi}}(\btheta)+\frac{\xi(\xi+1)}{\xi^2\tau^2}\sumN w_i^{-2}(\btheta)\\
        &-\frac{\xi+1}{\xi^2\tau^2}\sumN w_i^{-2-\frac{1}{\xi}}(\btheta),\\
        \evalat[\Big]{\frac{\partial^2 L_n}{\partial \tau \partial \xi}}{\btheta}=&\frac{1}{\tau\xi^2}\left[-n+\frac{\xi+1}{\xi}\sumN w_i^{-\frac{1}{\xi}}(\btheta)+(2+\xi)\sumN w_i^{-1}(\btheta)-\frac{2(\xi+1)}{\xi}\sumN w_i^{-1-\frac{1}{\xi}}(\btheta)\right.\\
        &-(\xi+1)\sumN w_i^{-2}(\btheta)+\frac{\xi+1}{\xi}\sumN w_i^{-2-\frac{1}{\xi}}(\btheta)-\frac{1}{\xi}\sumN w_i^{-\frac{1}{\xi}}(\btheta)\log w_i(\btheta)\\
        &\left.+\frac{1}{\xi}\sumN w_i^{-1-\frac{1}{\xi}}(\btheta)\log w_i(\btheta)\right],\\
        \evalat[\Big]{\frac{\partial^2 L_n}{\partial \xi^2}}{\btheta}=&\frac{n(\xi+3)}{\xi^3}-\frac{3\xi+1}{\xi^4}\sumN w_i^{-\frac{1}{\xi}}(\btheta)-\frac{2(\xi+2)}{\xi^3}\sumN w_i^{-1}(\btheta)+\frac{2(2\xi+1)}{\xi^4}\sumN w_i^{-1-\frac{1}{\xi}}(\btheta)\\
        &+\frac{(\xi+1)}{\xi^3}\sumN w_i^{-2}(\btheta)-\frac{(\xi+1)}{\xi^4}\sumN w_i^{-2-\frac{1}{\xi}}(\btheta)-\frac{2}{\xi^3}\sumN \log w_i(\btheta)\\
        &+\frac{2(\xi+1)}{\xi^4}\sumN w_i^{-\frac{1}{\xi}}(\btheta)\log w_i(\btheta)-\frac{2}{\xi^4}\sumN w_i^{-1-\frac{1}{\xi}}(\btheta)\log w_i(\btheta)-\frac{1}{\xi^4}\sumN w_i^{-\frac{1}{\xi}}(\btheta)(\log w_i(\btheta))^2.
    \end{split}
\end{equation*}
\end{footnotesize}
\hfill\(\Box\)

\vskip 0.5cm
\newcommand\Nsupp{1}
\begin{lemma}[Uniform bound]\label{shrink_neighborhood}
Suppose a set of true GEV parameters $\btheta_0$ satisfies $-1/K_0<\xi_0\leq -1/(K_0+1)$, where $K_0\geq 1$ is some integer. Fix any $r>0$ small enough such that both \eqref{range_xi_supp} and $r<\frac{\tau_0}{3}$ hold. Then there almost surely exists $N_{\Nsupp}>0$ such that for any $n>N_{\Nsupp}$,
\begin{equation*}
\begin{split}
     \sumN w_i^{-k-\frac{1}{\xi}}(\btheta)\leq n \phi_k(\btheta_0,r), \text{ for }\btheta\in B_{r}(\hat{\btheta}_n)\cap \Omega_n,
\end{split}
\end{equation*}
in which $k=0,\ldots, K_0+1$, and $\phi_k(\btheta_0,r)>0$ is determined by $k$, $\btheta_0$ and the radius $r$.
\end{lemma}

{\it Proof. }
\textbf{A. Case $k=0,\ldots, K_0$.} When $\btheta\in B_{r}(\hat{\btheta}_n)$, we know from \eqref{range_xi_supp2} that $$0>k-(K_0+\epsilon_0)>k+\frac{1}{\xi}>k+\frac{1}{\xi_0+r}.$$
By the monotonicity of the power means,
\begin{equation}\label{holder_try}
   \frac{1}{n}\sumN w_i^{-k-\frac{1}{\xi}}(\btheta)\leq \left(\frac{1}{n}\sumN w_i^{-k-\frac{1}{\xi_0+r}}(\btheta)\right)^{\frac{k+1/\xi}{k+1/(\xi_0+r)}},
\end{equation}

To prove the right-hand side of \eqref{holder_try} is uniformly bounded, we first define three positive constants 
\begin{equation}\label{concavity_const_def}
    \eta_k=-k-\frac{1}{\xi_0+r},\text{ }\Psi_0=-\frac{2\xi_0}{\tau_0}, \text{ and }\Lambda_0=\frac{5(\xi_0-\tau_0)}{\xi_0\tau_0}.
\end{equation}
Apply the classic $c_r$-inequality \eqref{Cr_ineq} to $\eta_k$, and we obtain
\begin{equation}\label{after_c_r}
\begin{split}
     \sumN w_i^{\eta_k}(\btheta)&= \left(\frac{\xi\hat{\tau}_n}{\tau\hat{\xi}_n}\right)^{\eta_k}\sumN \left(w_i(\hat{\btheta}_n)+\frac{\hat{\xi}_n}{\hat{\tau}_n}(\hat{\beta}_n-\beta)\right)^{\eta_k}\\
     &\leq C_{\eta_k}\left(\frac{\xi\hat{\tau}_n}{\tau\hat{\xi}_n}\right)^{\eta_k}\left(\sumN w_i^{\eta_k}(\hat{\btheta}_n)+n\left|\frac{\hat{\xi}_n}{\hat{\tau}_n}(\hat{\beta}_n-\beta)\right|^{\eta_k}\right).
\end{split}
\end{equation}
By the strong consistency of $\hat{\btheta}_n$ and the fact that $r<\tau_0/3$, there almost surely exists $N_{\Nsupp,1}>0$ such that
\begin{equation}\label{abs_distance_nb0}
  \hat{\xi}_n<0,\;r<\hat{\tau}_n/2,\;\left|\frac{\hat{\xi}_n}{\hat{\tau}_n}(\hat{\beta}_n-\beta)\right|<-\frac{2\xi_0}{\tau_0}r=\Psi_0 r,
\end{equation}
for all $n>N_{\Nsupp,1}$ and $\btheta\in B_{r}(\hat{\btheta}_n)$. Thereupon it can be shown via a convexity argument similar to that yielding Lemma 2.1 that%\ref{sup_supp} that 
\begin{equation*}
    \forall \btheta\in B_{r}(\hat{\btheta}_n),\; 0<\frac{\xi\hat{\tau}_n}{\tau\hat{\xi}_n}\stackrel{\hat{\xi}_n<0}{\leq} \frac{(\hat{\xi}_n-r)\hat{\tau}_n}{(\hat{\tau}_n-r)\hat{\xi}_n}<1+\frac{4(\hat{\xi}_n-\hat{\tau}_n)}{\hat{\xi}_n\hat{\tau}_n} r.
\end{equation*}
Then we can find $N_{\Nsupp,2}>0$ such that for all $n>N_{\Nsupp,2}$,
\begin{equation}\label{abs_distance_nb1}
    \forall \btheta\in B_{r}(\hat{\btheta}_n),\; 0<\frac{\xi\hat{\tau}_n}{\tau\hat{\xi}_n}<1+\Lambda_0 r.
\end{equation}
\indent Also, \eqref{range_xi_supp1} guarantees $\frac{1}{n}\sumN w_i^{-k-\frac{1}{\xi_0-r}}(\hat{\btheta}_n)\asconv \Gamma(k\xi_0+\frac{\xi_0}{\xi_0-r}+1)$, there almost surely exists $N_{\Nsupp,3}>0$ such that for any $n>N_{\Nsupp,3}$, $\frac{1}{n}\sumN w_i^{-k-\frac{1}{\xi_0-r}}(\hat{\btheta}_n)\leq \Gamma(k\xi_0+\frac{\xi_0}{\xi_0-r}+1)+1$, where $k=1,\ldots, K_0$.
\vskip 0.2cm
Consequently, there almost surely exists $N_{\Nsupp}=\max\{N_{\Nsupp,1},N_{\Nsupp,2},N_{\Nsupp,3}\}$ such that for any $n>N_{\Nsupp}$, \eqref{after_c_r} can be further bounded as follows:
\begin{equation*}
        \sumN w_i^{\eta_k}(\btheta)
     \leq nC_{\eta_k}\left(1+\Lambda_0 r\right)^{\eta_k}\left(\Gamma(k\xi_0+\frac{\xi_0}{\xi_0-r}+1)+1+(\Psi_0 r)^{\eta_k}\right).
\end{equation*}
Plugging this result back into \eqref{holder_try} while noticing that $0<\frac{k+1/\xi}{k+1/(\xi_0-r)}<1$, we have
\begin{footnotesize}
\begin{equation*}
    \begin{split}
        \frac{1}{n}\sumN w_i^{-k-\frac{1}{\xi}}&(\btheta)\leq  C^{\frac{k+1/\xi}{k+1/(\xi_0-r)}}_{\eta_k}\left(1+\Lambda_0r\right)^{-k-\frac{1}{\xi}}\left(\Gamma(k\xi_0+\frac{\xi_0}{\xi_0-r}+1)+1+(\Psi_0 r)^{-k-\frac{1}{\xi_0-r}}\right)^{\frac{k+1/\xi}{k+1/(\xi_0-r)}}\\
        &\stackrel{\eqref{range_xi_supp2}}{\leq}  C_{\eta_k}\left(1+\Lambda_0r\right)^{-k-\frac{1}{\xi_0+r}}\left(\Gamma(k\xi_0+\frac{\xi_0}{\xi_0-r}+1)+1+(\Psi_0 r)^{-k-\frac{1}{\xi_0-r}}\right)=:\phi_k(\btheta_0,r).
    \end{split}
\end{equation*}
\end{footnotesize}

\textbf{B. Case $k=K_0+1$.} For $\btheta\in B_{r}(\hat{\btheta}_n)$, \eqref{range_xi_supp2} ensures
$$K_0+1+\frac{1}{\xi_0-r}>K_0+1+\frac{1}{\xi}>K_0+1+\frac{1}{\xi_0+r},$$
and \eqref{range_xi_supp} ensures
$$K_0+1+\frac{1}{\xi_0-r}>K_0+1+\frac{1}{\xi_0}\geq 0.$$

If $K_0+1+\frac{1}{\xi}>0$, the monotonicity of the power means yields
$$\frac{1}{n}\sumN w_i^{-K_0-1-\frac{1}{\xi}}(\btheta)\leq \left(\frac{1}{n}\sumN w_i^{-K_0-1-\frac{1}{\xi_0-r}}(\btheta)\right)^{\frac{K_0+1+1/\xi}{K_0+1+1/(\xi_0-r)}}.$$

If $K_0+1+\frac{1}{\xi}<0$, the monotonicity of the power means then yields
$$\frac{1}{n}\sumN w_i^{-K_0-1-\frac{1}{\xi}}(\btheta)\leq \left(\frac{1}{n}\sumN w_i^{-K_0-1-\frac{1}{\xi_0+r}}(\btheta)\right)^{\frac{K_0+1+1/\xi}{K_0+1+1/(\xi_0+r)}}.$$

For both cases, we follow a similar proof to the previous case with the help of the classic $c_r$-inequality \eqref{Cr_ineq} to obtain upper bounds that depends only on $\btheta_0$ and $r$. Pick the greater of the two bounds, and we have $\phi_{K_0+1}(\btheta_0,r)$.
\hfill\(\Box\)

\vskip 0.5cm
\newcommand\NsuppTwo{2}
\begin{lemma}[Uniform bound]\label{shrink_neighborhood2}
Under the assumptions of Lemma \ref{shrink_neighborhood}, fix any $r>0$ small enough such that both \eqref{range_xi_supp} and $r<\min\{\frac{\tau_0}{3},\frac{0.8}{\Lambda_0}\}$ hold, where $\Lambda_0$ is defined in \eqref{concavity_const_def}. Then there almost surely exists $N_{\NsuppTwo}$ such that for any $n>N_{\NsuppTwo}$ and $\btheta\in B_{r}(\hat{\btheta}_n)\cap \Omega_n$,
\begin{equation*}
   \sumN w_i^{-k-\frac{1}{\xi}}(\btheta)\left|\log w_i(\btheta)\right|^b\leq n \varphi_{k}(\btheta_0,r),
\end{equation*}
where $k\in\{0,\ldots, K_0+1-b\}$, $b\in\{1,2,3\}$, and $\varphi_{k}(\btheta_0,r)$ only depends on $k$, $\btheta_0$ and $r$.
\end{lemma}
{\it Proof. }
We deduce by \eqref{wi} 
\begin{equation}\label{decomposition}
\begin{split}
    \sumN w_i^{-k-\frac{1}{\xi}}(\btheta)\left|\log w_i(\btheta)\right|^b=&\left|\log \left(\frac{\xi\hat{\tau}_n}{\tau\hat{\xi}_n}\right)\right|^b\sumN w_i^{-k-\frac{1}{\xi}}(\btheta)+\\
    &\hspace{0.6cm}\sumN w_i^{-k-\frac{1}{\xi}}(\btheta)\bigg|\log\bigg( w_i(\hat{\btheta}_n)+\frac{\hat{\xi}_n(\hat{\beta}_n-\beta)}{\hat{\tau}_n}\bigg)\bigg|^b.
\end{split}
\end{equation}
Then \eqref{abs_distance_nb1} implies for all $n>N_{\Nsupp}$,
\begin{equation*}
    \left|\log \left(\frac{\xi\hat{\tau}_n}{\tau\hat{\xi}_n}\right)\right|\leq \frac{\sqrt{2}}{1-\Lambda_0r}\Lambda_0r< \frac{\sqrt{2}}{1-0.8}\Lambda_0r=5\sqrt{2}\Lambda_0r,
\end{equation*}
and Lemma \ref{shrink_neighborhood} guarantees
\begin{equation}\label{log_abs_approx}
    \left|\log \left(\frac{\xi\hat{\tau}_n}{\tau\hat{\xi}_n}\right)\right|^b\cdot\sumN w_i^{-k-\frac{1}{\xi}}(\btheta)<n(5\sqrt{2}\Lambda_0r)^b\phi_k(\btheta_0,r), \text{ for }\btheta\in B_{r}(\hat{\btheta}_n)\cap \Omega_n.
\end{equation}

Since $|\log x-\log y|\leq |x-y|/\min(x,y)$, and \eqref{abs_distance_nb0} ensures $\left|\frac{\hat{\xi}_n(\hat{\beta}_n-\beta)}{\hat{\tau}_n}\right|<\Psi_0 r$ for $n>N_{\Nsupp,1}$,
\begin{equation*}
    \left|\log\bigg( w_i(\hat{\btheta}_n)+\frac{\hat{\xi}_n(\hat{\beta}_n-\beta)}{\hat{\tau}_n}\bigg)-\log w_i(\hat{\btheta}_n)\right|\leq \frac{\hat{\xi}_n(\hat{\beta}_n-\beta)}{\hat{\tau}_n w_i(\hat{\btheta}_n)}<\Psi_0 r w^{-1}_i(\hat{\btheta}_n).
\end{equation*}
Therefore,
\begin{equation}\label{b1_approx}
    \left|\log\bigg( w_i(\hat{\btheta}_n)+\frac{\hat{\xi}_n(\hat{\beta}_n-\beta)}{\hat{\tau}_n}\bigg)\right|<\left|\log w_i(\hat{\btheta}_n)\right|+\Psi_0 r w^{-1}_i(\hat{\btheta}_n).
\end{equation}
Apply the $c_r$-inequality \eqref{Cr_ineq} on \eqref{b1_approx} to get 
\begin{equation}\label{b2_approx}
    \left|\log\bigg( w_i(\hat{\btheta}_n)+\frac{\hat{\xi}_n(\hat{\beta}_n-\beta)}{\hat{\tau}_n}\bigg)\right|^b<C_b\left(\left|\log w_i(\hat{\btheta}_n)\right|^b+(\Psi_0 r)^b w^{-b}_i(\hat{\btheta}_n)\right).
\end{equation}

Plugging \eqref{log_abs_approx} and \eqref{b2_approx} back into \eqref{decomposition}, we have %for $b=1$
\begin{small}
\begin{equation*}
\begin{split}
    \sumN w_i^{-k-\frac{1}{\xi}}(\btheta)\left|\log w_i(\btheta)\right|^b< C_b\sumN w_i^{-k-\frac{1}{\xi}}(\btheta)&|\log w_i(\hat{\btheta}_n)|^b+C_b(\Psi_0 r)^b\sumN w_i^{-k-\frac{1}{\xi}}(\btheta)w^{-b}_i(\hat{\btheta}_n)\\
    &+n(5\sqrt{2}\Lambda_0r)^b\phi_k(\btheta_0,r).
\end{split}
\end{equation*}
\end{small}
By Lemma \ref{shrink_neighborhood}, the right-hand side of the previous inequation can be uniformly bounded by a constant $\varphi_{k}(\btheta_0,r)$ that is determined by $k$, $\btheta_0$ and $r$, $k=0,\ldots,K_0+1-b$. 

\hfill\(\Box\)

\vskip 0.5cm
\begin{proposition}[Lipschitz continuity]\label{lem:lipschitz_cont}
Under the assumptions of Lemma \ref{shrink_neighborhood} and \ref{shrink_neighborhood2}, there almost surely exists $N$ such that, for any $n>N$ and $\btheta\in B_{r}(\hat{\btheta}_n)\cap \Omega_n$,
\begin{equation*}
\begin{split}
    \left|\sumN w_i^{-k-\frac{1}{\xi}}(\btheta)-\sumN w_i^{-k-\frac{1}{\hat{\xi}}}(\hat{\btheta}_n)\right|&\leq n\psi_k(\btheta_0,r)||\btheta-\hat{\btheta}_n||_{\infty},\\
    \left|\sumN w_i^{-k'-\frac{1}{\xi}}(\btheta)\log w_i(\btheta)-\sumN w_i^{-k'-\frac{1}{\hat{\xi}}}(\hat{\btheta}_n)\log w_i(\hat{\btheta}_n)\right|&\leq n\psi_{k'}(\btheta_0,r)||\btheta-\hat{\btheta}_n||_{\infty},\\
    \left|\sumN w_i^{-k''-\frac{1}{\xi}}(\btheta)\log^2 w_i(\btheta)-\sumN w_i^{-k''-\frac{1}{\hat{\xi}}}(\hat{\btheta}_n)\log^2 w_i(\hat{\btheta}_n)\right|&\leq n\psi_{k''}(\btheta_0,r)||\btheta-\hat{\btheta}_n||_{\infty},
\end{split}
\end{equation*}
where $k=0,\ldots, K_0$, $k'=0,\ldots, K_0-1$, $k''=0,\ldots, K_0-2$, and $\psi_{k}(\btheta_0,r),\;\psi_{k}(\btheta_0,r),\;\psi_{k''}(\btheta_0,r)$ are positive constants determined by $k$ (or $k'$, $k''$), $\btheta_0$ and the radius $r$.
\end{proposition}
{\it Proof. }
Given that $\btheta\in B_{r}(\hat{\btheta}_n)\cap \Omega_n$ is a convex set under the parametrization $(\beta,\xi,\tau)$, the mean value theorem can be applied to assess
\begin{equation*}
    \left|\sumN w_i^{-k-\frac{1}{\xi}}(\btheta)-\sumN w_i^{-k-\frac{1}{\hat{\xi}_n}}(\hat{\btheta}_n)\right|\leq \left\lVert\sumN \evalat[\Big]{\bigtriangledown \bigg(w_i^{-k-\frac{1}{\xi}}\bigg)}{\btheta^+_n}\right\rVert_{\infty} \cdot||\btheta-\hat{\btheta}_n||_{\infty},
\end{equation*}
where $\btheta^+_n=(\beta^+_n,\xi^+_n,\tau^+_n)$ lies between $\btheta$ and $\hat{\btheta}_n$, and the Jacobian
\begin{footnotesize}
\begin{equation*}
    \evalat[\Big]{\bigtriangledown \big(w_i^{-k-\frac{1}{\xi}}\big)}{\btheta}=\frac{k\xi+1}{\xi}w_i^{-k-\frac{1}{\xi}}(\btheta)\cdot\left(\frac{\xi}{\tau}w_i^{-1}(\btheta),\;\frac{1}{k\xi^2+\xi}\log w_i(\btheta)-\frac{1}{\xi}(1-w_i^{-1}(\btheta)),\;\frac{1}{\tau}(1-w_i^{-1}(\btheta))\right).
\end{equation*}
\end{footnotesize}
By Lemma \ref{shrink_neighborhood} and \ref{shrink_neighborhood2}, we can easily find $\psi_k(\btheta_0,r)>0$ which is determined by $k$, $\btheta_0$ and the radius $r$ such that
\begin{footnotesize}
\begin{equation*}
    \left\lVert\sumN \evalat[\Big]{\bigtriangledown \bigg(w_i^{-k-\frac{1}{\xi}}\bigg)}{\btheta^+_n}\right\rVert_{\infty}\leq n\psi_k(\btheta_0,r).
\end{equation*}
\end{footnotesize}
Therefore the first inequality holds.

Similarly, we derive
\begin{footnotesize}
\begin{equation*}
    \begin{split}
        \left.\evalat[\Big]{\bigtriangledown \bigg(w_i^{-k'-\frac{1}{\xi}}\log w_i\bigg)}{\btheta}=\right(&\left.w_i^{-k'-1-\frac{1}{\xi}}(\btheta)\left[\frac{k'\xi+1}{\tau}\log w_i(\btheta)-\frac{\xi}{\tau}\right],\right.\\
        &w_i^{-k'-\frac{1}{\xi}}(\btheta)\log^2 w_i(\btheta)\left[\frac{1}{\xi^2}-\frac{k'\xi+1}{\xi^2}\frac{1-w_i^{-1}(\btheta)}{\log w_i(\btheta)}+\frac{1-w_i^{-1}(\btheta)}{\xi\log^2 w_i(\btheta)}\right],\\
        &\left.w_i^{-k'-\frac{1}{\xi}}(\btheta)(1-w_i^{-1}(\btheta))\left[\frac{k\xi+1}{\xi\tau}\log w_i(\btheta)-\frac{1}{\tau}\right]\right),
    \end{split}
\end{equation*}
\end{footnotesize}
and 
\begin{footnotesize}
\begin{equation*}
    \begin{split}
        \evalat[\Big]{\bigtriangledown \bigg(w_i^{-k'-\frac{1}{\xi}}\log^2 w_i\bigg)}{\btheta}&=\left(w_i^{-k'-1-\frac{1}{\xi}}(\btheta)\log w_i(\btheta)\left[\frac{k'\xi+1}{\tau}\log w_i(\btheta)-\frac{2\xi}{\tau}\right],\right.\\
        &w_i^{-k'-\frac{1}{\xi}}(\btheta)\log^2 w_i(\btheta)\left[\frac{1}{\xi^2}\log w_i(\btheta)-\frac{k'\xi+1}{\xi^2}(1-w_i^{-1}(\btheta))+\frac{2}{\xi}\frac{1-w_i^{-1}(\btheta)}{\log w_i(\btheta)}\right],\\
        &\left.w_i^{-k'-\frac{1}{\xi}}(\btheta)(1-w_i^{-1}(\btheta))\log w_i(\btheta)\left[\frac{k\xi+1}{\xi\tau}\log w_i(\btheta)-\frac{2}{\tau}\right]\right),
    \end{split}
\end{equation*}
\end{footnotesize}
in which all the terms under expansion have been studied in Lemma \ref{shrink_neighborhood} and \ref{shrink_neighborhood2}. Hence, the second and third inequality also holds.
\hfill\(\Box\)

%%%%%%%%%%%%%%%%%%%%%%%%%%%%%%%%%%%%%%%%%%%%%%%%%%
%%                 Appendix E                   %%
%%%%%%%%%%%%%%%%%%%%%%%%%%%%%%%%%%%%%%%%%%%%%%%%%%

\section{Proof of Proposition \ref{prop:step1}}\label{proof:step1}

We first study the relationship between $\beta_n(\xi)$ and $\hat{\beta}_n=\beta_n(\hat{\xi}_n)$ when $\xi\neq \hat{\xi}_n$. This relationship is illustrated visually in Figure \ref{fig:num_sim_cross_section}, where we simulate $n=1,000$ samples using $(\tau_0,\mu_0,\xi_0)=(0.5,20,0.2)$, and then evaluate the log-likelihood on different cross sections. Through comparing the intercepts, the numerical results confirm Proposition \ref{prop:profile_lik_uniqueML} in that $\beta_n(\xi)<\hat{\beta}_n$ when  $\xi<\hat{\xi}_n$, and $\beta_n(\xi)>\hat{\beta}_n$ as  $\xi>\hat{\xi}_n$. This property stems from the monotonicity of $\beta_n(\xi)$ as a function of $\xi$.
\begin{figure}
    \centering
    \includegraphics[width=0.327\textwidth]{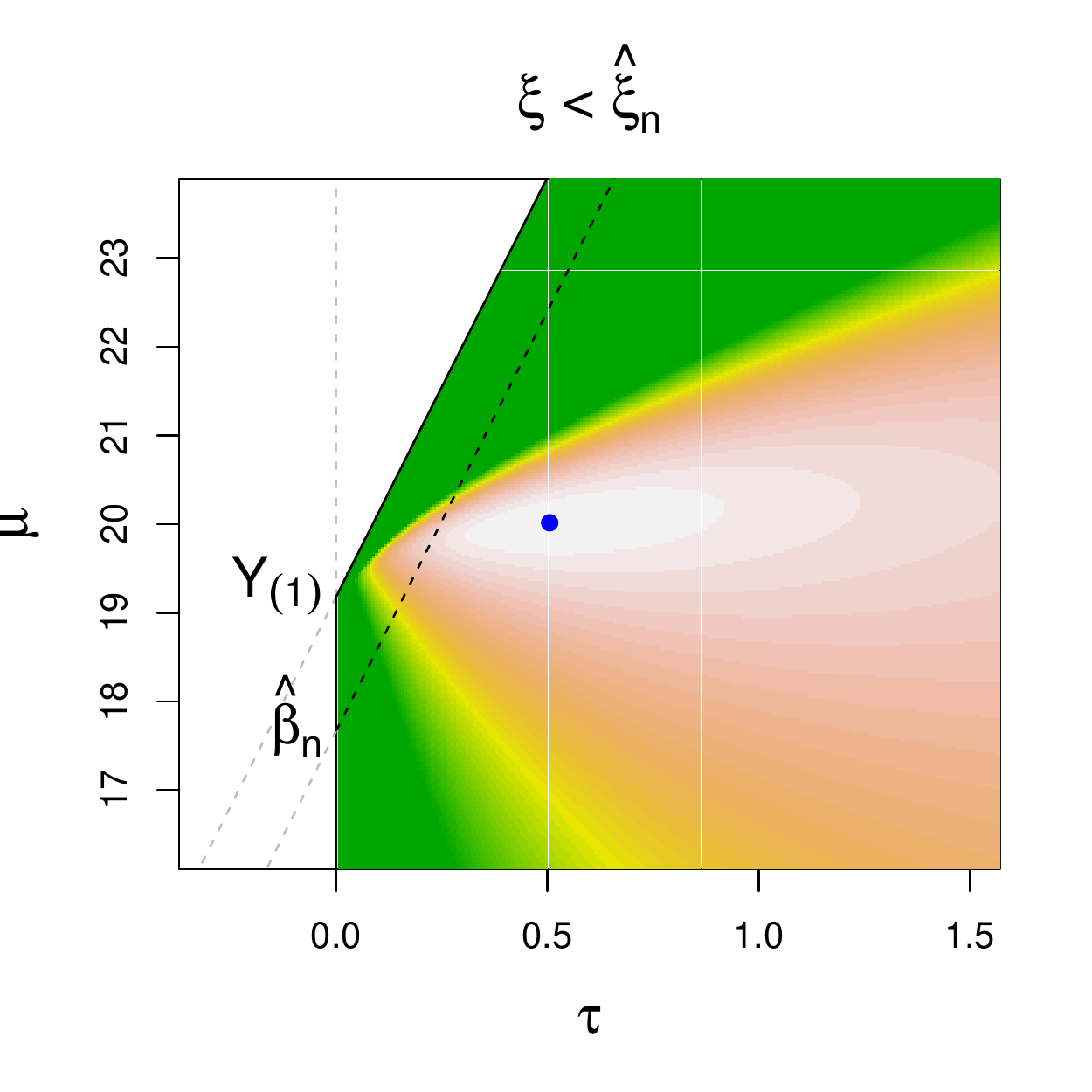}
    \includegraphics[width=0.327\textwidth]{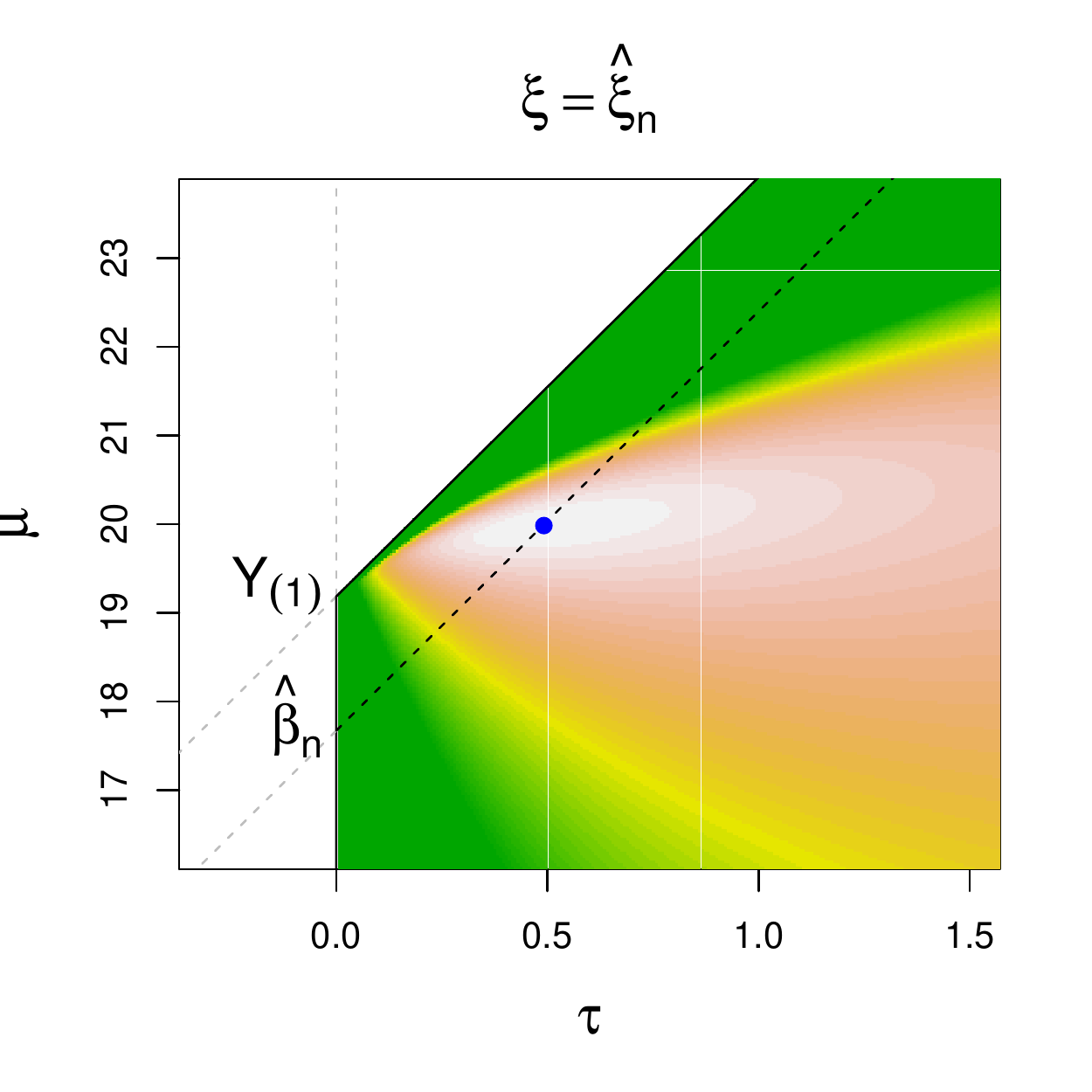}
    \includegraphics[width=0.327\textwidth]{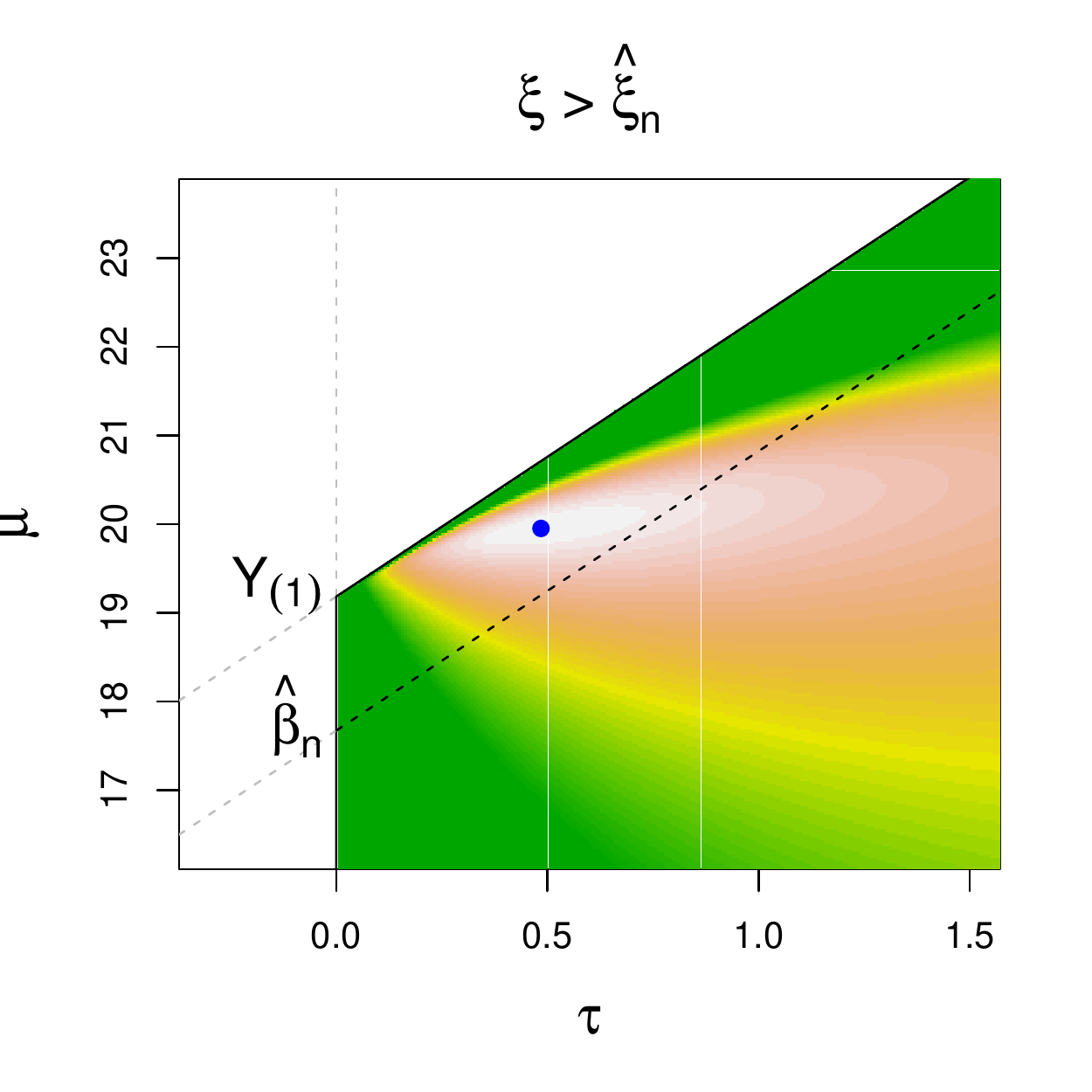}
    \caption{Contour plots of the log-likelihood $L_n(\btheta)$ at cross sections of different levels of $\xi$, where $Y_i,\;i=1,\ldots,n$ that are sampled from true $\xi_0=0.2$. Blue points mark the maximizer $(\tau_n(\xi),\mu_n(\xi))$ on each slice, and the black dashed lines showing the lines that has intercept $\hat{\beta}_n$ and slope $1/\xi$. \textit{Left}: $\xi=0.5\hat{\xi}_n$, and we see $\beta_n(\xi)<\hat{\beta}_n$. \textit{Middle}: $\xi=\hat{\xi}_n$, and $\hat{\beta}_n=\beta_n(\hat{\xi}_n)$. \textit{Right}: $\xi=1.5\hat{\xi}_n$, and we see $\beta_n(\xi)>\hat{\beta}_n$.}
    \label{fig:num_sim_cross_section}
\end{figure}
\begin{lemma}\label{lem:beta_xi_mono}
Under the same assumptions as Proposition \ref{prop:profile_lik_uniqueML}, $\beta_n(\xi)=\mu_n(\xi)-\tau_n(\xi)/\xi$ as function of $\xi$ satisfies $\beta'_n(\xi)>0$ for any $\xi\neq 0$ and $-1<\xi<n-1$. Furthermore, Lemma \ref{lem:monotone_on_beta} guarantees $\beta_n(\xi)\nearrow Y_{(1)}$ as $\xi\nearrow n-1$, and $\beta_n(\xi)\searrow Y_{(n)}$ as $\xi\searrow -1$; see Figure \ref{fig:beta_xi}.
\end{lemma}
\begin{proof}
See the supplementary material for the proof, in which we calculate the derivative via applying the implicit function theorem on \eqref{max_cross_section_cond}.
\end{proof}
\begin{figure}
    \centering
    \includegraphics[width=0.35\textwidth]{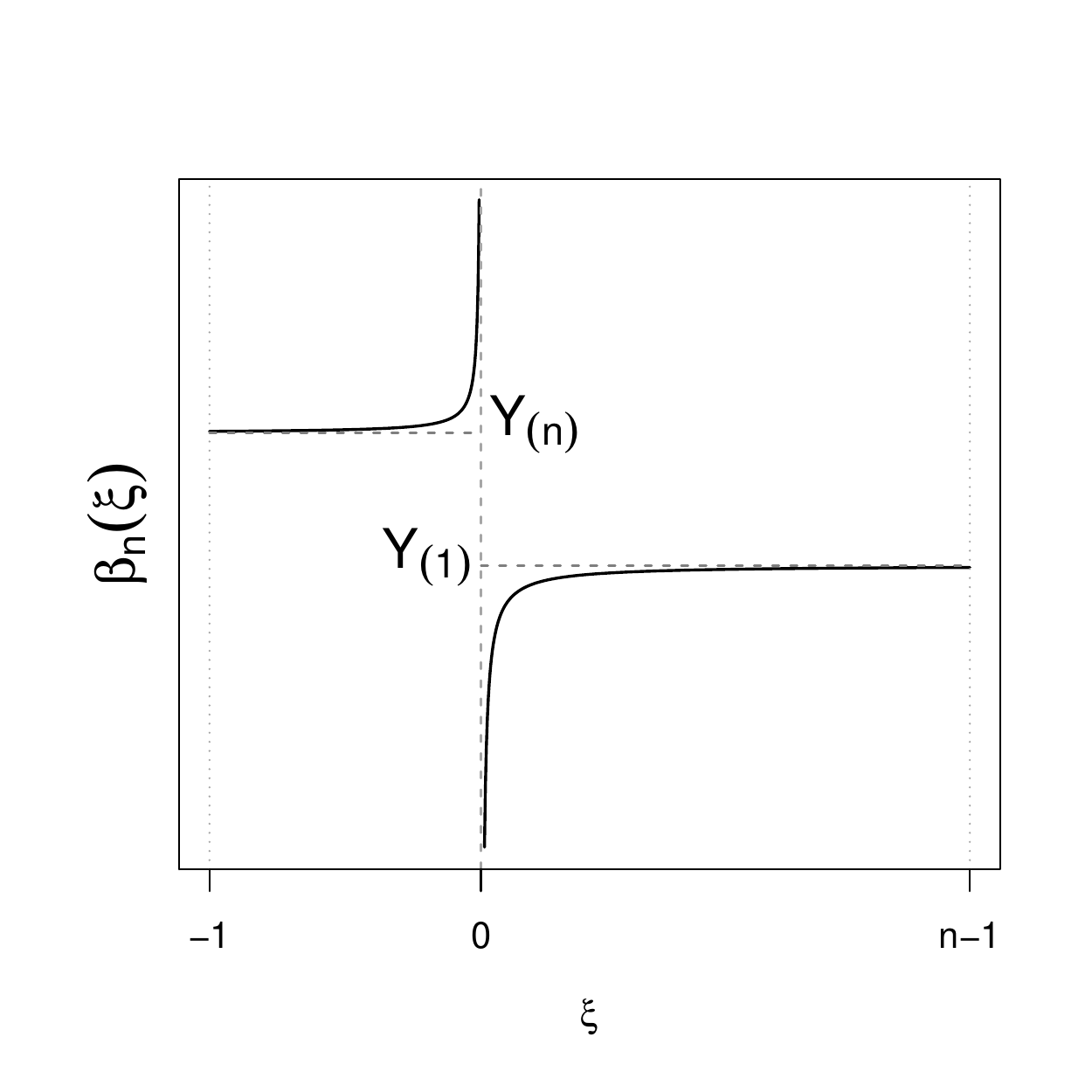}
    \caption{Illustration of $\beta_n(\xi)$ when $\xi\in (-1,n-1)$.
    }
\label{fig:beta_xi}
\end{figure}

\begin{lemma}\label{lem:monotone_on_xi}
Suppose $\xi_0>0$, and $Y_1,\ldots, Y_n\iidY P_{\btheta_0}$. Define
$$F_n(\xi)=n\log\xi+\frac{1}{\xi}\sumN\log w_i(\hat{\btheta}_n)+n\log\sumN w_i(\hat{\btheta}_n)^{-\frac{1}{\xi}}.$$
Then for any $\xi>0$, $F_n(\xi)-F_n(\hat{\xi}_n)\geq n\left(\frac{\hat{\xi}_n}{\xi}-\log\frac{\hat{\xi}_n}{\xi}-1\right)$.
\end{lemma}
\begin{proof}
The first derivative of $F_n$ is
$$F_n'(\xi)=\frac{1}{\xi^2}\left[n\xi-\sumN\log w_i(\hat{\btheta}_n)+\frac{n\sumN w_i(\hat{\btheta}_n)^{-\frac{1}{\xi}}\log w_i(\hat{\btheta}_n)}{\sumN w_i(\hat{\btheta}_n)^{-\frac{1}{\xi}}}\right].$$
Since $\evalat[\big]{\frac{\partial L_n}{\partial \xi}}{\hat{\btheta}_n}=0$ by definition, we deduce
$$-\sumN\log w_i(\hat{\btheta}_n)=-n\hat{\xi}_n-\frac{n\sumN w_i(\hat{\btheta}_n)^{-\frac{1}{\hat{\xi}_n}}\log w_i(\hat{\btheta}_n)}{\sumN w_i(\hat{\btheta}_n)^{-\frac{1}{\hat{\xi}_n}}}.$$
Therefore,
$$F_n'(\xi)=\frac{n}{\xi^2}\left[\xi-\hat{\xi}_n+\frac{\sumN w_i(\hat{\btheta}_n)^{-\frac{1}{\xi}}\log w_i(\hat{\btheta}_n)}{\sumN w_i(\hat{\btheta}_n)^{-\frac{1}{\xi}}}-\frac{\sumN w_i(\hat{\btheta}_n)^{-\frac{1}{\hat{\xi}_n}}\log w_i(\hat{\btheta}_n)}{\sumN w_i(\hat{\btheta}_n)^{-\frac{1}{\hat{\xi}_n}}}\right].$$
Denote $x_i=w_{(i)}(\hat{\btheta}_n)^{-\frac{1}{\xi}}$, $y_i=w_{(i)}(\hat{\btheta}_n)^{-\frac{1}{\hat{\xi}_n}}$, $z_i=\log w_{(i)}(\hat{\btheta}_n)$ and $u_i=1$. When $\xi>\hat{\xi}_n$,
\begin{equation*}
\begin{vmatrix}
x_i&x_j\\
y_i&y_j
\end{vmatrix}\times
\begin{vmatrix}
z_r&z_s\\
u_r&u_s
\end{vmatrix}=
\frac{\left[w_{(i)}(\hat{\btheta}_n)^{\frac{1}{\hat{\xi}_n}-\frac{1}{\xi}}-w_{(j)}(\hat{\btheta}_n)^{\frac{1}{\hat{\xi}_n}-\frac{1}{\xi}}\right][\log w_{(r)}(\hat{\btheta}_n) -\log w_{(s)}(\hat{\btheta}_n) ]}{[w_{(i)}(\hat{\btheta}_n)w_{(j)}(\hat{\btheta}_n)]^{\frac{1}{\hat{\xi}_n}}}>0,
\end{equation*}
for any pairs of indices $i<j$ and $r<s$. Hence we can apply the Seitz inequality to acquire
$$\sumN w_i(\hat{\btheta}_n)^{-\frac{1}{\xi}}\log w_i(\hat{\btheta}_n)\sumN w_i(\hat{\btheta}_n)^{-\frac{1}{\hat{\xi}_n}}\geq \sumN w_i(\hat{\btheta}_n)^{-\frac{1}{\hat{\xi}_n}}\log w_i(\hat{\btheta}_n)\sumN w_i(\hat{\btheta}_n)^{-\frac{1}{\xi}}.$$
Similarly, we can prove the above inequality holds in an opposite direction when $\xi<\hat{\xi}_n$. 

Consequently,
\begin{equation*}
    F_n'(\xi)\geq\frac{n(\xi-\hat{\xi}_n)}{\xi^2} \text{ for }\xi>\hat{\xi}_n, \text{ and }F_n'(\xi)\leq\frac{n(\xi-\hat{\xi}_n)}{\xi^2} \text{ for }\xi<\hat{\xi}_n.
\end{equation*}
By Newton-Leibniz formula, 
\begin{equation}\label{An1}
    F_n(\xi)-F_n(\hat{\xi}_n)\geq \int_{\hat{\xi}_n}^{\xi}\frac{n(\xi-\hat{\xi}_n)}{\xi^2} d\xi=n\left(\frac{\hat{\xi}_n}{\xi}-\log\frac{\hat{\xi}_n}{\xi}-1\right)\text{ for all }\xi>0.
\end{equation}
\end{proof}

\begin{lemma}\label{lem:A_n2}
Suppose $\xi_0>0$, and $Y_1,\ldots, Y_n\iidY P_{\btheta_0}$. Define
$$G_n(\xi)=\frac{\xi+1}{\xi}\sumN\log\frac{Y_i-\hat{\beta}_n}{Y_i-\beta_n(\xi)}+n\log\frac{\sumN (Y_i-\hat{\beta}_n)^{-1/\xi}}{\sumN (Y_i-\beta_n(\xi))^{-1/\xi}}.$$
Then for $0<\xi<\hat{\xi}_n$, 
$G_n(\xi)\leq -n\left[2\xi_0\gamma\left(\frac{1}{\hat{\xi}_n}-\frac{1}{\xi}\right)+\frac{1}{2}\log\Gamma\left(\frac{\xi_0}{\xi}+1\right)\right]$ almost surely when $n$ is sufficiently large.
\end{lemma}
\begin{proof}
We first calculate
\begin{footnotesize}
\begin{align*}
    G'_n(\xi)=&\frac{\beta'_n(\xi)}{\xi}\left[(\xi+1)\sumN (Y_i-\beta_n(\xi))^{-1}-\frac{n\sumN (Y_i-\beta_n(\xi))^{-1-\frac{1}{\xi}}}{\sumN (Y_i-\beta_n(\xi))^{-\frac{1}{\xi}}}\right]+\frac{1}{\xi^2}\sumN\log \frac{Y_i-\beta_n(\xi)}{Y_i-\hat{\beta}_n}+\\
    &\frac{n\sumN (Y_i-\hat{\beta}_n)^{-\frac{1}{\xi}}\log(Y_i-\hat{\beta}_n)}{\xi^2\sumN (Y_i-\hat{\beta}_n)^{-\frac{1}{\xi}}}-\frac{n\sumN (Y_i-\beta_n(\xi))^{-\frac{1}{\xi}}\log(Y_i-\beta_n(\xi))}{\xi^2\sumN (Y_i-\beta_n(\xi))^{-\frac{1}{\xi}}},
\end{align*}
\end{footnotesize}
in which the term in the square brackets is 0 for $\xi>0$ due to the second equation of \eqref{max_cross_section_cond}.

For $\xi>0$, we deduce from Chebyshev's sum inequality that 
\begin{equation*}
    \sumN\log {(Y_i-\beta_n(\xi))}\sumN (Y_i-\beta_n(\xi))^{-\frac{1}{\xi}}\geq n\sumN (Y_i-\beta_n(\xi))^{-\frac{1}{\xi}}\log(Y_i-\beta_n(\xi)),
\end{equation*}
and thus 
\begin{equation}\label{eqn:G_n_approx}
    G'_n(\xi)\geq -\frac{1}{\xi^2}\sumN\log (Y_i-\hat{\beta}_n)+\frac{n\sumN (Y_i-\hat{\beta}_n)^{-\frac{1}{\xi}}\log(Y_i-\hat{\beta}_n)}{\xi^2\sumN (Y_i-\hat{\beta}_n)^{-\frac{1}{\xi}}}.
\end{equation}
Apply the Newton-Leibniz formula to both sides of \eqref{eqn:G_n_approx}, and we obtain, for $\xi<\hat{\xi}_n$, that
\begin{small}
\begin{equation}\label{eqn:G_n_approx2}
    \begin{split}
        G_n(\hat{\xi}_n)-G_n(\xi)&\geq \int_{\xi}^{\hat{\xi}_n}-\frac{1}{\xi^2}\sumN\log (Y_i-\hat{\beta}_n)+\frac{n\sumN (Y_i-\hat{\beta}_n)^{-\frac{1}{\xi}}\log(Y_i-\hat{\beta}_n)}{\xi^2\sumN (Y_i-\hat{\beta}_n)^{-\frac{1}{\xi}}}d\xi\\
        &=\left(\frac{1}{\hat{\xi}_n}-\frac{1}{\xi}\right)\sumN\log (Y_i-\hat{\beta}_n)+n\log\frac{\sumN (Y_i-\hat{\beta}_n)^{-\frac{1}{\hat{\xi}_n}}}{\sumN (Y_i-\hat{\beta}_n)^{-\frac{1}{\xi}}}\\
        &=\left(\frac{1}{\hat{\xi}_n}-\frac{1}{\xi}\right)\sumN\log w_i(\hat{\btheta}_n)-n\log\frac{\sumN w_i^{-\frac{1}{\xi}}(\hat{\btheta}_n)}{n}.
    \end{split}
\end{equation}
\end{small}
Proposition \ref{prop:uniform_consistency} ensures $\log\frac{1}{n}\sumN w_i^{-\frac{1}{\xi}}(\hat{\btheta}_n)\leq  2\log\Gamma\left(\frac{\xi_0}{\xi}+1\right)$ uniformly, and $\frac{1}{n}\sumN\log w_i(\hat{\btheta}_n)\leq 2\xi_0\gamma$ for sufficiently large $n$. Note that $G_n(\hat{\xi}_n)=0$, and \eqref{eqn:G_n_approx2} becomes
\begin{equation*}
    -G_n(\xi)\geq 2n\xi_0\gamma\left(\frac{1}{\hat{\xi}_n}-\frac{1}{\xi}\right)+\frac{n}{2}\log\Gamma\left(\frac{\xi_0}{\xi}+1\right),
\end{equation*}
which proves the lemma.

\end{proof}

\subsection{Proof of relation (\ref{A_cond})} %\eqref{} is not allowed.

\begin{proof}
When $\xi>\hat{\xi}_n$, $\beta_n(\xi)>\hat{\beta}_n$ and
\begin{equation*}
    -\frac{n\sumN\delta_i^{-\frac{1}{\xi}}\log\delta_i}{\sumN\delta_i^{-\frac{1}{\xi}}}+\sumN \log\delta_i \leq -\frac{n\sumN\hat{\delta}_i^{-\frac{1}{\xi}}\log\hat{\delta}_i}{\sumN\hat{\delta}_i^{-\frac{1}{\xi}}},
\end{equation*}
where $\hat{\delta}_i=\hat{\xi}_n(Y_i-\hat{\beta}_n)$ and $\delta_i=\xi(Y_i-\beta_n(\xi))$. Plug this inequality in \eqref{eqn:PL_dev}, and we obtain
\begin{equation*}
    PL'_n(\xi)\leq -\frac{n}{\xi}-\frac{n\sumN\hat{\delta}_i^{-\frac{1}{\xi}}\log\hat{\delta}_i}{\xi^2\sumN\hat{\delta}_i^{-\frac{1}{\xi}}}.
\end{equation*}
Apply the Newton-Leibniz formula, and we get
\begin{equation}\label{PL_n_approx1}
\begin{split}
    PL_n(\xi)-PL_n&(\hat{\xi}_n)\leq -\int_{\hat{\xi}_n}^{\xi}\frac{n}{\xi}d\xi-\int_{\hat{\xi}_n}^{\xi}\frac{n\sumN\hat{\delta}_i^{-\frac{1}{\xi}}\log\hat{\delta}_i}{\xi^2\sumN\hat{\delta}_i^{-\frac{1}{\xi}}}d\xi\\
    &=-n\log\frac{\xi}{\hat{\xi}_n}-n\log\frac{\sumN \hat{\delta}_i^{-\frac{1}{\xi}}}{\sumN \hat{\delta}_i^{-\frac{1}{\hat{\xi}_n}}}\\
    &=-n\log\frac{\xi}{\hat{\xi}_n}-n\left(\frac{1}{\xi}-\frac{1}{\hat{\xi}_n}\right)\log\hat{\tau}_n-n\log\left[\frac{1}{n}\sumN w_i^{-\frac{1}{\xi}}(\hat{\btheta}_n)\right].
\end{split}
\end{equation}
Proposition \ref{prop:uniform_consistency} ensures $\frac{1}{n}\sumN w_i^{-\frac{1}{\xi}}(\hat{\btheta}_n)\asconv \Gamma\left(\frac{\xi_0}{\xi}+1\right)$ uniformly. Therefore, for sufficiently large $n$, the right-hand side of \eqref{PL_n_approx1} can be further bounded as follows
\begin{equation*}
    PL_n(\xi)-PL_n(\hat{\xi}_n)\leq -\frac{n}{2}\left[\left(\frac{1}{\xi}-\frac{1}{\xi_0}\right)\log\tau_0+\log\Gamma\left(\frac{\xi_0}{\xi}\right)\right].
\end{equation*}
We know from the definition of $C_0$ in \eqref{K_compact} that the right-hand side of the above inequality is negative when $\xi>C_0\hat{\xi}_n$.

For $\xi<\hat{\xi}_n$, we evaluate $PL_n(\xi)-PL_n(\hat{\xi}_n)$ using \eqref{eqn:PL_xi}:
\begin{small}
\begin{align*}
        PL_n(\xi)-PL_n(\hat{\xi}_n)=&-n\log\xi-n\log\left[\frac{1}{n}\sumN (Y_i-\beta_n(\xi))^{-1/\xi}\right]-\frac{\xi+1}{\xi}\sumN\log(Y_i-\beta_n(\xi))\\
        &+n\log\hat{\xi}_n+n\log\left[\frac{1}{n}\sumN (Y_i-\hat{\beta}_n)^{-1/\hat{\xi}_n}\right]+\frac{\hat{\xi}_n+1}{\hat{\xi}_n}\sumN\log(Y_i-\hat{\beta}_n)\numberthis \label{PL_n_Diff}\\
        =&F_n(\hat{\xi}_n)-F_n(\xi)+G_n(\xi),
\end{align*}
\end{small}
where $F_n$ and $G_n$ are defined in Lemma \ref{lem:monotone_on_xi} and \ref{lem:A_n2} respectively. Now combine the results from those two lemmas to get
\begin{small}
\begin{equation}\label{eqn:PL_diff_pos}
    PL_n(\xi)-PL_n(\hat{\xi}_n)\leq -n\left[\frac{1}{2}\log\Gamma\left(\frac{\xi_0}{\xi}+1\right)-\log\frac{\hat{\xi}_n}{\xi}-\left(\frac{2\xi_0\gamma}{\hat{\xi}_n}-1\right)\left(1-\frac{\hat{\xi}_n}{\xi}\right)-1\right].
\end{equation}
\end{small}
Since $\hat{\xi}_n\asconv\xi_0$, we have $\log\frac{\xi_0}{\hat{\xi}_n}>-1$, $\frac{2\xi_0\gamma}{\hat{\xi}_n}-1>\frac{1}{10}$ and 
\begin{equation*}
    \left(\frac{2\xi_0\gamma}{\hat{\xi}_n}-1\right)\left(1-\frac{\hat{\xi}_n}{\xi}\right)\leq \frac{1}{10}\left(1-\frac{\xi_0}{2\xi}\right)
\end{equation*}
for sufficiently large $n$. Therefore, the right-hand side of \eqref{eqn:G_n_approx2} can be further bounded by
\begin{small}
\begin{equation}
\begin{split}
    PL_n(\xi)-PL_n(\hat{\xi}_n)&\leq -n\left[\frac{1}{2}\log\Gamma\left(\frac{\xi_0}{\xi}\right)-\frac{1}{2}\log\frac{\xi_0}{\xi}-\frac{1}{10}\left(1-\frac{\xi_0}{2\xi}\right)-2\right]\\
    &=-n\left[\frac{1}{2}\log\Gamma\left(\frac{\xi_0}{\xi}\right)-\frac{1}{2}\log\frac{\xi_0}{\xi}+\frac{\xi_0}{20\xi}-\frac{21}{10}\right].
\end{split}
\end{equation}
\end{small}
Similarly, we know from the definition of $c_0$ in \eqref{K_compact} that the right-hand side of the above inequality is negative when $\xi<c_0\hat{\xi}_n$.
\end{proof}
% in which
% \begin{equation*}
%     \begin{split}
%         A_n^1&=-n\log\left(\frac{\xi}{\hat{\xi}_n}\right)+\left(\frac{1}{\hat{\xi}_n}-\frac{1}{\xi}\right)\sumN\log(Y_i-\hat{\beta}_n)+n\log\frac{\sumN (Y_i-\hat{\beta}_n)^{-1/\hat{\xi}_n}}{\sumN (Y_i-\hat{\beta}_n)^{-1/\xi}},\\
%         A_n^2&=\frac{\xi+1}{\xi}\sumN\log\frac{Y_i-\hat{\beta}_n}{Y_i-\beta_n(\xi)}+n\log\frac{\sumN (Y_i-\hat{\beta}_n)^{-1/\xi}}{\sumN (Y_i-\beta_n(\xi))^{-1/\xi}}.
%     \end{split}
% \end{equation*}

% \vskip 0.1cm
% Next we examine $A_n^1$ and $A_n^2$ separately.

% First look at $A_n^1$. We can easily verify that 
% \begin{equation*}
%     \begin{split}
%         A_n^1&=-n\log\left(\frac{\xi}{\hat{\xi}_n}\right)+\left(\frac{1}{\hat{\xi}_n}-\frac{1}{\xi}\right)\sumN\log w_i(\hat{\btheta}_n)+n\log\left(\sumN w_i(\hat{\btheta}_n)^{-\frac{1}{\hat{\xi}_n}}\Big/\sumN w_i(\hat{\btheta}_n)^{-\frac{1}{\xi}}\right)\\
%         &=F_n(\hat{\xi}_n)-F_n(\xi),
%     \end{split}
% \end{equation*}
% where $F_n$ is defined in Lemma \ref{lem:monotone_on_xi}. 

% Then look at $A_n^2$. When $2\xi_0<\xi<n-1$, we know by the strong consistency of $\hat{\xi}_n$ that $\hat{\xi}_n<\xi$ for large $n$. By Lemma \ref{lem:beta_xi_mono}, $\hat{\beta}_n<\beta_n(\xi)$, and 
% \begin{equation*}
%   \sumN (Y_i-\hat{\beta}_n)^{-1/\xi}<\sumN (Y_i-\beta_n(\xi))^{-1/\xi}.
% \end{equation*}
% Therefore, 
% \begin{equation*}
%     A_n^2\leq \frac{\xi+1}{\xi}\sumN\log\frac{Y_i-\hat{\beta}_n}{Y_i-\beta_n(\xi)}.
% \end{equation*}

\subsection{Proof of relation (\ref{B_cond})}
\begin{proof}
We need to show that the maximizer $(\mu_n(\xi),\tau_n(\xi))$ for $\xi\in[c_0\xi_0,C_0\xi_0]$ is enfolded by the bounds in \eqref{K_compact}. For positive $\xi$, \eqref{max_cross_section_cond} can be expressed as,
\begin{equation}\label{tau_n_xi}
   \tau_n(\xi)=\xi\left[\frac{1}{n}\sumN (Y_i-\beta_n(\xi))^{-\frac{1}{\xi}}\right]^{-\xi}.
\end{equation}
It follows immediately that
\begin{equation*}
    \tau_n(\xi)>\xi\left[(Y_{(1)}-\beta_n(\xi))^{-\frac{1}{\xi}}\right]^{-\xi}=\xi(Y_{(1)}-\beta_n(\xi)),
\end{equation*}
which results in
\begin{equation*}
    \mu_n(\xi)=\beta_n(\xi)+\frac{\tau_n(\xi)}{\xi}>Y_{(1)}.
\end{equation*}
Since $Y_{(1)}\asconv\beta_0$, we have $\mu_n(\xi)>\beta_0-1$, which is the lower bound in \eqref{K_compact}. In the following, we validate the upper bound for $\mu_n(\xi)$ and the bounds for $\tau_n(\xi)$ via splitting $[c_0\xi_0,C_0\xi_0]$ into $[\hat{\xi}_n,C_0\xi_0]$ and $[c_0\xi_0,\hat{\xi}_n]$.

\textbf{A. Case $\xi\in [\hat{\xi}_n,C_0\xi_0]$}. Lemma \ref{lem:beta_xi_mono} implies $\hat{\beta}_n\leq \beta_n(\xi)<Y_{(1)}$, and
\begin{equation*}
    \frac{1}{n}\sumN (Y_i-\beta_n(\xi))^{-\frac{1}{\xi}}\geq \frac{1}{n}\sumN (Y_i-\hat{\beta}_n)^{-\frac{1}{\xi}},
\end{equation*}
which results in
\begin{equation*}
    \tau_n(\xi)\leq \xi\left[\frac{1}{n}\sumN (Y_i-\hat{\beta}_n)^{-\frac{1}{\xi}}\right]^{-\xi}=\frac{\hat{\tau}_n\xi}{\hat{\xi}_n}\left[\frac{1}{n}\sumN w^{-1/\xi}_i(\hat{\btheta}_n)\right]^{-\xi}.
\end{equation*}
Since $\xi\in [\hat{\xi}_n,C_0\xi_0]$, $\frac{1}{\xi}$ can be enclosed in $[\frac{1}{C_0\xi_0}, \frac{2}{\xi_0}]$ for large $n$. By the uniform consistency in Proposition \ref{prop:uniform_consistency}, there almost surely exists $N>0$ such that for all $n>N$,
\begin{equation*}
    \frac{1}{n}\sumN w^{-\frac{1}{\xi}}_i(\hat{\btheta}_n)>\frac{1}{2}\Gamma\left(\frac{\xi_0}{\xi}+1\right)>0.4\text{  for }\xi\in [\hat{\xi}_n,C_0\xi_0],
\end{equation*}
where we utilize the fact that $\Gamma(x)>0.8$ for all $x>0$. Further, we have
\begin{equation*}
    \tau_n(\xi)<\frac{\hat{\tau}_n\xi}{\hat{\xi}_n}0.4^{-\xi}.
\end{equation*}
By the strong consistency of $\hat{\btheta}_n$, $\frac{\hat{\tau}_n\xi}{\hat{\xi}_n}<\frac{2.5\tau_0\xi}{\xi_0}\leq 2.5\tau_0C_0$ and $Y_{(1)}<\beta_0+1$ for large $n$. Thus,
\begin{equation}\label{tau_n_upper}
    \tau_n(\xi)<\tau_0C_02.5^{C_0\xi_0+1},\text{ and }\mu_n(\xi)=\beta_n(\xi)+\frac{\tau_n(\xi)}{\xi}<\beta_0+1+\frac{\tau_0C_0}{\xi_0}2.5^{C_0\xi_0+1},
\end{equation}
where we have used the relation $\beta_n(\xi)<Y_{(1)}$ when $0<\xi<n-1$.

\noindent\textbf{B. Case $\xi\in [c_0\xi_0,\hat{\xi}_n]$}. Denote $\kappa_0=\min\{\frac{1}{2},\frac{1}{2\xi_0}\}$, and we have $1-\kappa_0\xi_0>0$. We deduce by the power means inequality that 
\begin{equation}\label{for_mu}
    \frac{\tau_n(\xi)}{\xi}=\left[\frac{1}{n}\sumN (Y_i-\beta_n(\xi))^{-\frac{1}{\xi}}\right]^{-\xi}<\left[\frac{1}{n}\sumN (Y_i-\beta_n(\xi))^{\kappa_0}\right]^{\frac{1}{\kappa_0}}.
\end{equation}
For $\xi<\hat{\xi}_n$, $\beta_n(\xi)<\hat{\beta}_n$. Apply the classic $c_r$-inequality \eqref{Cr_ineq} twice to get
\begin{equation*}
    \frac{1}{n}\sumN(Y_i-\beta_n(\xi))^{\kappa_0}\leq \frac{1}{n}\sumN(Y_i-\hat{\beta}_n)^{\kappa_0}+(\hat{\beta}_n-\beta_n(\xi))^{\kappa_0}
\end{equation*}
and 
\begin{small}
\begin{equation*}
     \left[\frac{1}{n}\sumN(Y_i-\hat{\beta}_n)^{\kappa_0}+(\hat{\beta}_n-\beta_n(\xi))^{\kappa_0}\right]^{\frac{1}{\kappa_0}} \leq 2^{\frac{1}{\kappa_0}-1}\left[\frac{1}{n}\sumN(Y_i-\hat{\beta}_n)^{\kappa_0}\right]^{\frac{1}{\kappa_0}}+2^{\frac{1}{\kappa_0}-1}(\hat{\beta}_n-\beta_n(\xi)),
\end{equation*}
\end{small}
where $2^{\frac{1}{\kappa_0}-1}>1$. Plugging the results back in \eqref{for_mu}, we obtain
\begin{equation*}
    \frac{\tau_n(\xi)}{\xi}\leq 2^{\frac{1}{\kappa_0}-1}\left[\frac{1}{n}\sumN(Y_i-\hat{\beta}_n)^{\kappa_0}\right]^{\frac{1}{\kappa_0}}+2^{\frac{1}{\kappa_0}-1}(\hat{\beta}_n-\beta_n(\xi)).
\end{equation*}
In the supplementary material, we show that for $\xi\in[c_0\xi_0,\hat{\xi}_n]$,
\begin{equation*}
    \hat{\beta}_n-\beta_n(\xi)<\frac{4\xi_0^{\frac{1}{\xi_0}-2}\kappa_1^{-\frac{1}{c_0\xi_0}}}{\tau_0c_0^3},
\end{equation*}
where $\kappa_1=\min\{1,c_0\xi_0\}$. In the meantime, for $n$ sufficiently large,
\begin{equation*}
    2^{\frac{1}{\kappa_0}-1}\left[\frac{1}{n}\sumN(Y_i-\hat{\beta}_n)^{\kappa_0}\right]^{\frac{1}{\kappa_0}}=\frac{2^{\frac{1}{\kappa_0}-1}\hat{\tau}_n}{\hat{\xi}_n}\left[\frac{1}{n}\sumN w_i^{\kappa_0}(\hat{\btheta}_n)\right]^{\frac{1}{\kappa_0}}\leq \frac{2^{\frac{1}{\kappa_0}}\tau_0}{\xi_0}\Gamma^{\frac{1}{\kappa_0}}(1-\kappa_0\xi_0).
\end{equation*}

Consequently,
\begin{equation*}
\begin{split}
    \tau_n(\xi)&< 2^{\frac{1}{\kappa_0}}\tau_0\Gamma^{\frac{1}{\kappa_0}}(1-\kappa_0\xi_0)+\frac{2^{\frac{1}{\kappa_0}+1}\xi_0^{\frac{1}{\xi_0}-1}\kappa_1^{-\frac{1}{c_0\xi_0}}}{\tau_0c_0^3},\\
    \mu_n(\xi)=\beta_n(\xi)+\frac{\tau_n(\xi)}{\xi}&<2\beta_0+\frac{2^{\frac{1}{\kappa_0}}\tau_0}{\xi_0}\Gamma^{\frac{1}{\kappa_0}}(1-\kappa_0\xi_0)+\frac{2^{\frac{1}{\kappa_0}+1}\xi_0^{\frac{1}{\xi_0}-2}\kappa_1^{-\frac{1}{c_0\xi_0}}}{\tau_0c_0^3}.
\end{split}
\end{equation*}

\noindent Combining the previous two cases, we complete the proof of relation \eqref{B_cond}.
\end{proof}

\end{appendix}

%%%%%%%%%%%%%%%%%%%%%%%%%%%%%%%%%%%%%%%%%%%%%%
%% Support information (funding), if any,   %%
%% should be provided in the                %%
%% Acknowledgements section.                %%
%%%%%%%%%%%%%%%%%%%%%%%%%%%%%%%%%%%%%%%%%%%%%%
% \section*{Acknowledgements}
% The authors would like to thank the anonymous referees, an Associate
% Editor and the Editor for their constructive comments that improved the
% quality of this paper.

% The first author was supported by NSF Grant DMS-??-??????.

% The second author was supported in part by NIH Grant ???????????.

%%%%%%%%%%%%%%%%%%%%%%%%%%%%%%%%%%%%%%%%%%%%%%
%% Supplementary Material, if any, should   %%
%% be provided in {supplement} environment  %%
%% with title and short description.        %%
%%%%%%%%%%%%%%%%%%%%%%%%%%%%%%%%%%%%%%%%%%%%%%
\begin{supplement}
% \BASadd{Remember to replace the title here with whatever title we settle on}
\textbf{Supplement to ``Uniqueness and global optimality of the maximum likelihood estimator for the generalized extreme value distribution"}.
The supplementary material contains  additional technical results and proofs that complete the proofs in the appendices and extend the proof of relation \eqref{B_cond} to the negative $\xi_0$.
\end{supplement}
% \BASadd{Is there a section to acknowledge funding? NSF DMS-2001433}\LZadd{Yes, it is in the footnote of the title on the front page.}

%%  MR numbers will be added by VTeX.                      %%
%%                                                         %%
%%  Use \cite{...} to cite references in text.             %%
%%                                                         %%
%%%%%%%%%%%%%%%%%%%%%%%%%%%%%%%%%%%%%%%%%%%%%%%%%%%%%%%%%%%%%

%% if your bibliography is in bibtex format, uncomment commands:
%\bibliographystyle{imsart-nameyear.bst} % Style BST file (imsart-number.bst or imsart-nameyear.bst)
%\bibliography{bibliography}       % Bibliography file (usually '*.bib')

%% or include bibliography directly:

\end{document}